\documentclass[11pt, oneside]{article}
\usepackage{amsmath, enumerate, amsthm, amssymb, wasysym, verbatim, bbm, color, graphics, graphicx, mathrsfs}
\usepackage[margin=.75in]{geometry}
\usepackage{algpseudocode,algorithm}
\usepackage{lipsum}

\makeatletter
\newcommand\notsotiny{\@setfontsize\notsotiny{7}{7.5}}
\makeatother

\usepackage{listings}
\usepackage{tikz}
\usepackage{mathdots,mathtools,enumitem}
\usepackage{yhmath}
\usepackage{cancel}
\usepackage{color}
\usepackage{siunitx}
\usepackage{array}
\usepackage{multirow}
\usepackage{subcaption}
\usepackage{enumitem}

\usepackage{gensymb}
\usepackage{tabularx}
\usepackage{extarrows}
\usepackage{booktabs}
\usetikzlibrary{fadings}
\usetikzlibrary{patterns}
\usetikzlibrary{shadows.blur}
\usetikzlibrary{shapes}

\tikzset{every picture/.style={line width=0.75pt}} 

\usepackage{float}
\usepackage{extarrows}
\usepackage{framed,bbm}
\usepackage[citecolor=blue,colorlinks=true,linkcolor=black]{hyperref}

\newtheorem{thm}{Theorem}[section]

\newtheorem{lem}[thm]{Lemma}
\newtheorem{cor}[thm]{Corollary}

\newtheorem{prop}[thm]{Proposition}

\theoremstyle{remark}

\newtheorem{remark}{Remark}[section]

\newcommand{\mi}{\mathrm{i}}
\newcommand{\md}{\mathrm{d}}
\newcommand{\me}{\mathrm{e}}
\newcommand{\eqd}{\overset{d}{=}}
\newcommand{\N}{\mathbb{N}}
\newcommand{\E}{\mathbb{E}}
\newcommand{\p}{\mathbb{P}}
\newcommand{\R}{\mathbb{R}}
\newcommand{\RP}{\mathbb{R}_+}

\newcommand{\ve}{\varepsilon}
\newcommand{\sign}{\mathrm{sgn}}
\newcommand{\Oh}{\mathcal{O}}

\newcommand{\Unif}{\mathrm{U}}
\newcommand{\Normal}{\mathrm{N}}
\newcommand{\rmS}{\mathrm{S}}

\newcommand{\mF}{\mathcal{F}}
\newcommand{\mG}{\mathcal{G}}
\newcommand{\mT}{\mathcal{T}}
\newcommand{\Exp}{\mathrm{Exp}}
\newcommand{\Gam}{\mathrm{Gamma}}
\newcommand{\erf}{\mathrm{erf}}
\newcommand{\les}{\leqslant}
\newcommand{\ges}{\geqslant}
\newcommand{\wt}[1]{\widetilde{#1}}
\numberwithin{equation}{section}
\numberwithin{figure}{section}

\newcommand{\1}{\mathbbm{1}}

\newcommand{\ov}{\overline}
\newcommand{\indep}{\perp \!\!\! \perp}
\DeclareMathOperator*{\esssup}{ess\,sup}

\begin{document}

\title{Fast exact simulation of the first passage of a tempered stable subordinator across a non-increasing function}

\author{
Jorge Ignacio Gonz\'alez C\'azares\footnote{Universidad Nacional Autónoma de México, IIMAS, \texttt{jorge.gonzalez@sigma.iimas.unam.mx}}\and
Feng Lin\footnote{University of Warwick, \texttt{feng.lin.1996@outlook.com}}\and
Aleksandar Mijatovi\'c\footnote{University of Warwick \& The Alan Turing Institute, \texttt{a.mijatovic@warwick.ac.uk}}}

\maketitle
\begin{abstract}
We construct a fast exact algorithm for the simulation of the first-passage time, jointly with the undershoot and overshoot, of a tempered stable subordinator over an arbitrary non-increasing absolutely continuous function. We prove that the running time of our algorithm has finite exponential moments and provide bounds
on its expected running time with explicit 
dependence on the characteristics of the process and the initial value of the function. 
The expected running time grows at most cubically in the stability parameter (as it approaches either $0$ or $1$) and is linear in the tempering parameter and the initial value of the function. 
Numerical performance, based on the implementation in the dedicated GitHub repository, exhibits a good agreement with our theoretical bounds.
We provide numerical examples to illustrate the performance of our algorithm in Monte Carlo estimation.
\end{abstract}



\tableofcontents

\section{Introduction}
\label{section:intro}

The study of the first-passage event across a barrier is a classical subject that has long been of interest for many stochastic processes, including L\'evy processes. For instance, first-passage events describe the ruin event (i.e. the ruin time and penalty function) in risk theory~\cite[Ch.~12]{doi:10.1142/7431}, are used to describe the law of the steady-state waiting time and workload in queuing theory~\cite{MR1978607} and arise in the probabilistic representation of the solutions to fractional partial differential equations (FPDEs) (see e.g.~\cite{hernandez2017generalised} and the references therein), as well as in financial mathematics (in payoffs of  barrier-type derivative securities). In all such areas, simulation algorithms and Monte Carlo methods are of great interest and widely used~(see~\cite[Ex.~5.15]{MR2331321}). However, direct biased Monte Carlo methods, based on random walk approximations of the crossing time, may lead to errors that are often difficult and computationally expensive to control, see e.g. Subsection~\ref{subsec:naiveMC} below. Fast exact simulation of the crossing time of a level by a subordinator is thus needed for a stable Monte Carlo algorithm in these contexts. For example, in~\cite{MR4219829}, exact simulation of the first-passage time of a stable subordinator was used to solve numerically an FPDE with a Caputo fractional derivative~\cite{MR2854867}.  Generalising~\cite{MR2854867} beyond Caputo fractional derivatives requires an exact simulation algorithm for more general subordinators, along with an \textit{a priori} bound on its expected running time. 
The exact simulation algorithm in~\cite{chi2016exact}, applicable to a broad class of subordinators, typically turns out to  have an infinite expected running time (see Subsection~\ref{subsec:Chi} below), making it hard to exploit the general algorithm in~\cite{chi2016exact} in this context.

The present paper develops an exact simulation algorithm for the first-passage event of a tempered stable subordinator across a non-increasing function, such that its random running time has finite exponential moments with explicit (in terms of model parameters) control over the expected running time. While our algorithm is widely applicable to problems in applied probability discussed above
when the underlying model is a finite variation 
tempered stable process (cf.~\cite{MR3500619} for examples in financial mathematics),
our main interest in this algorithm stems from our aim to generalise~\cite{MR4219829} to tempered fractional derivatives and possibly, in future work, to more general non-local operators using  ideas from~\cite{chi2016exact,MR4122822}. 

\subsection{The complexity of the main exact simulation algorithm}
\label{subsec:intro-complexity}

A subordinator $S=(S_t)_{t\in\RP}$ (where $\RP\coloneqq[0,\infty)$) is a non-decreasing process with independent and stationary increments and right-continuous paths, started at zero $S_0=0$. It is well known that $S_t=\mu t+\sum_{s\in(0,t]}J_s$, where $\mu\in\RP$ is  a non-negative drift parameter and the (positive) jumps $\{(s,J_s):s>0, J_s>0\}$ form a Poisson point process with mean measure $\md t\otimes \nu(\md x)$ on $(0,\infty)\times (0,\infty)$, where $\nu$ is the \emph{L\'evy measure} of $S$, satisfying $\int_0^\infty\min\{1,x\}\nu(\md x)<\infty$ (see 
 details in~\cite[Def.~8.2]{ken1999levy}). 
In this work we consider a \emph{driftless tempered stable subordinator}~$S$: $\mu=0$ and $\nu(\md x)\coloneqq\theta(\alpha/\Gamma(1-\alpha))\me^{-qx}x^{-\alpha-1}\md x$ for some (fixed) $\theta>0$, $\alpha\in(0,1)$ and $q\in[0,\infty)$. The first-passage event of $S$ is described by the random element 
\[
\big(\tau_b,
    S_{\tau_b-},
    S_{\tau_b}\big),
\qquad
\text{where}\quad 
\tau_b\coloneqq\inf\{t>0\,:\,S_t>b(t)\}
\]
is the first crossing time of $S$ over a non-increasing absolutely continuous function $b:[0,\infty)\to[0,\infty)$, with $b(0)\in(0,\infty)$. Since $S_0=0$, we have $0<\tau_b<\infty$ almost surely and the variable $S_{\tau_b-}\coloneqq \lim_{t\uparrow\tau_b}S_t$, 
referred to as the \textit{undershoot} in this paper, 
gives the position of the subordinator $S$ just before it crosses the function $b$.
The position of the subordinator $S$ at first crossing time $\tau_b$ over $b$ is given by $S_{\tau_b}$. 
\hyperref[alg:improved_triple_temper_stable]{TSFFP-Alg} below, which is the main simulation algorithm of this paper, samples from the law of the vector $(\tau_b,S_{\tau_b-},S_{\tau_b})$.

Denote by $\mT$ the (random) running time  of our main algorithm \hyperref[alg:improved_triple_temper_stable]{TSFFP-Alg}. Here and throughout we  assume that elementary operations such as addition, multiplication  and evaluation of elementary functions $\sin$, $\cos$, $\exp$ all have constant computational cost. Further, we assume all operations use floating point precision and denote by $N\in\N$ the desired number of precision bits of the output. Then $\mT$ has finite exponential moments (up to some order) and its mean is bounded above by
\[ 
\kappa_{\hyperref[alg:improved_triple_temper_stable]{TSFFP}}
(1+qb(0)/\alpha)((1-\alpha)^{-3}+|\log\alpha|+\log N)/\alpha,
\]
where the constant $\kappa_{\hyperref[alg:improved_triple_temper_stable]{TSFFP}}$ does not depend on  $\alpha\in(0,1)$, $q\in[0,\infty)$ or $b$ (see Corollary~\ref{cor:improved_tempered_expected_time} below). Moreover, $\kappa_{\hyperref[alg:improved_triple_temper_stable]{TSFFP}}$  can be made explicit in terms of the costs of elementary operations listed above. 

Since the law of the vector $(\tau_b,S_{\tau_b-},S_{\tau_b})$ degenerates in any of the limiting regimes, $\alpha\uparrow1$, $\alpha\downarrow0$, $q\to\infty$ or $b(0)\to\infty$, 
the deterioration in the performance of the algorithm in these regimes is expected.
 Indeed, the subordinator $S$ converges weakly to a linear drift as $\alpha\uparrow1$ (see the stable characteristic function in~\eqref{eq:char_fun_stable} below), making the law of the undershoot $S_{\tau_b-}$ increasingly harder to simulate from. If instead $\alpha\downarrow0$ or $q\to\infty$, then the subordinator $S$ converges weakly to the trivial process $0$ for which the first-passage time $\tau_b$ is infinite. Finally, if $b\equiv b(0)$ increases, then the crossing time $\tau_b$ increases, extending the running time of several loops in the algorithm. The main point of the formula in the display above
 is that it provides an explicit insight into the deterioration  of the performance of the algorithm in any of the limiting regimes.
 There is a good match between this theoretical upper bound on the performance and what we observe through the implementation found in the repository~\cite{repository}, see Subsection~\ref{sec:numerics} below for details. 

The key step in the exact simulation 
of the first-passage triplet $(\tau_b,S_{\tau_b-},S_{\tau_b})$
for a general tempering parameter $q\in[0,\infty)$ (\hyperref[alg:improved_triple_temper_stable]{TSFFP-Alg} below)
 is the stable case $q=0$, given in~\hyperref[alg:triple_stable_conditional_on_time]{SFP-Alg} below. 
The general case $q\in(0,\infty)$  in \hyperref[alg:improved_triple_temper_stable]{TSFFP-Alg} requires only an additional accept-reject step based on the Esscher transform described below. In turn, the simulation of the first-passage event of a stable process in~\hyperref[alg:triple_stable_conditional_on_time]{SFP-Alg} is essentially reduced to the simulation of the undershoot $S_{\tau_b-}$ on the event 
$\{S_{\tau_b-}<b(\tau_b)\}$
when the first passage did not occur continuously (i.e. by \textit{creeping}\footnote{Note that creeping occurs with positive probability 
$\p[S_{\tau_b-}=b(\tau_b)]>0$
 when the boundary function $b$ is not constant, even though the tempered stable subordinator has no drift, see~\cite{https://doi.org/10.48550/arxiv.2205.06865,chi2016exact}.}, which occurs on the event $\{S_{\tau_b-}=b(\tau_b)\}$, cf. Subsection~\ref{subsec:joint_law_description} below). Sampling the undershoot efficiently, which turns out to be deceivingly hard (see Subsection~\ref{subsec:How_not_to_sample} below), is the main technical contribution of this paper.  
Although the density of the undershoot $S_{\tau_b-}$, conditional on $\{\tau_b=t\}$, admits a semi-analytic explicit form (see~\eqref{eq:undershoot_law} below), it is unbounded and does not lend itself well to direct accept-reject methods. Our main strategy to circumvent this issue is to break up the state space $(0,b(t))$ of $S_{\tau_b-}|\{\tau_b=t\}$ into subintervals, depending on the values of the (possibly random) parameters, and on each subinterval use an appropriate   accept-reject algorithm. 
The algorithm in~\cite{devroye2012note}, see \hyperref[alg:Devroye]{LC-Alg} in Appendix~\ref{sec:devroye} below, plays a key role precisely in the subinterval where the density of $S_{\tau_b-}$ decreases sharply from a very high value to super-exponentially small values, making the standard accept-reject bounds hard to find.

Most algorithms developed and used in this paper are based on a combination of two components: (I) a numerical search step in the form of binary search (also known as the bisection method), possibly followed by an application of the Newton--Raphson method (rigorous care is taken to verify the assumptions that guarantee quadratic convergence of the Newton--Raphson algorithm) and (II) rejection (or accept-reject) sampling. For both components, we do a careful complexity analysis accounting for all computational operations. This is crucial because parameter values themselves are often random, requiring control over their distribution as well as the impact on the computational complexity when these random parameters take extreme values. Note that the infinite expected running time of the algorithm in~\cite{chi2016exact} arises essentially due to an accept-reject step having a random acceptance probability with large mass close to zero, see Subsection~\ref{subsec:Chi} below for details.

\subsection{Comparison with the literature}

Exact simulation of the first-passage event of a tempered stable subordinator is a central topic of two articles:~\cite{chi2016exact} and~\cite{MR4122822}. \cite{chi2016exact} was the first to develop a simulation algorithm for the first-passage event for a wide class of subordinators  over a general non-decreasing absolutely continuous boundary functions~$b$. The case of the tempered stable subordinator plays a prominent role in~\cite{chi2016exact}, because the entire general class
of subordinators considered in~\cite{chi2016exact}
can be viewed as a modification of the tempered stable case. Indeed, the main algorithm in~\cite{chi2016exact} relies on this structure. As discussed above, the algorithm for the tempered stable first-passage event 
in~\cite{chi2016exact}  typically has infinite expected running time (see Subsection~\ref{subsec:Chi}), making it hard to use in Monte Carlo estimation. 

\cite{MR4122822} consider a tempered stable subordinator and a constant  function $b$. They 
develop an exact simulation algorithm~\cite[Alg.~4.1]{MR4122822}
for the pair $(\tau_b,S_{\tau_b}-b(\tau_b))$. In this case, the problem is first reduced to the stable case via the Esscher transform. Then the constant barrier makes the density of $(\tau_b,S_{\tau_b}-b)$ explicit with a stable scaling property. The scaling property in turn allows for a change of variables (in the form of the variable $r=y^{-\alpha/(1-\alpha)}t^{1/(1-\alpha)}$ in~\cite[l.11, \S4.2]{MR4122822}) that makes it easy to sample from the resulting density. This technique,
which depends on the explicit density of the random vector $(\tau_b,S_{\tau_b}-b)$,
does not generalise easily to a non-constant function~$b$, including a linear function $b(t)= a_0-a_1t$ with $a_1>0$. Part of the difficulty in generalising such methodology is that a new possibility arises: the subordinator may \emph{creep} (i.e. cross continuously~\cite{https://doi.org/10.48550/arxiv.2205.06865,chi2016exact}) through the curve $b$, in which case $S_{\tau_b-}=b(\tau_b)=S_{\tau_b}$. Moreover, the probability of creeping is not uniform in time even for linear functions (see Proposition~\ref{prop:joint_law}(b) below), making it hard to compute even the mass of the absolutely continuous component of the joint law.

\subsection{Organisation of the paper}

The main algorithms and the results on their respective computational complexities are provided in Section~\ref{sec:main_simulation}. In Section~\ref{sec:applications} we present numerical evidence for the theoretical bounds on the computational complexities of our algorithms. We also present two numerical applications of our simulation algorithms, using Monte Carlo estimation, to solve fractional partial differential equations and to price barrier options. 
The GitHub repository~\cite{repository} contains the implementation (in Julia) of the algorithms in this paper, used in Section~\ref{sec:applications}.
Finally, the proofs of all our results on the validity and computational complexities of our algorithms are given in Section~\ref{sec:proofs}. 
See~\cite{Presentation_Jorge} for a short \href{https://youtu.be/dO-cQeABHdM}{YouTube} presentation 
on \href{https://www.youtube.com/@prob-am7844}{Prob-AM} channel
discussing the algorithm. 

\section{Sampling the first passage of a tempered stable subordinator}
\label{sec:main_simulation}

For any stability index $\alpha\in(0,1)$, intensity $\theta\in(0,\infty)$ and tempering parameter $q\in\RP$, let $S=(S_t)_{t\in\RP}$ be, under $\p_q$ (with expectation operator $\E_q$), a driftless tempered stable subordinator started at $0$ with Laplace transform and L\'evy measure
\begin{equation}
\label{eq:char_fun_stable}
\E_q[\me^{-uS_1}]
=\exp\big((q^\alpha-(u+q)^\alpha)\theta\big),
    \enskip u\ges0
\quad \&\quad \nu_q(\md x)
=(\alpha/\Gamma(1-\alpha))\theta\me^{-qx}x^{-\alpha-1}\md x,
    \enskip x>0.
\end{equation}
Note that, under $\p\coloneqq\p_0$, $S$ is a stable subordinator with L\'evy measure $\nu\coloneqq\nu_0$. 

\subsection{Dependence of the algorithms for the simulation of \texorpdfstring{$(\tau_b,
    S_{\tau_b-},
    S_{\tau_b})$ under $\p_q$}{First-passage triplet}}
\begin{figure}[ht]
\caption{Dependence of algorithms ($B\leftarrow A$ means algorithm $A$ calls algorithm $B$)\label{fig:alg_dependence}}\centering
\begin{tikzpicture}[x=0.73pt,y=0.75pt,yscale=-1,xscale=1]

\draw [->] (411,20.8) -- (195.6,20.8) ;
\draw  [->] (30,50) .. controls (29.01,79.7) and (29.49,80) .. (78.11,80.38) ;
\draw  [->] (30,50) .. controls (29.51,129.2) and (29.5,129.99) .. (78.76,130.45) ;
\draw  [->] (330.96,210.93) .. controls (330.96,175.65) and (330.29,172.07) .. (358.5,170.63) ;
\draw  [->] (330.96,210.93) .. controls (330.28,103.13) and (329.63,100.98) .. (359.19,100.31) ;
\draw  [->] (160.82,149.47) .. controls (159.53,179.72) and (159.98,180.48) .. (135.55,180.4) ;
\draw  [->] (160.82,149.47) .. controls (160.02,243.1) and (160,240.74) .. (178.25,240.06) ;
\draw  [->] (460.67,99.67) -- (518.27,99.92) ;
\draw  [->] (461.29,169.6) .. controls (509.12,169.93) and (473,100.35) .. (518.85,99.92) ;
\draw  [->] (460.67,169.67) -- (519.27,169.92) ;
\draw  [->] (30,50) -- (29.51,158) ;

\draw    (12,12) -- (195,12) -- (195,50) -- (12,50) -- cycle  ;
\draw (15,16) node [anchor=north west][inner sep=0.75pt]  [font=\small] [align=left] {\textbf{\hyperref[alg:triple_temper_stable]{TSFP-Alg}}\\{\notsotiny Sampling $\displaystyle ( \tau _{c} ,S_{\tau _{c} -} ,S_{\tau _{c}})$ under $\displaystyle \p_{q}$}};
\draw    (411,12) -- (616,12) -- (616,50) -- (411,50) -- cycle  ;
\draw (414,16) node [anchor=north west][inner sep=0.75pt]  [font=\small] [align=left] {\hyperref[alg:improved_triple_temper_stable]{\textbf{TSFFP-Alg}}\\{\notsotiny Fast sampling of $\displaystyle ( \tau _{b} ,S_{\tau _{b} -} ,S_{\tau _{b}})$ under $\displaystyle \mathbb{P}_{q}$}};
\draw (210.67,24.67) node [anchor=north west][inner sep=0.75pt]  [font=\scriptsize] [align=left] {Repeated application of \hyperref[alg:triple_temper_stable]{TSFP-Alg} to  \\capped boundary $c(t)=\min\{b(t) ,R\}$};
\draw    (80,59.67) -- (206,59.67) -- (206,97.67) -- (80,97.67) -- cycle  ;
\draw (83,63.67) node [anchor=north west][inner sep=0.75pt]  [font=\small] [align=left] {\hyperref[alg:tempered_stable]{\textbf{TS-Alg}}\\{\notsotiny Sampling $\displaystyle S_{t}$ under $\displaystyle \mathbb{P}_{q}$}};
\draw    (80.67,111) -- (258.67,111) -- (258.67,149) -- (80.67,149) -- cycle  ;
\draw (83.67,115) node [anchor=north west][inner sep=0.75pt]  [font=\small] [align=left] {\textbf{\hyperref[alg:triple_stable_conditional_on_time]{SFP-Alg}}\\{\notsotiny Sampling $\displaystyle ( \tau _{c} ,S_{\tau _{c} -} ,S_{\tau _{c}})$ under $\displaystyle \mathbb{P}$}};
\draw    (11,160.33) -- (132,160.33) -- (132,198.33) -- (11,198.33) -- cycle  ;
\draw (14,164.33) node [anchor=north west][inner sep=0.75pt]  [font=\small] [align=left] {\hyperref[alg:stable]{\textbf{S-Alg}}\\{\notsotiny Sampling $\displaystyle S_{t}$ under $\displaystyle \mathbb{P}$}};
\draw    (180.33,211) -- (384.33,211) -- (384.33,271) -- (180.33,271) -- cycle  ;
\draw (183.33,215) node [anchor=north west][inner sep=0.75pt]  [font=\small] [align=left] {\textbf{\hyperref[alg:undershoot_stable]{SU-Alg}}\\{\notsotiny Sampling $\displaystyle S_{\tau _{c} -}$ under $\displaystyle \mathbb{P}$, conditional on }\\{\notsotiny $\displaystyle \{\tau _{c} =t,c( \tau _{c}) =w,\Delta _{S}( \tau _{c})  >0\}$}};
\draw    (361,81.33) -- (461,81.33) -- (461,128.33) -- (361,128.33) -- cycle  ;
\draw (364,85.33) node [anchor=north west][inner sep=0.75pt]  [font=\small] [align=left] {\textbf{\hyperref[alg:psi_1]{$\psi^{(1)}$-Alg}}\\{\notsotiny Sampling from $\displaystyle \tilde{\psi }_{s}^{( 1)}$}};
\draw    (361,141.33) -- (461,141.33) -- (461,188.33) -- (361,188.33) -- cycle  ;
\draw (364,145.33) node [anchor=north west][inner sep=0.75pt]  [font=\small] [align=left] {\textbf{\hyperref[alg:psi_2]{$\psi^{(2)}$-Alg}}\\{\notsotiny Sampling from $\displaystyle \tilde{\psi }_{s}^{( 2)}$}};
\draw    (521,84) -- (624,84) -- (624,139) -- (521,139) -- cycle  ;
\draw (524,88) node [anchor=north west][inner sep=0.75pt]  [font=\small] [align=left] {\hyperref[alg:Devroye]{\textbf{LC-Alg}}\\{\notsotiny Sampling from a }\\{\notsotiny log-concave density}};
\draw    (521,151) -- (651,151) -- (651,189) -- (521,189) -- cycle  ;
\draw (524,155) node [anchor=north west][inner sep=0.75pt]  [font=\small] [align=left] {\hyperref[alg:inversion_newton_raphson]{\textbf{NR-Alg}}\\{\notsotiny Newton-Raphson method}};
\end{tikzpicture}
\end{figure}
Figure~\ref{fig:alg_dependence} above summarises the dependence between the algorithms used to simulate
the triplet $(\tau_b,S_{\tau_b-},S_{\tau_b})$ for the tempered stable subordinator $S$. Algorithm
\hyperref[alg:improved_triple_temper_stable]{TSFFP-Alg} essentially reduces the problem to the same problem in the stable case, dealt with by~\hyperref[alg:triple_stable_conditional_on_time]{SFP-Alg}.
The hardest step of the stable first-passage algorithm \hyperref[alg:triple_stable_conditional_on_time]{SFP-Alg}
consists of the simulation of the undershoot $S_{\tau_b-}$, performed by \hyperref[alg:undershoot_stable]{SU-Alg}, which in turn relies on \hyperref[alg:psi_1]{$\psi^{(1)}$-Alg} and \hyperref[alg:psi_2]{$\psi^{(2)}$-Alg}.
These two rejection-sampling algorithms are tailored to the specific densities 
$\tilde \psi_s^{(1)}$ and $\tilde \psi_s^{(2)}$
that arise in our problem (see Proposition~\ref{prop:undershoot} for the definition). Upon noting that a certain function (related to $\tilde \psi_s^{(1)}$) is log-concave, \hyperref[alg:psi_1]{$\psi^{(1)}$-Alg} essentially only calls \hyperref[alg:Devroye]{LC-Alg}, a very fast rejection-sampling procedure for log-concave densities.  
\hyperref[alg:psi_2]{$\psi^{(2)}$-Alg}  requires a partitioning of the state space, which relies on numerical inversion of elementary functions, and finely tuned rejection sampling. The details are given in the remainder of this section.

\subsection{Sampling the first passage of a tempered stable subordinator}
Before presenting our simulation algorithm \hyperref[alg:triple_temper_stable]{TSFP-Alg}, we give a brief intuitive account. The measures $\p_q$ (for $q\ges0$) are known to be equivalent to each other (see Appendix~\ref{app:temper}). In fact, the law of the trajectory $(S_t)_{t\in[0,T]}$ (for any $T>0$) under $\p_q$ is equal to the law under $\p_0$ weighted by the change-of-measure $\exp(-qS_T+q^\alpha \theta T)$. Since this change-of-measure is bounded, it can be used for rejection sampling of $(S_t)_{t\in[0,T]}$ under $\p_q$: propose a path on $[0,T]$ under $\p_0$ and accept it with probability $\me^{-qS_T}$, which is on average equal to $\exp(-q^\alpha \theta T)$, see Lemma~\ref{lem:tempered_law} in Appendix~\ref{app:temper} below for details. Thus, it is natural to expect that the simulation of the first-passage triplet $(\tau_b,S_{\tau_b-},S_{\tau_b})$ under $\p_q$ can essentially be reduced to the case $q=0$ on the stochastic interval $[0,\tau_b]$. However, since $\tau_b$ may have unbounded support, the change-of-measure $\exp(-qS_{\tau_b}+q^\alpha \theta \tau_b)$ may also be unbounded, making a direct rejection-sampling algorithm for $\p_q$ infeasible. \hyperref[alg:triple_temper_stable]{TSFP-Alg} makes use of the change-of-measure idea by introducing a time grid and sampling $S$ on the grid until the compact interval in which $\tau_b$ lies is identified. 

\begin{algorithm}
\caption{(TSFP-Alg) Tempered Stable First-Passage  Algorithm: samples $(\tau_b,S_{\tau_b-},S_{\tau_b})$ under~$\p_q$}
\label{alg:triple_temper_stable}
\begin{algorithmic}[1]
\Require{Parameters $\alpha\in(0,1)$, $\theta>0$, $q\ges 0$ and function $b:[0,\infty)\to[0,\infty)$}
\State{Set $t_*\gets(2qb(0)+1-2^{-\alpha})/((2^\alpha-1)q^\alpha\theta)$, $T\gets 0$, $U\gets0$ and $c(\cdot)\gets b(\cdot)$}
\Repeat\label{line:loop_until_tau<t_*}
\State{Sample $s$ from the law of $S_{t_*}$ under $\p_q$ via \hyperref[alg:tempered_stable]{TS-Alg}}
\If{$s<c(t_*)$}
\State{Update $T\gets T+t_*$, $U\gets U+s$ and $c(\cdot)\gets c(\cdot+t_*)-s$}
\EndIf
\Until{$s\ges c(t_*)$}
    \Repeat
    \State{Sample $(\tau,u,v)$ via \hyperref[alg:triple_stable_conditional_on_time]{SFP-Alg} with $\alpha,\,\theta,\,t_*,\,c$}
    \Comment{Law of $(\tau_c,S_{\tau_c-},S_{\tau_c})|\{\tau_c\les t_*\}$ under $\p$}
    \State{Sample $w$ from the law of $S_{t_*-\tau}$ under $\p$ via \hyperref[alg:stable]{S-Alg} and $E\sim\Exp(1)$} 
    \Until{$E\ges q(w+v)$}\label{line:check_temper}
    \State{\Return $(T+\tau,U+u,U+v)$}
\end{algorithmic}
\end{algorithm}

\hyperref[alg:triple_temper_stable]{TSFP-Alg} makes use of \hyperref[alg:stable]{S-Alg} and~\hyperref[alg:tempered_stable]{TS-Alg}, well-known simple and fast procedures for sampling stable and tempered stable marginals, respectively (see Appendix~\ref{app:temp_stable_marginal} below). In lines~2--7, \hyperref[alg:triple_temper_stable]{TSFP-Alg} consecutively samples $S_T$ (and stores the value in $U$) over the grid $T\in\{t_*,2t_*,\ldots\}$ (i.e., the variable $T$ keeps track of time) under the tempered stable law $\p_q$ until the interval containing $\tau_b$ is identified, i.e., the smallest $T\in t_*\N$ with $S_{T+t_*}\ges b(T+t_*)$. When the interval $(T,T+t_*]$ containing $\tau_b$ under $\p_q$ is identified, in lines~8--11, \hyperref[alg:triple_temper_stable]{TSFP-Alg} samples $(\tau_b-T,S_{\tau_b-}-S_{T},S_{\tau_b}-S_{T})$ under $\p_q$ via rejection sampling: it proposes $(\tau_b-T,S_{\tau_b-}-S_{T},S_{\tau_b}-S_{T})$ and $S_{T+t_*}-S_{\tau_b}$ under $\p$, and accepts the sample (under $\p_q$) with probability $\exp(-q(S_{T+t_*}-S_{T}))$. The parameter $t_*$ affects the following three components of \hyperref[alg:triple_temper_stable]{TSFP-Alg}: the number of iterations of the loop in lines~2--7, the complexity of \hyperref[alg:triple_stable_conditional_on_time]{SFP-Alg} called in line~9 (cf.~Theorem~\ref{thm:expectation_of_time_complexity} below) and the acceptance probability in line~11. The choice of $t_*$ in line~1 of \hyperref[alg:triple_temper_stable]{TSFP-Alg} is made to control the overall complexity of the algorithm, see details in~\eqref{eq:cost-of-TSFP} of Section~\ref{sec:proofs} below.

Stable First-Passage algorithm~\hyperref[alg:triple_stable_conditional_on_time]{SFP-Alg} and its analysis constitute the main technical contribution of this paper. The complexity of \hyperref[alg:triple_temper_stable]{TSFP-Alg}
is given in the following theorem. Proofs of all the results in this paper are in Section~\ref{sec:proofs} below.

\begin{thm}
\label{thm:tempered_expect_time}
\hyperref[alg:triple_temper_stable]{TSFP-Alg} samples from the law of the triplet $(\tau_b, S_{\tau_b-}, S_{\tau_b})$ of a tempered stable subordinator $S$ under $\p_q$. Moreover, the running time of \hyperref[alg:triple_temper_stable]{TSFP-Alg} has exponential moments with mean bounded by
\[
\kappa_{\hyperref[alg:triple_temper_stable]{TSFP}}\me^{2qb(0)/(2^\alpha-1)}((1-\alpha)^{-3}+|\log\alpha|+\log N)/\alpha,
\]
where the constant $\kappa_{\hyperref[alg:triple_temper_stable]{TSFP}}$ does not depend on the parameters $\alpha\in(0,1)$, $\theta\in(0,\infty)$ or $q\in[0,\infty)$, and $N$ specifies the required number of precision bits of the output.
\end{thm}

The upper bound on the running time of \hyperref[alg:triple_temper_stable]{TSFP-Alg} has a term of order $\Oh(\exp(2qb(0)/(2^\alpha-1)))$, which may be quite large if either $b(0)$, $q$ or $1/\alpha$ are large. However, the exponential dependence on these parameters of the complexity of \hyperref[alg:triple_temper_stable]{TSFP-Alg} can be easily turned into a linear dependence (see Corollary~\ref{cor:improved_tempered_expected_time}) by picking a constant $R>0$ and applying  \hyperref[alg:triple_temper_stable]{TSFP-Alg} successively with the boundary function $t\mapsto \min\{b(t),R\}$ until the initial boundary $b$ is crossed. Our main algorithm \hyperref[alg:improved_triple_temper_stable]{TSFFP-Alg}, which we now describe, does precisely this. Indeed, variables $T$ and $U$, introduced in line~1 of \hyperref[alg:improved_triple_temper_stable]{TSFFP-Alg} keep track of time and the value $S_T$ of $S$ at this time. Line~3 of \hyperref[alg:improved_triple_temper_stable]{TSFFP-Alg}, samples $(\tau_c-T,S_{\tau_c-}-S_{T},S_{\tau_c}-S_T)$ under $\p_q$ via \hyperref[alg:triple_temper_stable]{TSFP-Alg}, where $c:t\mapsto \min\{b(t),S_T+R\}$. The algorithm updates $T$ and $U=S_T$ in line~4. If $S_{\tau_c}\ges b(\tau_c)$ (line~5), then $\tau_b=\tau_c$, the algorithm stops and returns the triplet $(\tau_b,S_{\tau_b-},S_{\tau_b})$ in line~6; otherwise (i.e., if $S_{\tau_c}<b(\tau_c)$), it repeats from line~2. The choice of $R$ in \hyperref[alg:improved_triple_temper_stable]{TSFFP-Alg} is fixed in line~1 and is chosen to minimise an upper bound on the expected running time of the algorithm, see details in Section~\ref{sec:proofs} below.

\begin{algorithm*}
\caption{(TSFFP-Alg) Tempered Stable Fast First-Passage Algorithm: samples $(\tau_b,S_{\tau_b-},S_{\tau_b})$ under $\p_q$}
\label{alg:improved_triple_temper_stable}
\begin{algorithmic}[1]
\Require{Parameters $\alpha\in(0,1)$, $\theta>0$, $q\ges 0$ and function $b:[0,\infty)\to[0,\infty)$}
\State{Set $R\gets(2^\alpha-1)/(2q)$, $T\gets0$, $U\gets0$ and $c(\cdot)\gets b(\cdot)$}
\Repeat
\State{Sample $(\tau,u,v)$ via \hyperref[alg:triple_temper_stable]{TSFP-Alg} with function $\min\{c,R\}$}
\State{Update $T\leftarrow T+\tau$, $U\leftarrow U+v$, $c(\cdot)\leftarrow c(\cdot+\tau)-v$}
\Until{$U\ges b(T)$}
\State{\Return ${(T, U-(v-u), U)}$}
\Comment{The jump that takes $S$ over $b$ has size $v-u$}
\end{algorithmic}
\end{algorithm*}

\begin{cor}
\label{cor:improved_tempered_expected_time}
\hyperref[alg:improved_triple_temper_stable]{TSFFP-Alg} samples from the law of $(\tau_b, S_{\tau_b-}, S_{\tau_b})$ under $\p_q$. Moreover, the running time of \hyperref[alg:improved_triple_temper_stable]{TSFFP-Alg} has exponential moments with mean bounded by
\[ 
\kappa_{\hyperref[alg:improved_triple_temper_stable]{TSFFP}}
(1+qb(0)/\alpha)((1-\alpha)^{-3}+|\log\alpha|+\log N)/\alpha,
\]
where $\kappa_{\hyperref[alg:improved_triple_temper_stable]{TSFFP}}$ does not depend on  $\alpha\in(0,1)$, $\theta\in(0,\infty)$ or on $q\in[0,\infty)$, and $N$ is as in Theorem~\ref{thm:tempered_expect_time}.
\end{cor}

\subsection{Sampling the first passage of a stable subordinator}
\label{subsec:simulation_stable_vector}

The algorithm \hyperref[alg:triple_stable_conditional_on_time]{SFP-Alg} for sampling $(\tau_b,S_{\tau_b-},S_{\tau_b})|\{\tau_b\les t_*\}$ under the stable law $\p$ is based on the formulae in Proposition~\ref{prop:joint_law} below, describing the law of $(\tau_b,S_{\tau_b-},S_{\tau_b})$ under $\p$. Before presenting \hyperref[alg:triple_stable_conditional_on_time]{SFP-Alg}, we give a brief account of its contents. First,  \hyperref[alg:triple_stable_conditional_on_time]{SFP-Alg} samples the crossing time $\tau_b$ of the stable subordinator, conditional on $\tau_b\les t_*$, in line~1 and stores its value in the variable $T$. The algorithm then determines in lines~2~\&~3 whether the crossing occurred via a jump $\Delta_S(\tau_b)\coloneqq S_{\tau_b}-S_{\tau_b-}$ of positive size $\Delta_S(\tau_b)>0$ or continuously (i.e., by creeping, with $\Delta_S(\tau_b)=0$). In the latter case, the algorithm simply returns the crossing triple (line~4) and, in the former, the algorithm samples the undershoot and the overshoot and returns the triplet (lines~6--8).

\begin{algorithm}
\caption{(SFP-Alg) Stable First-Passage  Algorithm: samples $(\tau_b,S_{\tau_b-},S_{\tau_b})$ conditional on $\tau_b\les t_*$ under $\p$}
\label{alg:triple_stable_conditional_on_time}
\begin{algorithmic}[1]
\Require{Parameters $\alpha\in(0,1)$, $\theta>0$, $t_*\in(0,\infty]$, functions $b:[0,\infty)\to[0,\infty)$, $B:t\mapsto t^{-1/\alpha}b(t)$ and its inverse $B^{-1}:(0,\infty)\to(0,T_b)$ specified in Proposition~\ref{prop:joint_law} (with convention $B(\infty)=0$)}
\State{Generate $V_1\sim\mathcal{L}(S_1)$ via \hyperref[alg:stable]{S-Alg} until $V_1\ges B(t_*)$ and set $T\gets B^{-1}(V_1)$}\label{line:tau<t_*}
\State{Generate $U_1\sim\Unif(0,1)$}
\If{$U_1\les -b'(T)/(-b'(T)+\alpha^{-1}T^{-1}b(T))$}
\State{\Return $(T, b(T), b(T))$}
\Else
\State{Generate $U_2\sim\Unif(0,1)$}
\State{Generate $V_2\sim\rmS\Unif_{\alpha}(T,b(T))$ with law of $S_{\tau_b-}|\{\tau_b=T,\,b(\tau_b)=b(T),\,\Delta_S(\tau_b)>0\}$ via \hyperref[alg:undershoot_stable]{SU-Alg}}
\label{line:undershoot_t<t_*}
\State{\Return $(T, V_2, V_2 + (b(T)-V_2)U_2^{-1/\alpha})$}
\EndIf
\end{algorithmic}
\end{algorithm}

\begin{thm}
\label{thm:expectation_of_time_complexity}
\hyperref[alg:triple_stable_conditional_on_time]{SFP-Alg} samples from the law of $(\tau_b,S_{\tau_b-},S_{\tau_b})$ under $\p$ conditional on $\tau_b\les t_*$ (for $0<t_*\les \infty$). 
The running time of \hyperref[alg:triple_stable_conditional_on_time]{SFP-Alg} has finite exponential moments and mean bounded by
\begin{gather*}
\kappa_{\hyperref[alg:triple_stable_conditional_on_time]{SFP}}((1-\alpha)^{-3}+|\log\alpha|+\log N)/\p[S_1\ges B(t_*)],
\end{gather*}
where the constant $\kappa_{\hyperref[alg:triple_stable_conditional_on_time]{SFP}}$ does not depend on  $\alpha\in(0,1)$, $\theta\in(0,\infty)$ or  $t_*\in(0,\infty]$ and
$N$ specifies the number of precision bits of the output.
\end{thm}

The proof of Theorem~\ref{thm:expectation_of_time_complexity} is in  Section~\ref{subsubsec:proof_of_main_thm_stable} below. If $t_*=\infty$, then $B(\infty)=0$, $\p[S_1>0]=1$ and Theorem~\ref{thm:expectation_of_time_complexity} implies that the expected running time of \hyperref[alg:triple_stable_conditional_on_time]{SFP-Alg} may increase as fast as $(1-\alpha)^{-3}$ as $\alpha\uparrow1$. A deterioration in the limit as $\alpha\uparrow1$ is expected since, in this limit, the stable subordinator $S$ converges to a pure-drift subordinator $\theta t$, for which the distribution of the triplet degenerates to a delta distribution. Similarly, the expected running time of \hyperref[alg:triple_stable_conditional_on_time]{SFP-Alg} may increase as fast as $|\log\alpha|$ as $\alpha\downarrow0$. This is expected as well since the stable subordinator $S$ converges to the trivial process $0$ as $\alpha\downarrow 0$. Numerical evidence for these increases in complexity and the theoretical bounds in Theorem~\ref{thm:expectation_of_time_complexity} can be found in Section~\ref{sec:numerics} below. 
We stress that existing algorithms for the simulation of the triplet $(\tau_b,S_{\tau_b-},S_{\tau_b})$ have infinite expected running time for any $\alpha\in(0,1)$, see Subsection~\ref{subsec:How_not_to_sample} below for more details. 

As we said before, \hyperref[alg:triple_stable_conditional_on_time]{SFP-Alg} has three main steps. First, the algorithm generates the first-passage time $T$ in line~1. This requires evaluating, up to $N$ bits of precision, the inverse $B^{-1}$ of the function $B:t\mapsto t^{-1/\alpha}b(t)$, which is defined on $(0,T_b)$ where $T_b\coloneqq \inf\{t>0:b(t)=0\}\in(0,\infty]$. This step may require a numerical inversion of $B$ (or, equivalently, numerically finding the root of $t\mapsto B(t)-V_1$). Such algorithms are standard in the literature (see Appendix~\ref{sec:NewtonRaphson}), making this a simple step. Second, the algorithm determines if the first passage occurred continuously (i.e., by creeping) in line~3. If the process crept, the triple equals $(T,b(T),b(T))$, but, if it did not, then we must sample the undershoot and overshoot of the process conditional on the crossing time. The third main step is precisely the simulation of the undershoot and overshoot in lines~7 and~8. By Proposition~\ref{prop:joint_law}, the conditional law of the overshoot is a power law and hence easy to sample from, while the law of the undershoot, although expressible in terms of the stable law and L\'evy measure, is complicated and hard to sample from. For this reason, the focus of the remainder of this section is to construct \hyperref[alg:undershoot_stable]{SU-Alg}, which samples from the law of the undershoot. 

\subsection{Simulation of the stable undershoot  via domination by a mixture of densities}
\label{subsec:undershoot}

Our main objective now is to present a simulation algorithm from the undershoot law $\rmS\Unif_\alpha(t,w)$, $t,w>0$, that is, the law of $S_{\tau_b-}$ under $\p$, conditional on $\{\tau_b=t, b(\tau_b)=w, \Delta_S(\tau_b)>0\}$. Our main tool is the general rejection-sampling principle. Unlike before, the tools and algorithms used here have little probabilistic interpretation and are mostly based on analytical properties and inequalities. The derivation of these inequalities, which validate \hyperref[alg:undershoot_stable]{SU-Alg}, will be developed and explained in Section~\ref{sec:proofs} below. For now, we only present the algorithms with a brief description of how it works. The algorithm has several nontrivial components that we will slowly introduce and explain. The reason for this is that the task of simulating from the undershoot law $\rmS\Unif_\alpha(t,w)$ comes with multiple pitfalls, with several tempting approaches having severe problems, see Subsection~\ref{subsec:How_not_to_sample} for an account of these pitfalls.

First, the density $f_{t,w}$ of $\rmS\Unif_\alpha(t,w)$ involves the density of the stable law, which can be written as a non-elementary integral of elementary functions (see Proposition~\ref{prop:joint_law}). Thus, by extending the state space and applying an elementary transformation, the problem is converted into producing samples from a density $\psi$ proportional to $(x,y)\mapsto (s-x^{-1/r})^{-\alpha}\sigma_\alpha(y)\me^{-\sigma_\alpha(y)x}$ on $(s^{-r},\infty)\times(0,1)$, where $r\coloneqq \alpha/(1-\alpha)$, $s\coloneqq (\theta t)^{-1/\alpha}w$ and $\sigma_\alpha$ is the convex function defined in terms of trigonometric functions in~\eqref{eq:zolotarev_density}. The density $\psi_s$ is difficult to sample from directly, so we first bound it with (and propose from) the mixture $\wt\psi_s=p\wt\psi^{(1)}_s+(1-p)\wt\psi^{(2)}_s$ of two simpler densities $\wt\psi^{(1)}_s$ and $\wt\psi^{(2)}_s$ on $(s^{-r},\infty)\times(0,1)$, which are respectively proportional to $(x,y)\mapsto\sigma_\alpha(y)\me^{-\sigma_\alpha(y)x}$ and $(x,y)\mapsto\sigma_\alpha(y)\me^{-\sigma_\alpha(y)x}(x-s^{-r})^{-\alpha}$ (see details in Subsection~\ref{subsec:Undershoot_simulation} below).

\begin{algorithm}
\caption{(SU-Alg) Stable  Undershoot Algorithm: samples $S_{t-}|\{\tau_b=t,\,b(\tau_b)=w,\,\Delta_S(\tau_b)>0\}$ under~$\p$}
\label{alg:undershoot_stable}
\begin{algorithmic}[1]
\Require{Parameters $\alpha\in(0,1)$ and $\theta,t,w>0$}
\State{Set $s\gets (\theta t)^{-1/\alpha}w$ and $r\gets\alpha/(1-\alpha)$, compute $p'$ in~\eqref{eq:psi_p'} and set $p\gets p'/(p'+1)$ }\label{line_1:alg:undershoot_stable}
\Repeat
\State{Generate $V_1,V_2\sim\Unif(0,1)$}
\If{$V_1<p$}
    \State{Generate $(\zeta_s,Y_s)$ with density $\wt\psi_s^{(1)}$ via \hyperref[alg:psi_1]{$\psi^{(1)}$-Alg}}
\Else
    \State{Generate $(\zeta_s,Y_s)$ with density $\wt\psi_s^{(2)}$ via \hyperref[alg:psi_2]{$\psi^{(2)}$-Alg}}
\EndIf
\Until{$V_2\les(s-\zeta_s^{-1/r})^{-\alpha}/((1-2^{\alpha-1})^{-\alpha}s^{-\alpha}+(2r)^\alpha s^{-r}(\zeta_s-s^{-r})^{-\alpha})$}\label{step:alg_undershoot_accept}
\State{\Return $(\theta t)^{1/\alpha}\zeta_s^{-1/r}$}
\end{algorithmic}
\end{algorithm}

\hyperref[alg:undershoot_stable]{SU-Alg} proposes from the mixture $\wt\psi_s$ in lines~2--8 and accepts the sample with probability proportional to $\psi_s/\wt\psi_s$. Lines~3~\&~4 of \hyperref[alg:undershoot_stable]{SU-Alg} decide whether the sample from the mixture $\wt\psi_s$ will come from $\wt\psi_s^{(1)}$ (sampled in line~5) or $\wt\psi_s^{(2)}$ (sampled in line~7). The simulation algorithms corresponding to the simpler densities $\wt\psi_s^{(1)}$ and $\wt\psi_s^{(2)}$ will be discussed later.

\begin{remark}
\label{rem:prob_p'}
The proportion $p'$ (and hence $p$) can be efficiently evaluated to arbitrary precision using Gauss--Legendre quadrature. In fact, the numerator in the formula in~\eqref{eq:psi_p'} equals the distribution function of a stable law while the integrand in the denominator is a convex geometric combination of the integrands of the density and distribution functions of a stable law. The computation time for $p'$ is thus assumed to be constant for any $s=(\theta t)^{-1/\alpha}w$, and to only depend on the number of precision digits $N$, see fast algorithms in e.g. \cite{nolan1997numerical,ament2018accurate}. The error of the quadrature is known to decay exponentially fast in the number of nodes but may also be sensitive to large values of the high order derivatives of the integrand~\cite{4272861}. In our case, the derivatives of the integrand also grow quickly in the order of differentiation, implying a locally large relative error in the numerical integral. Since, however, these derivatives are large precisely when the function $f$ is nearly vanishing, the overall relative error appears to be small, resulting in a complexity that is only a function of the precision $N$, as has been confirmed many times in the literature and by our own simulation times, see Section~\ref{sec:numerics} below.
\end{remark}

\begin{remark}
\label{rem:s-is-stable}
Line~7 of \hyperref[alg:triple_stable_conditional_on_time]{SFP-Alg} calls \hyperref[alg:undershoot_stable]{SU-Alg} on random input, making the parameter $s=(\theta t)^{-1/\alpha}w$ in \hyperref[alg:undershoot_stable]{SU-Alg} equal to $\theta^{-1/\alpha}V_1$ (where $V_1\sim\mathcal{L}(S_1)$ is as in \hyperref[alg:triple_stable_conditional_on_time]{SFP-Alg}), which follows the stable law of $S_{1/\theta}$ conditioned on $S_{1/\theta}\ges B(t_*)$. In particular, the density of $s$ is proportional to the restriction to $[B(t_*),\infty)$ of $\varphi_\alpha$, defined in~\eqref{eq:zolotarev_representation}, making $s$ heavy tailed and its law dependent only on $\alpha$ and $B(t_*)$.
\end{remark}

For any $a,x>0$, define $\log_a^+(x):=\max\{0, \log(x)/\log(a)\}$.
\begin{prop}
\label{prop:undershoot_algorithm}
\hyperref[alg:undershoot_stable]{SU-Alg} samples from the law $\rmS\Unif_\alpha(t,w)$ of the undershoot $S_{t-}$ under $\p$, conditional on the event $\{\tau_b=t,b(\tau_b)=w,\Delta_S(\tau_b)>0\}$. The random running time of \hyperref[alg:undershoot_stable]{SU-Alg} has exponential moments. Moreover, the expected running time of \hyperref[alg:undershoot_stable]{SU-Alg} is bounded above by 
\begin{align*}   \kappa_{\hyperref[alg:undershoot_stable]{SU}}\Bigl[(1-\alpha)^{-1}&\Bigl(1+\alpha(1-\alpha)^{-1}\log_2^+(s^{-1})+\log_2^+\big(1/(s^{\alpha/(1-\alpha)}-\alpha^{(1-2\alpha)/(1-\alpha)}(1-\alpha))\big) \\
&+\alpha\log_2^+(s)+|\log\alpha|+(1-\alpha)^{-2}\Bigr)+\log N\Bigr],
\end{align*}
where the constant $\kappa_{\hyperref[alg:undershoot_stable]{SU}}>0$ does not depend on $s=(\theta t)^{-1/\alpha}w\in(0,\infty)$ (see line~\ref{line_1:alg:undershoot_stable} of \hyperref[alg:undershoot_stable]{SU-Alg}), $t>0$, $\theta\in(0,\infty)$ or $\alpha\in(0,1)$.
\end{prop}

The proof of Proposition~\ref{prop:undershoot_algorithm} is given in Subsection~\ref{subsubsec:proof_conditional_law_complexity} below. In the following two subsections we describe how to sample from the laws given by the densities $\wt\psi^{(1)}_s$ and $\wt\psi^{(2)}_s$.

\subsubsection{Simulation from the density
\texorpdfstring{$\wt\psi_s^{(1)}$}{psi 1}.}
\label{subsec:psi_(1)}

The density $\wt\psi_s^{(1)}(x,y)$ is proportional to 
\[
(x,y)\mapsto\sigma_\alpha(y)\me^{-\sigma_\alpha(y)x}\1_{(s^{-r},\infty)}(x)\1_{(0,1)}(y)
=\me^{-\sigma_\alpha(y)s^{-r}}\1_{(0,1)}(y)
\times \sigma_\alpha(y)\me^{-\sigma_\alpha(y)(x-s^{-r})}\1_{(s^{-r},\infty)}(x).
\]
The first factor (the marginal of the second component) is proportional to a log-concave density and can be sampled easily via \hyperref[alg:Devroye]{LC-Alg} by virtue of being log-concave (by Lemma~\ref{lem:varphi_convex} established below). The second factor (the conditional density of the first component given the second) is a shifted exponential with shift $s^{-r}$ and scale $1/\sigma_\alpha(y)$, which is also easy to sample.

\begin{algorithm}
\caption{($\psi^{(1)}$-Alg) Simulation from the density $\wt\psi_s^{(1)}$\label{alg:psi_1}}
\begin{algorithmic}[1]
\Require{Parameters $\alpha\in(0,1)$ and $s>0$}
\State{Sample $Y_s$ with density $y\mapsto \exp(-\sigma_\alpha(y)s^{-r})$ via \hyperref[alg:Devroye]{LC-Alg}}
\State{Sample $E\sim\Exp(1)$ and \Return $(s^{-r}+E/\sigma_\alpha(Y_s),Y_s)$}
\end{algorithmic}
\end{algorithm}

\begin{prop}\label{prop:psi_1}
\hyperref[alg:psi_1]{$\psi^{(1)}$-Alg} samples from the density $\wt\psi_s^{(1)}$. Moreover, the running time of \hyperref[alg:psi_1]{$\psi^{(1)}$-Alg} is bounded by a multiple of $G+\alpha(1-\alpha)^{-1}\log^+(s^{-1})$ where $G$ is a geometric random variable with parameter $1/5$. In particular, the running time has exponential moments with mean bounded by 
\[
\kappa_{\hyperref[alg:psi_1]{(1)}}\left(1+\alpha(1-\alpha)^{-1}\log^+(s^{-1})\right),
\]
where the constant $\kappa_{\hyperref[alg:psi_1]{(1)}}>0$ depends neither on $s\in(0,\infty)$ nor $\alpha\in(0,1)$.
\end{prop}

The proof of Proposition~\ref{prop:psi_1} is given in Subsection~\ref{subsec:psi_(1)-2} below. \hyperref[alg:psi_1]{$\psi^{(1)}$-Alg} does not rely on numerical inversion (cf.~\hyperref[alg:inversion_newton_raphson]{NR-Alg}). Its expected running time does therefore not depend on the number of precision bits $N$. 
In fact, the computational complexity of \hyperref[alg:psi_1]{$\psi^{(1)}$-Alg} primarily comes from the pre-processing step (line~\ref{line_find_a_1} of \hyperref[alg:Devroye]{LC-Alg}), which depends on the values of the parameters in \hyperref[alg:psi_1]{$\psi^{(1)}$-Alg}.

\subsubsection{Simulation from the density \texorpdfstring{$\wt\psi_s^{(2)}$}{psi 2}.}
\label{subsec:psi_(2)}

Simulating from the density $\wt\psi_s^{(2)}$ is a delicate matter, particularly because of the sensitivity of the function on $\alpha\in(0,1)$ and on $s$, which is heavy-tailed in line~7 of \hyperref[alg:triple_stable_conditional_on_time]{SFP-Alg} (see Remark~\ref{rem:s-is-stable} above). \hyperref[alg:psi_2]{$\psi^{(2)}$-Alg} below samples from the density $\wt\psi_s^{(2)}(x,y)$, proportional to $\sigma_\alpha(y)\me^{-\sigma_\alpha(y)x}(x-s^{-r})^{-\alpha}$, in two steps. The algorithm first samples, in lines~1--25, the second marginal of $\wt\psi_s^{(2)}$, with density proportional to $y\mapsto \sigma_\alpha(y)^\alpha\me^{-\sigma_\alpha(y)s^{-r}}$ on $(0,1)$. Then, the algorithm samples, in line~26, the first marginal conditionally given the second, with conditional density proportional to $x\mapsto (x-s^{r})^{-\alpha}\me^{-\sigma_\alpha(y)(x-s^{r})}$ on $(s^{-r},\infty)$, which is simply a translated (by $s^{-r}$) gamma random variable with shape $\alpha$ and scale $1/\sigma_\alpha(y)$ and hence easy to sample from. The simulation from the second marginal density is done by breaking up the support depending on the parameters $\alpha$ and $s$ and using rejection sampling with respect to various dominating (proposal) densities on each subinterval. Finding appropriate dominating functions, from which we can sample, is hard since the function $y\mapsto \sigma_\alpha(y)^\alpha\me^{-\sigma_\alpha(y)s^{-r}}$ may have a very large peak at level $\me^{-\alpha}\alpha^\alpha s^{\alpha^2/(1-\alpha)}$ followed by a extremely steep descent to $0$, even for reasonable values of $\alpha$, if the parameter $s$ (recall $s$ is heavy-tailed, see Remark~\ref{rem:s-is-stable} above) takes moderately large values. Figure~\ref{fig:psi2_marginal} below illustrates how peaked the corresponding density may be in such cases.

\begin{figure}[ht]
\centering
    \includegraphics[width=10cm,height=8cm]{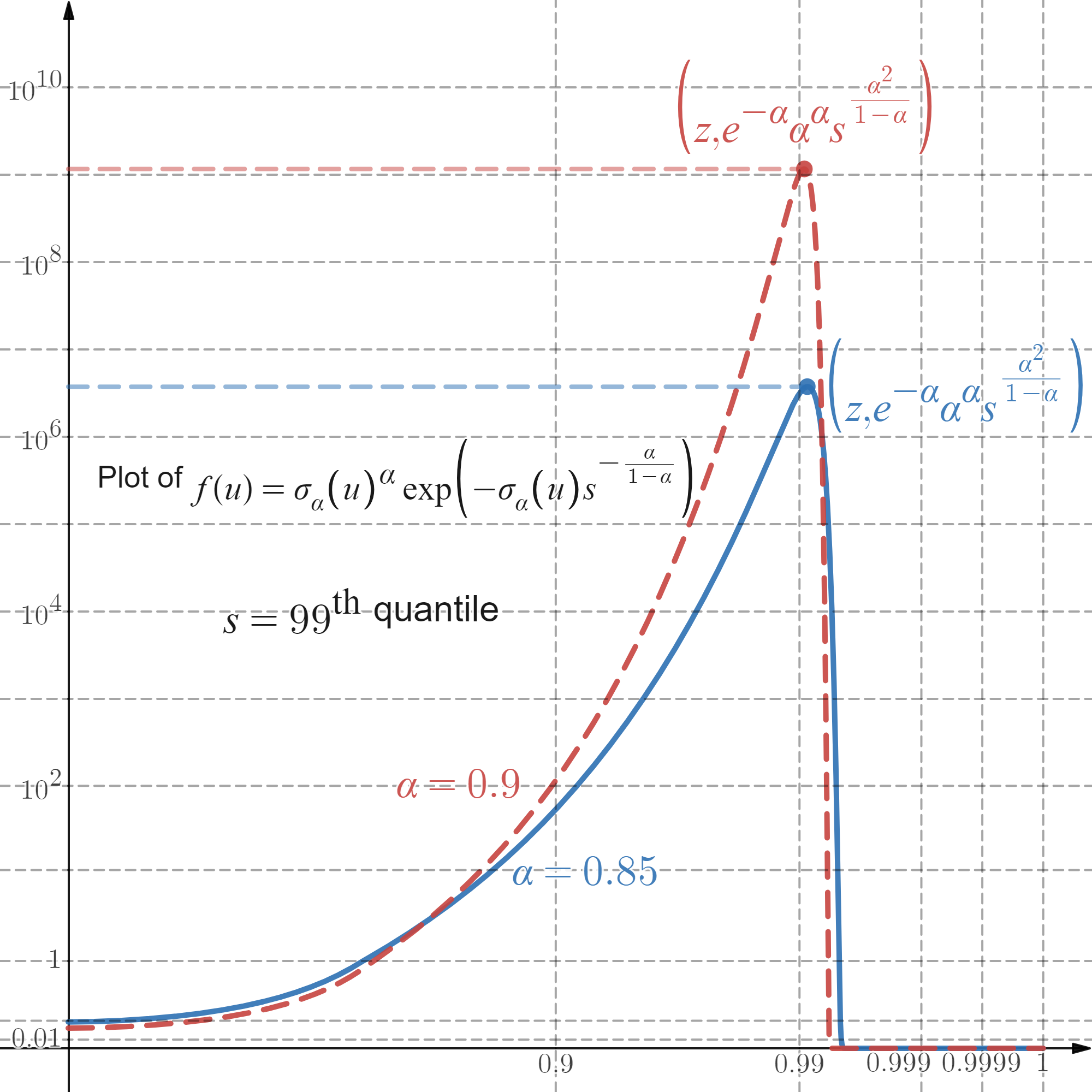}
    \caption{Density of the second marginal of $\wt\psi_s^{(2)}$.\label{fig:psi2_marginal} The figure shows the graphs of the density of the second marginal of $\wt\psi_s^{(2)}$, in a custom logarithmic-style scale, for two processes with $\alpha=0.85$ and $\alpha=0.9$, where the parameter $s$ is at its $99^{\text{th}}$ quantile.
    Thus, on average in one of every $100$ runs of \hyperref[alg:triple_stable_conditional_on_time]{SFP-Alg}, \hyperref[alg:psi_2]{$\psi^{(2)}$-Alg} is called to sample from a density which, at the mode $z=\sigma_\alpha^{-1}(\alpha s^r)$,  takes extremely large value followed by a steep drop towards zero. Moreover, in this case 
    the density is strictly increasing on the interval $(0,z)$ and strictly positive on $(0,1)$. But since  its derivatives
    of all orders are equal to zero at $1$, the graphs appear to be equal to zero on most of the interval $(z,1)$.
    Furthermore, the behaviour of the density on either side of the mode $z$ is markedly different. Taking a uniform upper bound $\me^{-\alpha}\alpha^\alpha s^{\alpha^2/(1-\alpha)}$ would result in an extremely inefficient algorithm with infinite expected running time (in fact, such a bound was used to sample from this type of function in~\cite{chi2016exact}, see Subsection~\ref{subsec:Chi} below). For parameter values as in the graph above, \hyperref[alg:psi_2]{$\psi^{(2)}$-Alg} sets $z_*=1/2$, splits the interval $(0,1)=(0,z_*]\cup(z_*,z]\cup(z,1)$ and uses appropriate proposal densities on each subinterval.}
\end{figure}

\hyperref[alg:psi_2]{$\psi^{(2)}$-Alg} needs to numerically invert several twice differentiable functions $f:[0,1]\to\R$ to produce samples of certain dominating (proposal) densities. This inversion is achieved via the Newton--Raphson root-finding method given in \hyperref[alg:inversion_newton_raphson]{NR-Alg}.  Each call of \hyperref[alg:inversion_newton_raphson]{NR-Alg} requires the input of a specific target function $f$ and an auxiliary function $M:[0,1]\times\N\to\R$ satisfying~\eqref{eq:newton_raphson_auxiliary} (see Appendix~\ref{sec:NewtonRaphson}).

\begin{algorithm}
  \caption{($\psi^{(2)}$-Alg) Simulation from the density $\wt\psi_s^{(2)}$\label{alg:psi_2}}
  \begin{algorithmic}[1]
\Require{Parameters $\alpha\in(0,1)$ and $s>0$}
\State{Set $z\gets 0$ if $\alpha s^r<\sigma_\alpha(0+)$. Else $z\gets\sigma_\alpha^{-1}(\alpha s^r)$ via
\hyperref[alg:inversion_newton_raphson]{NR-Alg} with $f=\log\sigma_\alpha$ and $M$ in~\eqref{eq:sigma_auxiliary_function}}
\label{step:inversion_sigma}
\State{Define the function $\tilde f(y)=\sigma_\alpha(y)^\alpha\exp(-\sigma_\alpha(y)s^{-r})$ and set $z_*\gets \min\{z,1/2\}$}
\State{Generate discrete variable $D$ on $\{0,1,2\}$ with weights $(\int_0^{z_*}\tilde f(y)\md y,\int_{z_*}^z\tilde f(y)\md y, \int_{z}^1\tilde f(y)\md y)$}
\label{step:generate_discrete_rv}
\If{$D=2$}
    \State{Generate $Y_s$ from the density proportional to $\tilde f|_{[z,1]}$ via \hyperref[alg:Devroye]{LC-Alg}}
    \label{step:devroye_in_psi_2}
\ElsIf{$D=0$}
    \Repeat
        \State{Generate $Y_s$ with distribution $f$ in~\eqref{eq:distribution_u*sigma^alpha} 
        via \hyperref[alg:inversion_newton_raphson]{NR-Alg} with $M$ in~\eqref{eq:u*sigma^alpha_auxiliary}}
        \label{step:inversion_D=0}
        \State{Generate $V\sim\Unif(0,1)$}
    \Until{$V\les \exp(-\sigma_\alpha(Y_s)s^{-r})/(1+\alpha Y_s\sigma_\alpha'(Y_s)/\sigma_\alpha(Y_s))$}
\Else  \Comment{(here, $D=1$)}
\If{$\alpha\les1/2$}
    \State{Set $a\gets \sin(\pi(1-\alpha))$ and $b\gets\pi\alpha(1-\alpha)\cos(\pi\alpha)$}
    \Repeat
        \State{Generate $Y_s$ with distribution $f$ defined in~\eqref{eq:density_c1} via \hyperref[alg:inversion_newton_raphson]{NR-Alg} with $M$ in~\eqref{eq:c2_auxiliary}}\label{step:inversion_alpha<1/2}
        \State{Generate $V\sim\Unif(0,1)$}
    \Until{$V\les \sigma_\alpha(Y_s)^\alpha\me^{-\sigma_\alpha(Y_s)s^{-r}}a^{1-\alpha}2^r(1-Y_s)^r/(a+b(1-Y_s))$}
    \Else
    \State{Set $c\gets\rho'(1/2)/\rho(1/2)^2$, $a\gets \rho(1/2)^{r-1}$ and $b\gets \rho(z)^{r-1}-a$}
    \Repeat
        \State{Generate $Y_s$ with distribution $f$ defined in~\eqref{eq:distribution_c2} via \hyperref[alg:inversion_newton_raphson]{NR-Alg} with $M$ in~\eqref{eq:sigma_auxiliary_function}}
        \label{step:inversion_alpha>1/2}
        \State{Generate $V\sim\Unif(0,1)$}
    \Until{$V\les c\me^{-\sigma_\alpha(Y_s)s^{-r}}\rho(Y_s)^2/\rho'(Y_s)$}
    \EndIf
\EndIf
\State{Generate $E\sim\Gam(1-\alpha,1)$, set $\zeta_s\gets s^{-r}+E/\sigma_\alpha(Y_s)$ and \Return $(\zeta_s,Y_s)$}
\end{algorithmic}
\end{algorithm}

\begin{remark}
\label{rem:prob_discrete_rv}
For the same reason we mentioned in Remark~\ref{rem:prob_p'}, the integral $\int_a^b\tilde f(y)\md u$ can be computed efficiently to arbitrary precision via Gaussian--Legendre quadrature. The computation time 
for Step~\ref{step:generate_discrete_rv} in \hyperref[alg:undershoot_stable]{SU-Alg}
is thus assumed to be constant.
\end{remark}

\begin{prop}\label{prop:psi_2}
\hyperref[alg:psi_2]{$\psi^{(2)}$-Alg} samples from the 
law given by the density $\wt\psi_s^{(2)}$. Moreover, the running time of \hyperref[alg:psi_2]{$\psi^{(2)}$-Alg} is bounded above by a constant multiple of 
\[
G+\log^+\big(1/\big(s^{\alpha/(1-\alpha)}-\alpha^{(2\alpha-1)/(1-\alpha)}\big)\big)+\alpha\log^+(s)+|\log\alpha|+\log N,
\] 
where $N$ specifies the number of precision bits of the output and $G$ is a geometric random variable with  probability of success $\p(G=1)$ bounded below by a multiple of $(1-\alpha)^2$. In particular, the running time of \hyperref[alg:psi_2]{$\psi^{(2)}$-Alg} has exponential moments and its mean is bounded by 
$$
\kappa_{\hyperref[alg:psi_2]{(2)}} 
\left[(1-\alpha)^{-2}
    +\log^+\big(1/\big(s^{\alpha/(1-\alpha)}
    -\alpha^{(2\alpha-1)/(1-\alpha)}\big)\big)
    +\alpha\log^+(s)+|\log\alpha|+\log N\right],
$$ 
where the constant $\kappa_{\hyperref[alg:psi_2]{(2)}}$ depends neither on $s\in(0,\infty)$ nor $\alpha\in(0,1)$.
\end{prop}

The proof of Proposition~\ref{prop:psi_2} is given in Subsection~\ref{subsec:psi_(2)-2} below. The computational cost in \hyperref[alg:psi_2]{$\psi^{(2)}$-Alg} comes from various sources: 
(I) applications of the inversion \hyperref[alg:inversion_newton_raphson]{NR-Alg}, which may depend on the parameters $\alpha\in(0,1)$ and $s>0$ (recall that the parameter $s$ is in fact random in \hyperref[alg:undershoot_stable]{SU-Alg}, which calls \hyperref[alg:psi_2]{$\psi^{(2)}$-Alg});
(II) in the cases $D=0$ and $D=1$,
we also perform an accept-reject step where the acceptence probability depends on the parameters $\alpha\in(0,1)$ and $s>0$; (III) in the case $D=2$, the pre-processing (line~\ref{line_find_a_1} in \hyperref[alg:Devroye]{LC-Alg})
also adds to the final computational complexity since the function \hyperref[alg:Devroye]{LC-Alg} is applied to in line~\ref{step:devroye_in_psi_2}  of \hyperref[alg:psi_2]{$\psi^{(2)}$-Alg} depends both on
$\alpha\in(0,1)$ and $s>0$. We note that the accept-reject step in the case $D=2$ is embedded in \hyperref[alg:Devroye]{LC-Alg} and has uniformly bounded (in parameter values) expected running time. 


\subsection{How \emph{not} to sample the stable undershoot \texorpdfstring{$S_{\tau_b-}|\{\tau_b=t, b(\tau_b)=w, \Delta_S(\tau_b)>0\}$}{undershoot}}
\label{subsec:How_not_to_sample}

At a glance, there are several apparently feasible ways to produce a sample from the undershoot law $\Unif\rmS(t,w)$ of $S_{\tau_b-}$, under $\p$, conditional on $\{\tau_b=t, b(\tau_b)=w,\Delta_S(\tau_b)>0\}$. However, most of these simulation methods are either infeasible altogether or their computational complexity grows rapidly as either $\alpha\downarrow 0$ or $\alpha\uparrow 1$. In this section, we will briefly discuss some of these ``approaches''.

First recall that sampling from the undershoot law is, by Proposition~\ref{prop:joint_law}(c), essentially reduced to producing samples from the density proportional to $x\mapsto \varphi_\alpha(x)(s-x)^{-\alpha}\1_{\{x<s\}}$ where $s=(\theta t)^{-1/\alpha}w=(\theta \tau_b)^{-1/\alpha}b(\tau_b)>0$ and $\varphi_\alpha$ is the density of $S_{1/\theta}$ and given in~\eqref{eq:zolotarev_representation}. A first guess would be to employ rejection sampling on this density. Indeed, since the function $x\mapsto (s-x)^{-\alpha}$ is unbounded, we cannot propose from the density $\varphi_\alpha$. However, the density $\varphi_\alpha$ \emph{is} bounded and hence we may propose from the density proportional to $x\mapsto (s-x)^{-\alpha}$ and accept the sample with probability proportional to $\varphi_\alpha(x)$. For a fixed $s$, this algorithm has a running time with finite mean; however, we require an algorithm that performs well for $s=(\theta\tau_b)^{-1/\alpha}b(\tau_b)\eqd S_{1/\theta}$. Since $x\mapsto(s-a)^{-\alpha}$ puts most of its mass close to $s$ (which is stable distributed and thus heavy tailed), the acceptance probability of such an algorithm is asymptotically equivalent to $\varphi_\alpha(s)/\sup_{u\in(0,\infty)}\varphi_\alpha(u)\sim C s^{-\alpha-1}$ as $s\to\infty$ for some $C>0$. In turn, this means that the running time is bounded below by $C'\max\{1,S_{1}^{\alpha+1}\}$, for some $C'>0$, which does not have a finite mean (in fact, it has a finite moment $p>0$ if and only if $p<\alpha/(\alpha+1)$).

Another possibility would be to perform a numerical inversion of the distribution function $F_\alpha$ proportional to $x\mapsto\int_0^x \varphi_\alpha(y)(s-y)^{-\alpha}\md y$. Since 
\[
\frac{F_\alpha'(x)}{F_\alpha(x)}
=\frac{\varphi_\alpha(x)(s-x)^{-\alpha}}{\int_0^x \varphi_\alpha(y)(s-y)^{-\alpha}\md y},
\]
where $\varphi_\alpha$ can be computed to arbitrary precision using numerical quadrature, one may expect Newton--Raphson's method to perform well in inverting $F_\alpha$ numerically. It is, however, computationally intensive. Every step in Newton--Raphson method and in the binary search required prior to the quadratic convergence of Newton--Raphson's method evaluates the double integral in $F_\alpha$ (recall that $\varphi_\alpha$ is itself given in terms of an integral). If the numerical quadrature for computing $\varphi_\alpha$ accurately requires $n$ nodes, the corresponding quadrature for $F_\alpha$ needs approximately $n^2$ nodes. Since we typically have $n\in[10^4,10^5]$ for double-precision floating-point numbers, each numerical evaluation of $F_\alpha$ requires evaluating at least $10^8$ functions (sums, products, exponentiation and trigonometric functions). This makes this approach infeasible despite the good properties of Newton--Raphson method and the boundedness and smoothness of the function $\varphi_\alpha$.

A third option is to maintain parts of our current methodology and modify the parts of the algorithm that sample from a density of the form $u\mapsto\sigma_\alpha(u)^p\exp(-\sigma_\alpha(u)s)$ for some $p\in(0,\infty)\setminus\{1\}$. For instance, one may attempt a rejection-sampling algorithm via elementary bounds on $\sigma_\alpha$, such as those found in Lemma~\ref{lem:sigma_bound} below. Such algorithms are feasible and, in fact, we use some of them for certain parameter combinations. The drawback of those approaches is that the acceptance probability becomes tiny as $\alpha\to 1$, which is why we introduce the auxiliary parameter $z_*$ and other simulation algorithms for certain parameter combinations. Indeed, the acceptance probability is often upper bounded by $c^{r}$ for some $c\in(0,1)$ where we recall that $r=\alpha/(1-\alpha)$. This probability is incredibly small even for moderate values of $\alpha$. For instance, the quotient between the lower and upper bounds on $\sigma_\alpha$ in Lemma~\ref{lem:sigma_bound} is proportional to $(\pi/4)^{r+1}(1-\alpha)^{r}$, as $\alpha\to 1$, which is approximately $4.28\times10^{-6}$ (resp. $8.93\times10^{-11}$) when $\alpha=0.85$ (resp. $\alpha=0.9$).

Our \hyperref[alg:undershoot_stable]{SU-Alg} (and its subalgorithms) are in fact informed by the pitfalls described above. In particular, the choices made in \hyperref[alg:psi_2]{$\psi^{(2)}$-Alg} may appear arbitrary but are, in fact, designed to control the computational complexity. It remains an open problem to design an exact simulation algorithm with uniformly bounded expected running time on the entire parameter space $\alpha\in(0,1)$. Figure~\ref{fig:large_alpha} below appears to suggest that the complexity of our algorithm grows as $(1-\alpha)^{-0.2}$ as $\alpha\uparrow1$, which is much smaller than $(1-\alpha)^{-3}$, as suggested by our theoretical bounds. 

\subsection{Simulation of \texorpdfstring{$\wt\psi_s^{(1)}$}{psi1} and \texorpdfstring{$\wt\psi_s^{(2)}$}{psi2} via numerical inversion with numerical integrals}
\label{subsec:direct_numerical_inversion}
It is natural to enquire\footnote{We thank an anonymous referee for drawing our attention to this questions.} about the performance of a simulation method based on a direct numerical inversion of the distribution functions corresponding to the marginals $y\mapsto \int \wt\psi_s^{(1)}(x,y)\md x$ and $y\mapsto \int \wt\psi_s^{(2)}(x,y)\md x$.
Using direct numerical inversion via Householder's method with $k=4$ (see \hyperref[alg:inversion_householder]{H-Alg} in Appendix~\ref{sec:NewtonRaphson} below), we can sample from these marginals of $\wt\psi_s^{(1)}$ and $\wt\psi_s^{(2)}$. 

Recall that \hyperref[alg:psi_1]{$\psi^{(1)}$-Alg} requires no numerical integration, while \hyperref[alg:psi_2]{$\psi^{(2)}$-Alg} uses 3 such integrations (see line~\ref{step:generate_discrete_rv}), each over an interval where the integrand is monotone. In contrast, 
since every step of \hyperref[alg:inversion_householder]{H-Alg} evaluates the distribution function, the numerical integration is used much more extensively over regions where the integrands are not monotone. 
Thus the direct inversion method is more 
sensitive to large values of the number $N$ of precision digits and to instabilities in the numerical integration appearing for such large $N$. 

We compared numerically the outputs of \hyperref[alg:psi_1]{$\psi^{(1)}$-Alg} and \hyperref[alg:psi_2]{$\psi^{(2)}$-Alg} with the outputs of their 
direct inversion counterparts:  Figure~\ref{fig:psi_num_inv} shows that our algorithms and the direct inversion alternatives sample from the same laws. We found that  algorithms 
\hyperref[alg:psi_1]{$\psi^{(1)}$-Alg} and \hyperref[alg:psi_2]{$\psi^{(2)}$-Alg}
were on average approximately $32.01$ and $14.82$ times faster, respectively than the direct inversion alternatives over the range of
$\alpha\in\{.1,.3,.5,.7,.9\}$.

\begin{figure}[ht]
\centering
    \includegraphics[width=8.3cm]{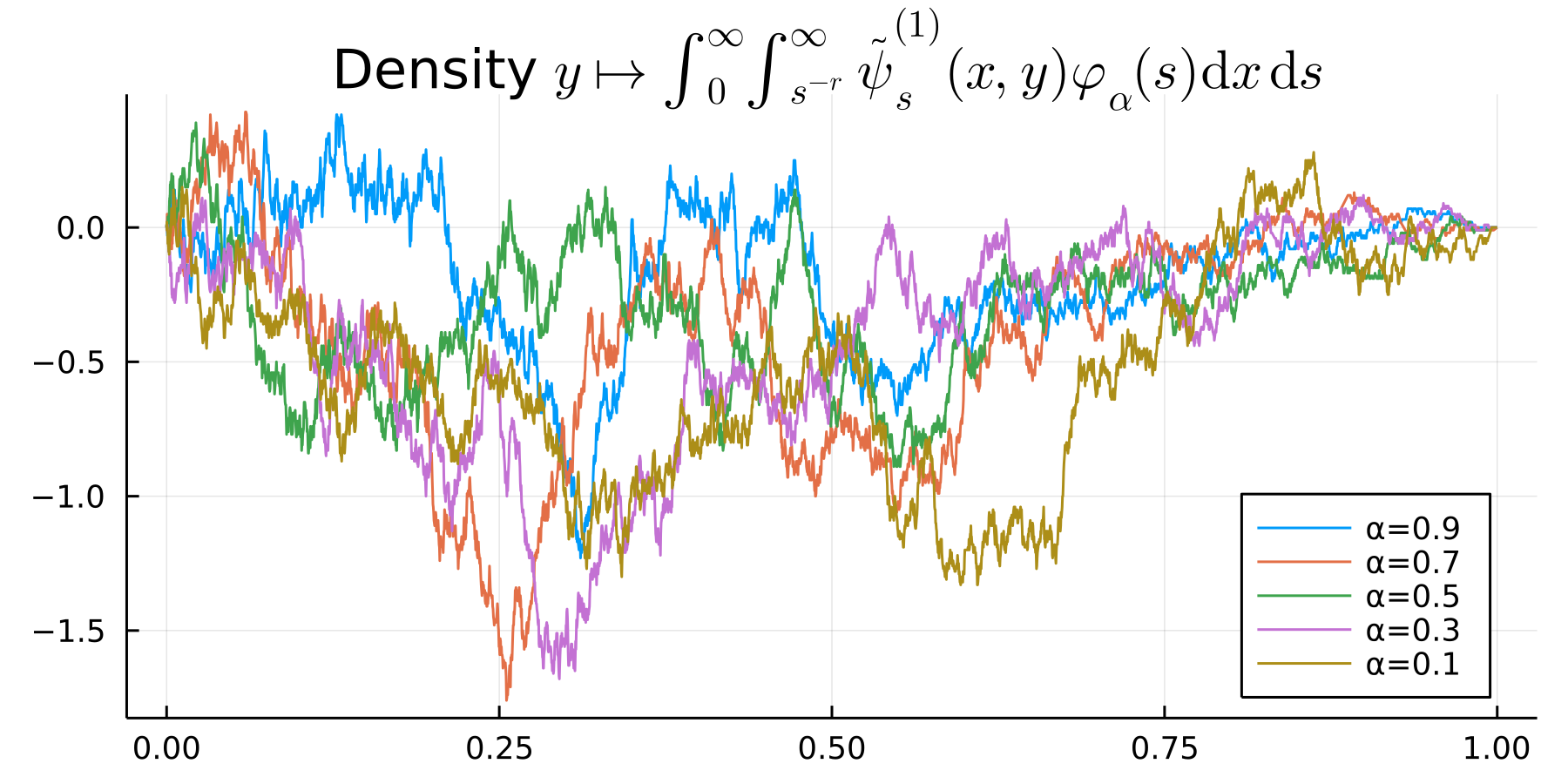}
    \includegraphics[width=8.3cm]{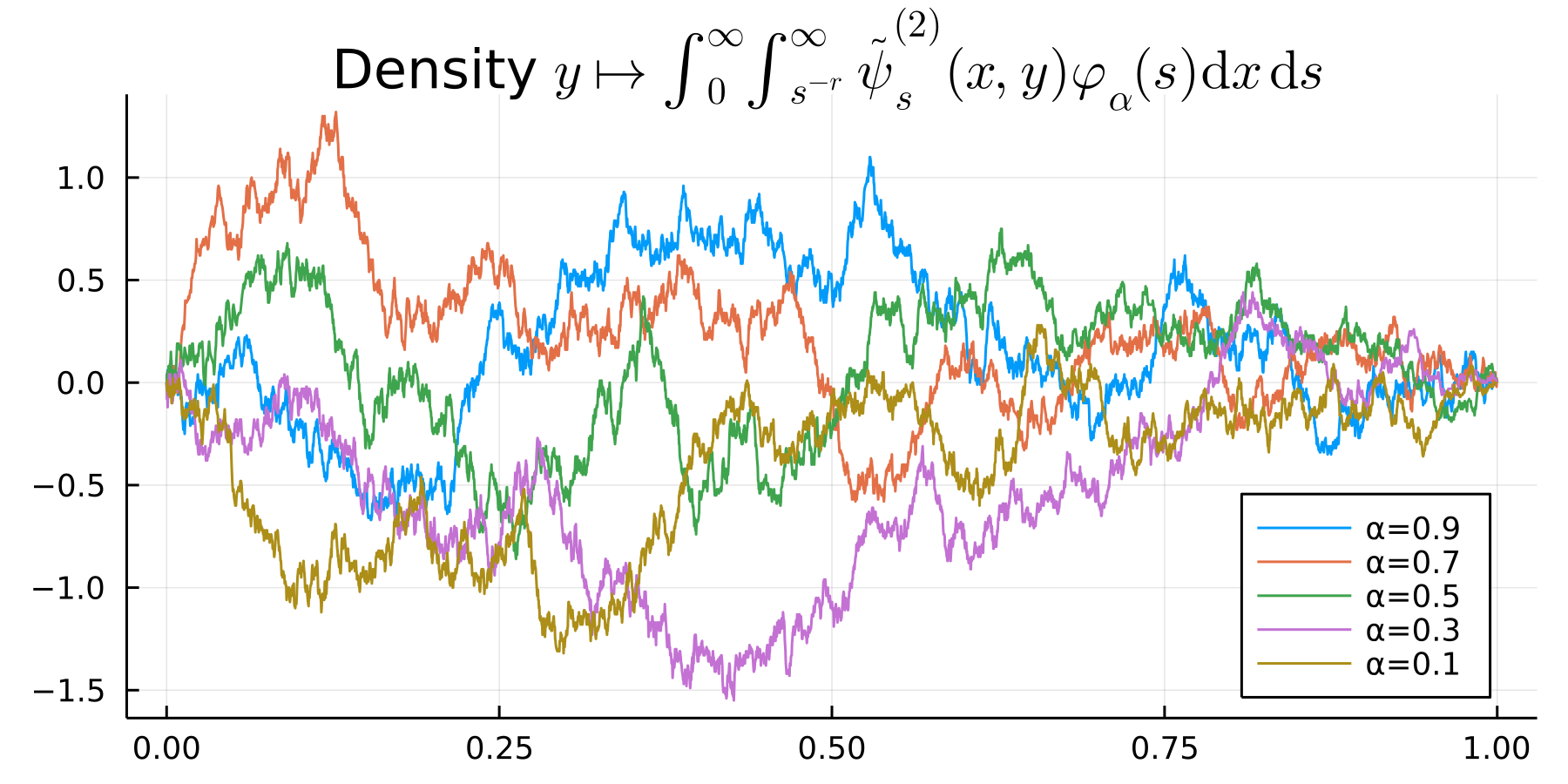}\caption{Numerical inversion with numerical integrals.\label{fig:psi_num_inv} The graph on the left (resp. right) show the function $y\mapsto \sqrt{n}(F_n(y)-G_n(y))$ where $F_n$ and $G_n$ are the empirical distribution functions of the samples produced using \hyperref[alg:psi_1]{$\psi^{(1)}$-Alg} (resp. \hyperref[alg:psi_2]{$\psi^{(2)}$-Alg}) and direct numerical inversion, respectively, for $n=10^4$. Such functions indeed resemble  Brownian bridges time-changed  by the corresponding distribution functions. All samples pass the two-sample Kolmogorov-Smirnov test, confirming that our algorithms and direct numerical inversion produce samples from the same laws.}
\end{figure}

\subsection{Infinite expected running time of the algorithm in~\texorpdfstring{\cite[\S7]{chi2016exact}}{Chi (2016)}}
\label{subsec:Chi}

In~\cite[\S7]{chi2016exact},  an algorithm that produces a sample of the first-passage triplet of a tempered stable process is given. Such an algorithm has almost surely finite running time. However, it can be seen as follows, that its expected running time is infinite. 

Suppose for simplicity that $\theta=1$ and $b$ is bounded by the variable denoted by $r$ in the algorithm in~\cite[\S7]{chi2016exact}. Consider Step 3 of the algorithm in~\cite[\S7]{chi2016exact}, which uses rejection sampling to simulate the pair $(S_{\tau_b-},\Delta_{S}(\tau_b))$ given $\{\tau_b=t\}$ (recall from Proposition~\ref{prop:joint_law} that  $\tau_b\eqd B^{-1}(S_1)$, where $B^{-1}$ is the inverse of $t\mapsto t^{-1/\alpha}b(t)$). The proposal for $s=S_{\tau-}$ is $s=\beta b(t)$ for some $\beta\sim \text{Beta}(1,1-\alpha)$, with acceptance probability $p=h(t^{-1/\alpha}s,U)/M_\alpha$, where $U\sim\Unif(0,1)$ is an auxiliary variable, $h(x,u)\coloneqq \sigma_\alpha(u)x^{-r-1}\me^{-\sigma_\alpha(u)x^{-r}}$ and $M_\alpha\coloneqq(1-\alpha)^{1-1/\alpha}\alpha^{-1-1/\alpha}\me^{-1/\alpha}$ is a global bound on $h$. Thus, the expected runnig time of this step equals $\E[C]$, where $C\coloneqq 1/p\eqd M_\alpha/h(\beta S_1,U)$ and the variables $U$, $\beta$ and $S_1$ are independent. Hence, recalling that $r=\alpha/(1-\alpha)$, we have
\begin{equation}
\begin{aligned}
\E[C]&=M_\alpha\E[\sigma_\alpha(U)^{-1}(S_1\beta)^{r+1}\exp(\sigma_\alpha(U)(S_1\beta_1)^{-r})]
\ges M_\alpha\E[\sigma_\alpha(U)^{-1}(S_1\beta)^{r+1}] \\
&=M_\alpha\E[\sigma_\alpha(U)^{-1}]\E[S_1^{r+1}]\E[\beta^{r+1}]
\end{aligned}
\end{equation}
However, $\E[S_1^\eta]<\infty$ if and only if $\eta<\alpha$ by~\eqref{eq:mellin_tranform} below. Since $r=\alpha/(1-\alpha)>0$ for all $\alpha\in(0,1)$, the expected running time of the algorithm in~\cite[\S7]{chi2016exact} is indeed infinite. 

We note that, similarly, Step 5 of the algorithm in~\cite[\S7]{chi2016exact} also has infinite expected running time. Since Step~5 is more complex than Step~3, we made the phenomenon above explicit only for Step~3.

\subsection{Can an implemented simulation algorithm  be exact?}
\label{subsec:exact?}

The answer depends on the definition of exact simulation. Since our algorithms are implemented inside computers with finite resources, the output of any algorithm cannot have the same law as the variable it is simulating if the variable has a density. However, we may consider an algorithm to be \emph{practically exact} if a chosen distance between the law of the ideal variable $I$ and the algorithm's output $A$ can be controlled within a given multiple of \emph{machine precision} (see Appendix~\ref{app:exact} for further discussion of this question and various natural suggestions for the distance). Our algorithms are practically exact with respect to the Prokhorov metric (which metrises weak convergence) because we can create a coupling under which, with high probability, the distance $|I-A|$ is bounded by a constant multiple of machine precision, and this multiplicative constant is only dependent on the parameters $(\alpha,\,\theta\,q)$ and $b(0)$, see details in Appendix~\ref{app:exact} below.

\section{Applications}
\label{sec:applications}
The main objective of this section is to present some applications of our algorithms. First, we present some numerical evidence for the theoretical bounds on the expected running time of \hyperref[alg:improved_triple_temper_stable]{TSFFP-Alg} and \hyperref[alg:triple_stable_conditional_on_time]{SFP-Alg}. We then show how \hyperref[alg:improved_triple_temper_stable]{TSFFP-Alg} can be used to sample the first-passage event of a two-sided tempered stable process, which in turn can be used to price barrier options via Monte Carlo. Next, we use the probabilistic representation of the solution of a fractional partial differential equation (FPDE)~\cite{hernandez2017generalised} along with \hyperref[alg:improved_triple_temper_stable]{TSFFP-Alg} to produce Monte Carlo estimators of the solution of such a FPDE.

\subsection{Implementation}
\label{sec:numerics}

In this section, we showcase an implementation of our algorithms in Julia computing language. Our first goal is to show that such an implementation makes the simulation of the first-passage event feasible and fast in practice. Our second goal is to show that, in practice, our bounds on the expected running time of \hyperref[alg:improved_triple_temper_stable]{TSFFP-Alg} and \hyperref[alg:triple_stable_conditional_on_time]{SFP-Alg} overestimate the true complexity. 

\begin{figure}[ht]
\centering
\begin{subfigure}[b]{0.49\textwidth}
 \centering
 \includegraphics[width=\textwidth]{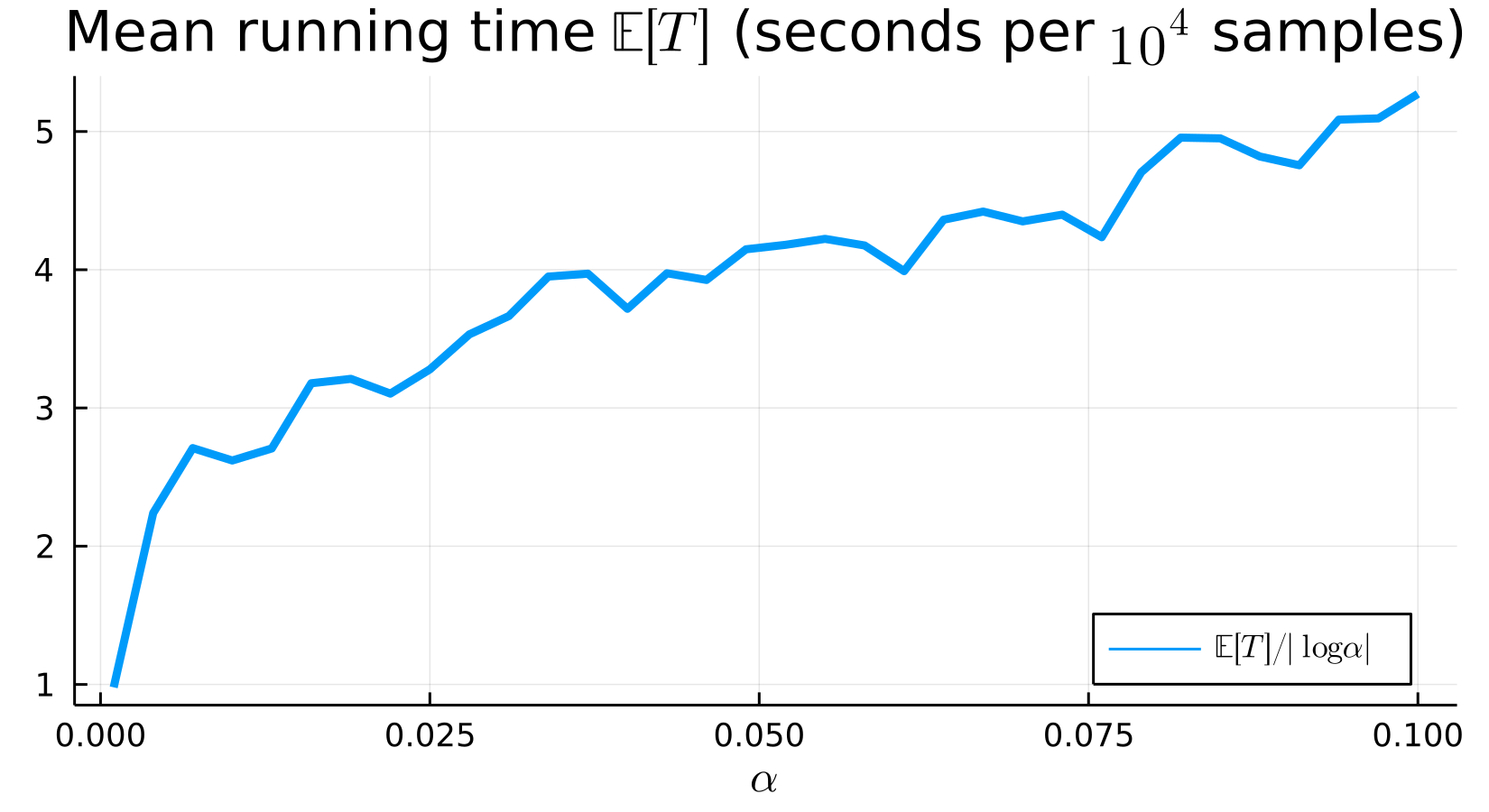}
 \caption{\footnotesize By Theorem~\ref{thm:expectation_of_time_complexity}, $\E[T]=\Oh(|\log\alpha|)$ as $\alpha\downarrow0$. The picture confirms this: $\E[T]/|\log\alpha|$ decreases to a number close to $1$ as $\alpha\downarrow0$ (the reported value of $\E[T]$ for $\alpha=0.001$ is $9.30$ seconds per $10^4$ samples).}
 \label{fig:small_alpha}
\end{subfigure}
\hfill
\begin{subfigure}[b]{0.49\textwidth}
 \centering
 \includegraphics[width=\textwidth]{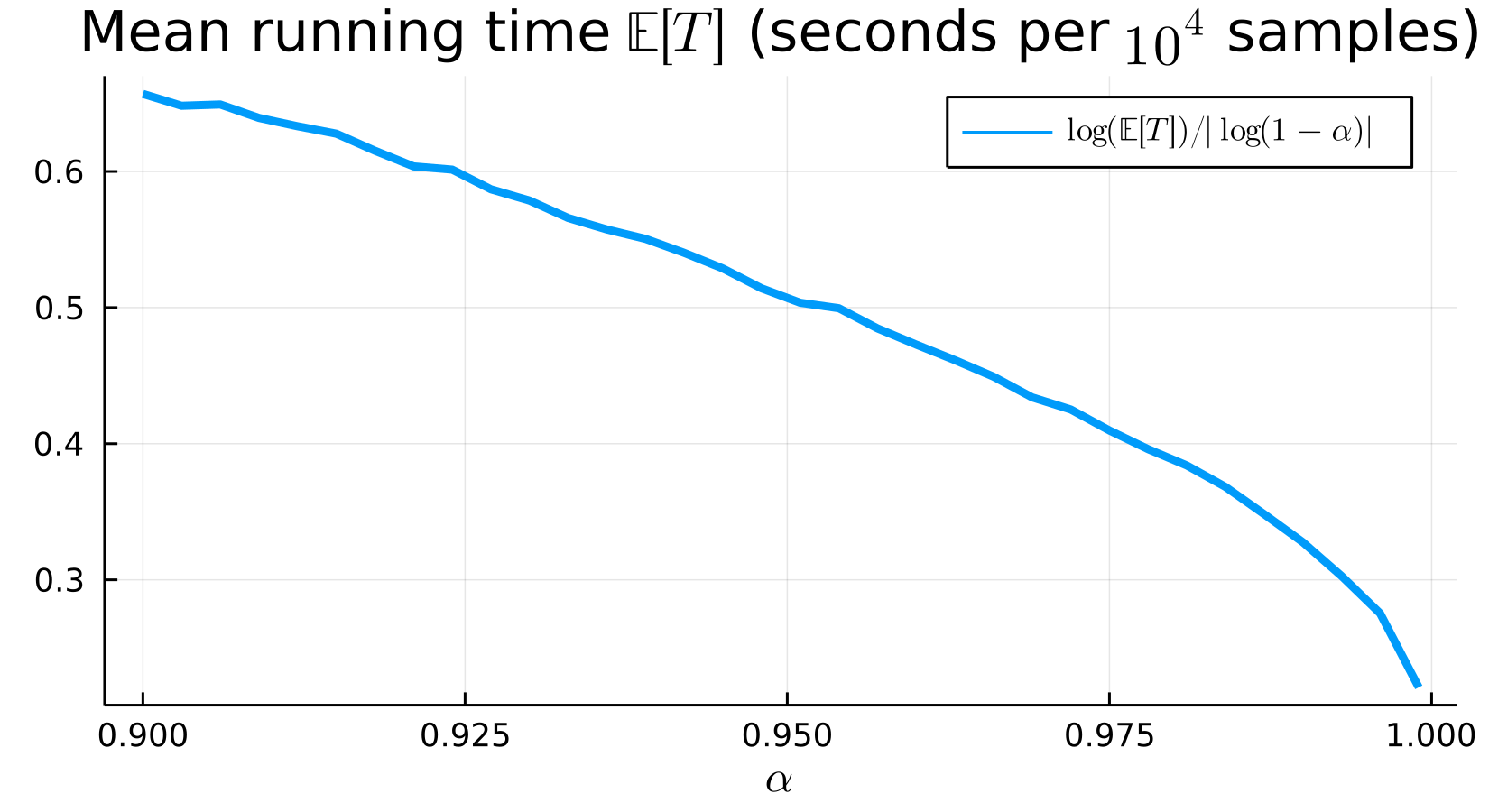}
 \caption{\footnotesize By Theorem~\ref{thm:expectation_of_time_complexity}, $\E[T]=\Oh((1-\alpha)^{-3})$ as $\alpha\uparrow1$. The picture confirms this: $\log(\E[T])/|\log(1-\alpha)|$ decreases as $\alpha\uparrow1$ to a number close to $0.2$ (the reported value of $\E[T]$ for $\alpha=0.999$ is $6.28$ seconds per $10^4$ samples).}
 \label{fig:large_alpha}
\end{subfigure}\caption{Implementation of \hyperref[alg:triple_stable_conditional_on_time]{SFP-Alg} with $t_*=\infty$ on Julia.\label{fig:stable_case} The average time $\E[T]$ is measured in seconds taken for every $10^4$ samples.\label{fig:large_small_alpha}}
\end{figure}

Observe that the expected running time in Figure~\ref{fig:large_small_alpha} appears to be smaller for $\alpha=0.999$ (close to $1$) than for $\alpha=0.001$ (close to $0$). Although this is surprising since our upper bounds are of order $\Oh((1-\alpha)^{-3})$ for large $\alpha$ and of order $\Oh(|\log\alpha|)$ for small $\alpha$, it is important to consider a few things. First, our key \hyperref[alg:psi_2]{$\psi^{(2)}$-Alg}, used to simulate the stable undershoot, uses a very different simulation method for $\alpha\les1/2$ and for $\alpha>1/2$. In particular, the multiplicative constants in the $\Oh$ notation can be significantly different. Second, the bounds in Theorem~\ref{thm:expectation_of_time_complexity} are only bounds, and need not be sharp. In particular, this does not necessarily mean that the asymptotic cost as $\alpha\uparrow1$ is truly smaller than the asymptotic cost as $\alpha\downarrow0$.

The behaviour of the expected running time of \hyperref[alg:improved_triple_temper_stable]{TSFFP-Alg},  as a function of $q>0$, is in good agreement with  our upper bound in Theorem~\ref{thm:tempered_expect_time}, see Figure~\ref{fig:tempered_stable_case} for details.

\begin{figure}[ht]
\centering
\includegraphics[width=.55\textwidth]{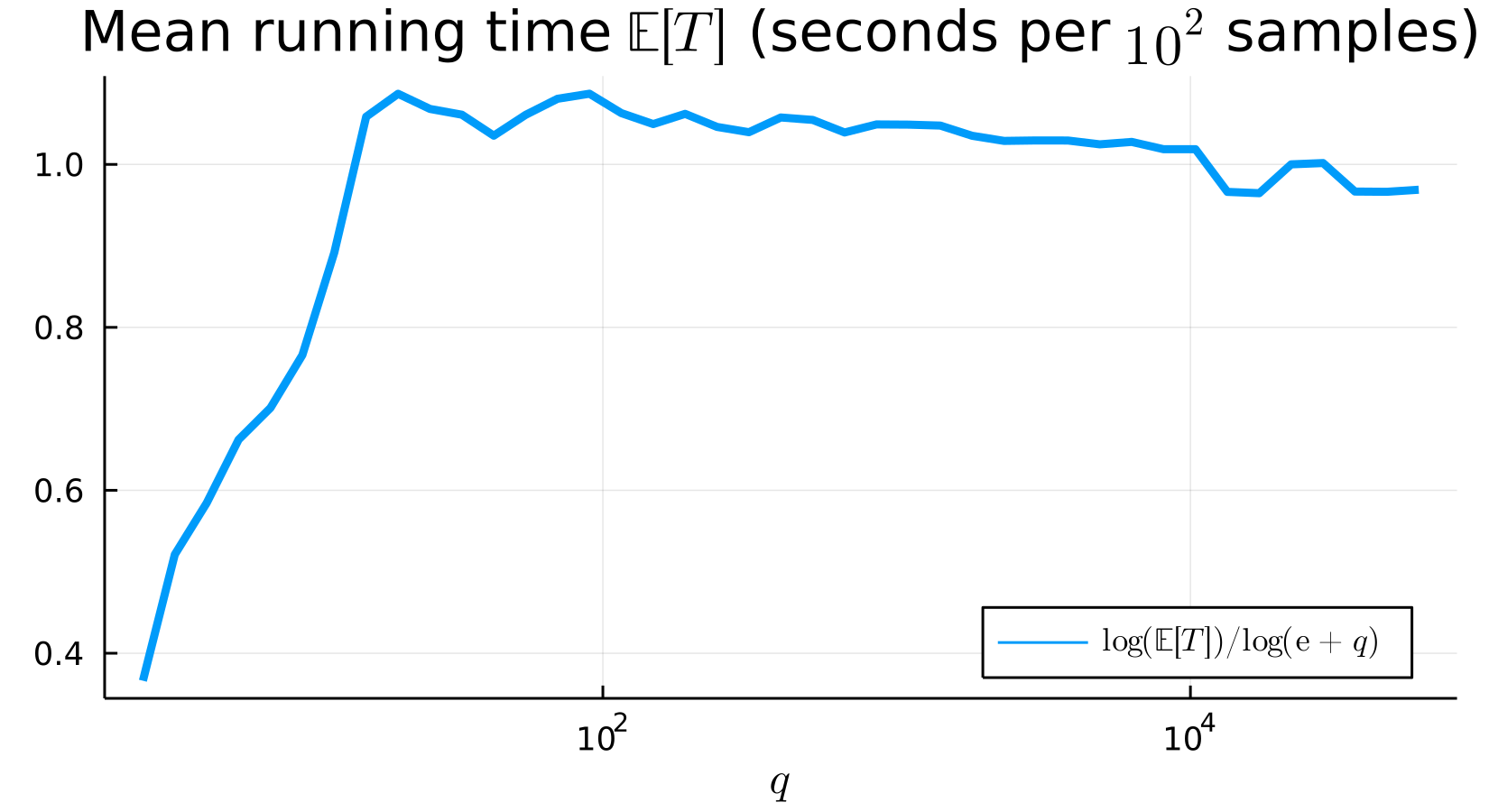}
\caption{Implementation of \hyperref[alg:improved_triple_temper_stable]{TSFFP-Alg}.\label{fig:tempered_stable_case}
The picture depicts an implementation of \hyperref[alg:improved_triple_temper_stable]{TSFFP-Alg} with parameters $\alpha=0.55$, $\theta=1$, $b(t)\equiv 1$ and $q\in\{\me^0, \me^{1/4},\ldots,\me^{11}\}$. The mean running time $\E[T]$ is measured in seconds taken for every $10^2$ samples. By Corollary~\ref{cor:improved_tempered_expected_time}, $\E[T]=\Oh(q)$ as $q\to\infty$. This is supported by the picture: the quotient $\log(\E [T])/\log(\me+q)$ appears to converge to $1$ as $q\to\infty$.}
\end{figure}

\subsection{Extensions to L\'evy processes of bounded variation}
\label{subsec:BV}

We start by recalling the algorithm in~\cite[\S5.1]{chi2016exact}.
Let $Z$ be a L\'evy process with non-positive drift, such that its L\'evy measure satisfies
\[
\nu(\RP)=\infty
\quad\text{and}\quad
\int_{\R}(|x|\land 1)\nu(\md x)<\infty.
\]
Decompose $Z$ as $Z^+-Z^-$, where $Z^+$ and $Z^-$ are independent subordinators with Levy measures $\nu^+$ and $\nu^-$, respectively, and where $Z^+$ is driftless. Fix some $c>0$ and let $\tau_c^Z:=\inf\{t>0:Z_t>c\}$. The following algorithm provides a method to simulate $\tau_c^Z$ under the assumption that the first-passage triplets $(\tau_x^+,Z_{\tau_x^+-},Z_{\tau_x^+})$, where $\tau_x^+=\inf\{t>0:Z_t^+>x\}$, as well as the increments $Z^+_t$ and $Z^-_t$ can be sampled for any $x>0$ and $t>0$. Further consider some $T\in(0,\infty]$. According to~\cite[Proposition~5.1]{chi2016exact}, if either $\limsup_{t\to\infty}Z_t=\infty$ a.s. (which can be determined, e.g., via Rogozin's criterion~\cite[Thm~2.7]{MR4448688}) or $T<\infty$, then \hyperref[alg:FPE_of_bounded_variation]{BVFP-Alg} below stops a.s. and it samples the triplet $(\min\{\tau_c^Z,T\},Z_{\min\{\tau_c^Z,T\}-},Z_{\min\{\tau_c^Z,T\}})$.

\begin{algorithm}
\caption{(BVFP-Alg)  Bounded Variation First-Passage Algorithm: for a bounded variation L\'evy process $Z$, samples $(\min\{\tau_c^Z,T\},Z_{\min\{\tau_c^Z,T\}-},Z_{\min\{\tau_c^Z,T\}})$ \label{alg:FPE_of_bounded_variation}}
\begin{algorithmic}[1]
\Require{Process $Z=Z^+-Z^-$ with non-positive drift, barrier level $c\in(0,\infty)$, time horizon $T\in(0,\infty]$, such that either $\limsup_{t\to\infty}Z_t=\infty$ a.s. or $T<\infty$.}
\State{Set $t\gets0$, $h\gets0$, $v\gets 0$ and $b\gets c$}
\Repeat
    \State{Sample $(s,u,v)\sim\mathcal{L}(\tau_b^+,Z^+_{\tau_b^+-},Z^+_{\tau_b^+})$ 
    and $w\sim\mathcal{L}(Z^-_s)$
    }
    \If{$s+t\ges T$}
        \State{Sample $u\sim\mathcal{L}(Z^+_{T-t})$ until $u<b$ and sample $w\sim\mathcal{L}(Z^-_{T-t})$}
        \State{\Return $(T,h+u-w,h+u-w)$}
    \Else
        \State{Set $t\gets t+s$, $h\gets h+v-w$, $b\gets b-v+w$}
    \EndIf
\Until{$b<0$}
\State{\Return $(t,h+u-v,h)$}
\end{algorithmic}
\end{algorithm}

\subsubsection{Pricing barrier options.}

Let $Z$ be a L\'evy process of bounded variation started from $0$ and let $R_t:=R_0\me^{Z_t}$, $t\ges0$, model the risky asset for some initial value $R_0>0$. The simulation of the first-passage event can be used to obtain Monte Carlo estimators of the price of a barrier option. Indeed, fix $K<M$ and consider the payoff of the up-and-out barrier call option with payoff 
\begin{align*}
g(T,R)
&=\max\{R_T-K,0\}\1\bigg\{\sup_{t\in[0,T]}R_t<M\bigg\}
=\max\{R_0\me^{Z_T}-K,0\}\1\bigg\{\sup_{t\in[0,T]}Z_t<\log(M/R_0)\bigg\}\\
&=\max\{R_0\me^{Z_T}-K,0\}\1\{\tau_{\log(M/R_0)}>T\}
=\max\{R_0\me^{Z_{\tau_*}}-K,0\}\1\{\tau_*=T\},
\end{align*}
where $\tau_*\coloneqq \min\{\tau_{\log(M/R_0)}^Z,T\}$ and $\tau_x^Z=\inf\{t>0:Z_t>x\}$. Then the price of the option is given by $G(T,R_0)\coloneqq\me^{-\delta T}\E[g(T,R)]$, where $\delta\in\R$ is the discount rate.

Let $(\tau_*^i,Z_{\tau_*^i}^i)$, $i\in\N$, be iid samples with the law of $(\tau_*,Z_{\tau_*})$ generated via \hyperref[alg:FPE_of_bounded_variation]{BVFP-Alg}. Then, for every $n$, we consider the Monte Carlo estimator
\[
G_n(T,R_0)\coloneqq
    \frac{1}{n}\sum_{k=1}^n 
        \max\{R_0\me^{Z_{\tau_*}^i}-K,0\}\1\{\tau_*^i=T\},
\]
of $G(T,R_0)$. Figure~\ref{fig:barrier_option} shows such an implementation for the case where $Z$ is the difference of two tempered stable subordinators. To obtain smoother graphs, we apply \hyperref[alg:FPE_of_bounded_variation]{BVFP-Alg} to successive values of the barrier level $\log(M/R_0)$ (i.e., by decreasing values of $R_0\in\{98,98.031,\ldots,101.999\}$), obtaining strongly correlated samples. This is a well-known procedure in Monte Carlo estimation that retains the accuracy of each estimate as if we had used independent samples for each value of $R_0$, but reduces the random fluctuations between any two estimated values, see~\cite[\S4.2]{MR1999614}.

\begin{figure}[ht]
\centering
\includegraphics[width=.6\textwidth]{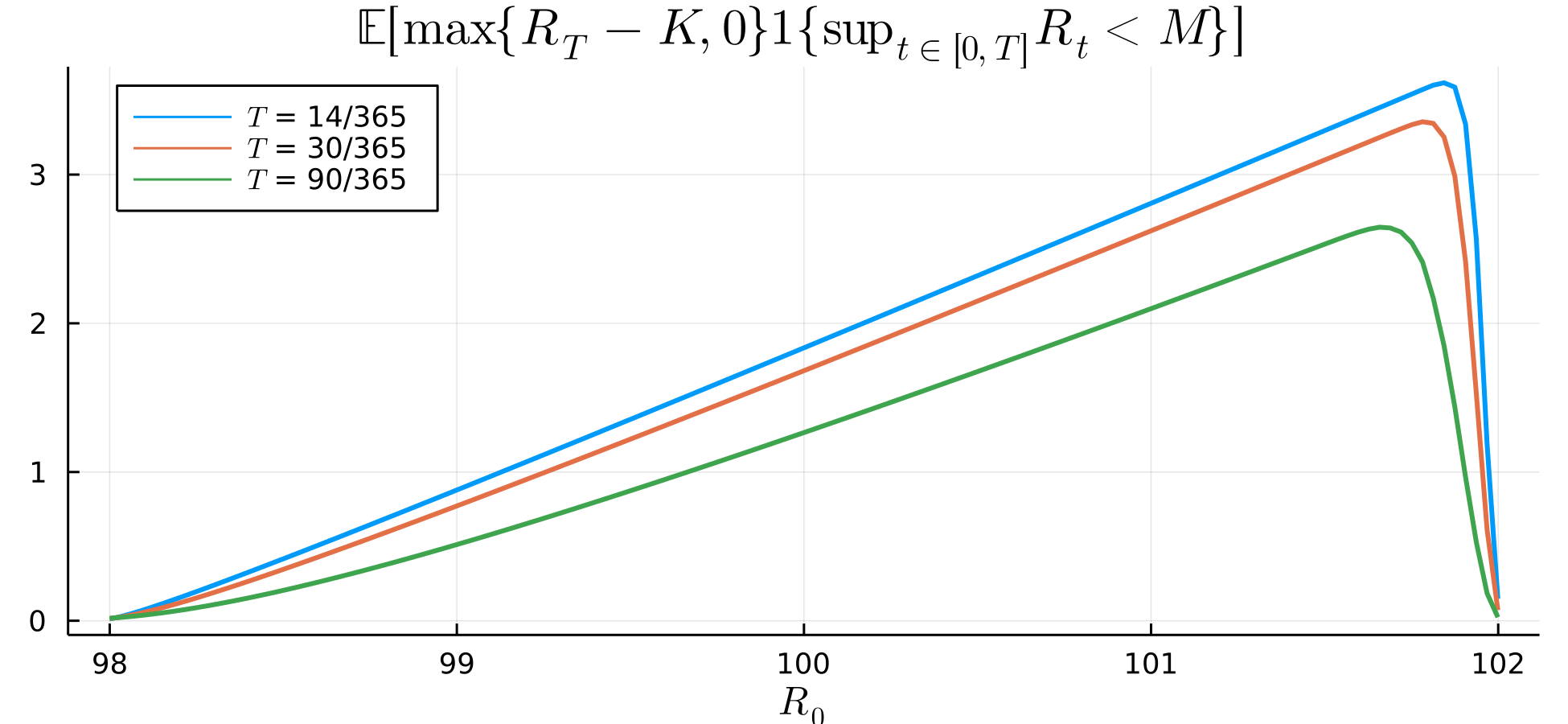}
\caption{Estimated values of $\me^{-\delta T}\E[\max\{R_T-K,0\}\1\{\sup_{t\in[0,T]}R_t<M\}]$ using $n=10^4$ samples.
\label{fig:barrier_option}
We considered the time horizons $T\in\{14/365, 30/365, 90/365\}$ and parameters $\delta=0$, $K=98$ and $M=102$. We assume $Z=Z^+-Z^-$ where $Z^+$ and $Z^-$ are both driftless tempered stable subordinators with parameters $(\alpha^+,\theta^+,q^+)=(0.66,0.1305,6.5022)$ and $(\alpha^-,\theta^-,q^-)=(0.66,0.0615,3.3088)$, respectively, as calibrated from the USD/JPY currency pair, see~\cite[Table 3]{MR3038608}.
}
\end{figure}

\subsection{Solution of fractional partial differential equations}

Let $g:\R_+\times \R^d\to  \R$ be a continuous function that vanishes at infinity and $\phi:\R^d\to\R$.
For $a<b$, consider the following fractional partial differential equation (FPDE):
\begin{equation}
\label{eq:fpde}
\begin{split}
    (-_tD_a+A_x)u(t,x)&=-g(t,x) \quad (t,x)\in(a,b]\times\mathbb R^d,\\
    u(a,x)&=\phi(x), \quad x\in \mathbb R^d,
\end{split}
\end{equation}
where $A_x$ is a generator of a Feller semigroup on $C_\infty(\mathbb R^d)$ acting on $x$, $\phi\in\text{Dom}(A_x)$, the operator $-_tD_a$ is a generalised differential operator of Caputo type of order less than $1$ acting on the time variable $t\in[a,b]$.
The solution $u$ of the FPDE problem in~\eqref{eq:fpde} exists and its stochastic representation is given by (see~\cite{hernandez2017generalised}):
\begin{equation*}
u(t,x)=\E\left[\phi(X_{T_t}^x)+\int_0^{T_t}g(Y_s^{a,t},X_s^x)\md s \right]
\end{equation*}
where $\{X_s^x\}_{s\ges0}$ is the stochastic process generated by $A_x$ and $T_t=\inf\{s>0,Y_s^{a,t}<a\}$ where $\{Y_s^{a,t}\}_{s\ges0}$ is the decreasing $[a,b]$-valued  process started at $t\in[a,b]$ generated by $-_tD_a$. 
Note that $T_t$
is the first-passage time of the subordinator 
$(t-Y_s^{a,t})_{s\in\R_+}$
over the constant level $t-a$.

In the special case $g\equiv 0$, $\phi(x)=x^2$, $A_x= (x^2/2)\Delta$ (where $\Delta$ is the Laplacian on $\R$) and the subordinator is tempered stable, we can simulate the Monte Carlo estimator 
\begin{equation}
\label{eq:mc_estimator}
    u_n(t,x)=\frac{1}{n}\sum_{k=1}^n\phi(X_{T_t^k}^{x,k})
\end{equation}
of the solution $u$ of the FPDE in~\eqref{eq:fpde}.
In particular, the variables
$T_t^k$ are an iid samples of $T_t$, obtained by \hyperref[alg:improved_triple_temper_stable]{TSFFP-Alg}, while 
$$X^{x,k}_{s}:=x\exp\left(\sqrt{s} N^k-s/2\right),$$ where 
$N^k$ are iid standard normal random variables.
Since $X^x$ is a geometric Brownian motion and \hyperref[alg:improved_triple_temper_stable]{TSFFP-Alg} is exact, the variables $\phi(X_{T_t^k}^{x,k})$ in the estimator in~\eqref{eq:mc_estimator} have the same law as $\phi(X^x_{T_t})$.
A plot of $u_n(t,x)$, based on $n=10^4$ samples, is presented in Figure~\ref{fig:fpde}.
In this case, the dependence of the estimator $u_n(t,x)$
for fixed $t$ (as a function of $x$) is explicit. The section of $u_n(t,x)$ as a function of $t$ (for fixed $x$)
is obtained by sampling consecutive crossing times for increasing barriers. This is a well-know procedure in Monte Carlo estimation that retains the accuracy of each estimate $u_n(t,x)$ but reduces random fluctuations between the values of $u_n$ at consecutive time points and a given level $x$
by introducing correlation~\cite[Sec.4.2]{MR1999614}.
The spacing between values of $t$ (resp. $x$) where $u_n(t,x)$ is evaluated in Figure~\ref{fig:fpde} is $1/20$ (resp. $1/100$).

\subsubsection{Comparison with a biased approximation.}
\label{subsec:naiveMC}

We stress that being able to simulate exactly from the law of $T_t$ is essential for the stability of  the estimator $u_n(t,x)$.
It is tempting to simulate a random walk approximation (i.e. the skeleton) of the subordinator 
$(t-Y_s^{a,t})_{s\in\R_+}$
on a time grid with mesh $h>0$ and approximate the law of $T_t$ with the first time 
$T_t^{(h)}$
the random walk crosses level $t-a$.
Unlike $u_n(t,x)$, the estimator 
\begin{equation*}
u_n^{(h)}(t,x)=\frac{1}{n}\sum_{k=1}^n\phi\Big(X_{T_t^{(h),k}}^{x,k}\Big),
\quad (t,x)\in\R_+^2,
\end{equation*}
is biased. Thus, in order for the $L^2$-error to be at most $\epsilon^2>0$, its bias and variance need to be of order $\epsilon$ and $\epsilon^2$, respectively. In the specific case considered here, it is easy to see that, for fixed $(t,x)$, the bias of $u_n^{(h)}(t,x)$
is proportional to $h$. In particular, the computational complexity of evaluating $\phi(X_{T_t^{(h),k}}^{x,k})$
is proportional to $1/h$, making the total complexity of 
$u_n^{(h)}(t,x)$ proportional 
to $n/h\propto\epsilon^{-3}$. In contrast, the computational complexity of evaluating $u_n(t,x)$ based on our exact simulation algorithm is proportional to $n\propto\epsilon^{-2}$. This yields an improvement in computational complexity proportional to $1/h\propto \epsilon^{-1}$.

The analysis outlined in the previous paragraph can be made precise for all Lipschitz functions $\phi$. We note that, if $\phi$ is only locally Lipschitz (as is the case in the example above), the error of the naive random walk algorithm may deteriorate significantly as a function of $(t,x)$. For instance, in the example above, the bias is proportional to $hx^2\E_q[\exp(\sqrt{T_t}N^1-T_t/2)]=hx^2\E_q[\exp(T_t)]$, where $\E_q[\exp(T_t)]\ges\E[\exp(T_t)]$ (because stable paths dominate tempered stable paths) and, by Proposition~\ref{prop:joint_law}(a) and~\eqref{eq:mellin_tranform},
\[
\E\big[\me^{T_t}\big]
=\E\big[\me^{t^\alpha S_1^{-\alpha}}\big]
=\sum_{k=0}^\infty
    \frac{t^{k\alpha}\E[S_1^{-k\alpha}]}{k!}
=\sum_{k=0}^\infty
    \frac{t^{k\alpha}}{\Gamma(1+k\alpha)}
\ges\frac{1}{\Gamma(1+\alpha)}+\sum_{k=1}^\infty
    \frac{t^{k\alpha}}{k!}
    =\me^{t^\alpha}-1+\frac{1}{\Gamma(1+\alpha)}.
\]
Thus, the bias of the naive approximation $u_n^{(h)}(t,x)$ is proportional to $hx^2\me^{t^\alpha}$. This would require $h$ to be proportional to $\epsilon x^{-2}\me^{-t^\alpha}$ if the $L^2$-error is to be smaller than $\epsilon^2$.
Thus, the improvement in computational complexity of $u_n(t,x)$ over $u_n^{(h)}(t,x)$ is proportional to $1/h\propto \epsilon^{-1}x^2\me^{t^\alpha}$.

\begin{figure}[ht]
\centering
\includegraphics[width=.49\textwidth]{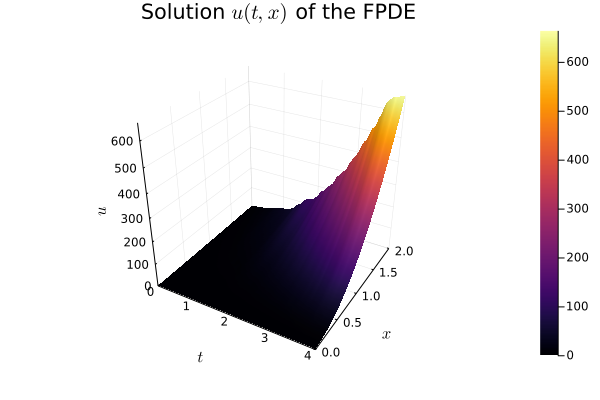}
\caption{Estimated values $u_n(t,x)$ of $u(t,x)$ using $n=10^4$ samples.\label{fig:fpde}
We chose $\phi(x)=x^2$. Note $T_t$ has the law of the first-passage time of a tempered stable subordinator with parameters $\alpha=0.4$, $\theta=1$, $q=1$ crossing level $b(t)\equiv 1$ and $X=(X_t)_{t\in\RP}$ is geometric Brownian motion $X_t=x\exp(B_t-t/2)$ with $X_0=x>0$.
}
\end{figure}

\section{Analyses of the algorithms}
\label{sec:proofs}

For any $\theta>0$ and $\alpha\in(0,1)$, let $S=(S_t)_{t\in\RP}$ (where $\RP:=[0,\infty)$) be the stable subordinator (for background on L\'evy, including stable, processes see  monograph~\cite{MR1406564}) with Laplace transform and L\'evy measure 
\begin{equation}
\label{eq:stable_Levy_measure}
\E[\me^{-uS_t}]=\exp(-t\theta u^\alpha), \quad\text{$u\in\RP$},\quad\&\quad  \nu(\md x):=\theta w_\alpha^{-1}  x^{-\alpha-1}\md x,\quad x\in(0,\infty),
\end{equation}
respectively, 
where $w_\alpha:=\Gamma(1-\alpha)/\alpha>0$ ($\Gamma$ denotes the Gamma function).
The constant factor $1/w_\alpha$
in the L\'evy measure $\nu(\md x)$ is chosen so that the L\'evy-Khinchin formula for the Laplace exponent yields the desired Laplace exponent $u\mapsto \theta u^\alpha$.
Let
$\overline\nu :(0,\infty)\to(0,\infty)$ be the tail function of the L\'evy measure $\nu$: $$\ov \nu(x):=\nu((x,\infty))=x^{-\alpha}\theta/\Gamma(1-\alpha),\qquad x\in(0,\infty).$$  

In particular, $S$ starts at zero, has zero drift and satisfies $S_t\eqd t^{1/\alpha}S_1$ for all $t\in(0,\infty)$.
Denote by $g_t:\RP\to\RP$ 
the density of $S_t$ (for $t>0$).
Introduce 
the functions 
\begin{equation}
\label{eq:zolotarev_density}
\rho(u)
\coloneqq
\frac{\sin(\alpha\pi u)^{\alpha}\sin((1-\alpha)\pi u)^{1-\alpha}}{\sin(\pi u)}
\quad\text{and}\quad
\sigma_\alpha(u)
\coloneqq 
\rho(u)^{r+1}=\rho(u)^{1/(1-\alpha)},\quad u\in(0,1),
\end{equation}
where $r:=\alpha/(1-\alpha)\in(0,\infty)$. Note $r/\alpha=r+1$ and hence $\sigma_\alpha(0+)\coloneqq\lim_{x\downarrow0}\sigma_\alpha(x)=(1-\alpha)\alpha^{\alpha/(1-\alpha)}$.  The density $g_t$ of $S_t$ can be expressed as follows~\cite[\S4.4]{uchaikin2011chance}: for any $x>0$,
\begin{equation}
\label{eq:zolotarev_representation}
    g_t(x)=\varphi_\alpha((\theta t)^{-1/\alpha}x)(\theta t)^{-1/\alpha},\quad\text{where}\quad
\varphi_\alpha(x)\coloneqq r
    \int_0^1 \sigma_\alpha(u)x^{-r-1}\me^{-\sigma_\alpha(u)x^{-r}}\md u.
\end{equation}
In particular, the $n$-th derivative $g_t^{(n)}$ exists on $(0,\infty)$ for any $n\in\N\cup\{0\}$ and satisfies $\lim_{x\downarrow0}g_t^{(n)}(x)=0$.

\subsection{Distribution of the first-passage event of a stable subordinator}
\label{subsec:joint_law_description}
The paths of $S$ are assumed to be right-continuous with left limits,
where for $t>0$ we define $S_{t-}:=\lim_{s\uparrow t}S_s$.
The jump at time $t$ is given by $\Delta_S(t):=S_t-S_{t-}$.
The first-passage event for $S$ over 
a non-increasing, absolutely continuous function 
$b:[0,\infty)\to[0,\infty)$, 
with $b(0)\in(0,\infty)$,
is described by the  random vector 
\begin{equation}
\label{eq:first_passage_verctor_stable}
\big(\tau_b,
    S_{\tau_b-},
    S_{\tau_b}\big),
\qquad
\text{where}\quad 
\tau_b\coloneqq\inf\{t>0\,:\,S_t>b(t)\}
\end{equation}
is the crossing time of $S$ over the function $b$ and $S_{\tau_b-}=\lim_{s\uparrow \tau_b}S_s$ is the left limit of $S$ at the crossing time of $b$. The size $\Delta_S(\tau_b):=S_{\tau_b}-S_{\tau_b-}$ of the jump of $S$ at the crossing time $\tau_b$ may be zero with positive probability (i.e. it does not jump at time $\tau_b$) if $b$ is not constant. In this case we say that the subordinator $S$ \textit{creeps} over $b$.

Theorem~\ref{thm:joint_law}, stated and proved in Appendix~\ref{app:proof_of_Thm_Joint_law} below for completeness but originally established in~\cite{chi2016exact} (see also~\cite{https://doi.org/10.48550/arxiv.2205.06865}), describes the law of the triplet $(\tau_b,S_{\tau_b-},\Delta_S(\tau_b))$ (and hence of~\eqref{eq:first_passage_verctor_stable}, since $S_{\tau_b}=S_{\tau_b-}+\Delta_S(\tau_b)$) for a general class of subordinators. In particular, as a consequence of~\eqref{subeq:joint_undershoot_law}, the stable subordinator $S$ does not jump onto or out of the function $b$ at $\tau_b$: 
\[
\p(S_{\tau_b-}<b(\tau_b), S_{\tau_b}=b(\tau_b))
=0\qquad\&\qquad \p(S_{\tau_b-}=b(\tau_b), S_{\tau_b}>b(\tau_b))=0.
\]

Note that, since $b$ is monotone and absolutely continuous, it is differentiable Lebesgue-a.e. on $(0,\infty)$ and its derivative is a version of its density. We denote its derivative by $b'$ and, on the Lebesgue zero measure set where it is not defined, we set it equal to $-1$ (arbitrary choice, without ramifications). The next proposition is key for the sampling of $(\tau_b,S_{\tau_b-},S_{\tau_b})$.

\begin{prop}\label{prop:joint_law}
Let $S$ be a stable subordinator and $b:\RP\to\RP$ a non-increasing absolutely continuous function with $b(0)>0$. Define $T_b\coloneqq\inf\{t>0:b(t)=0\}\in(0,\infty]$ (with convention $\inf\emptyset :=\infty$). The following statements hold.
\begin{enumerate}
\item[{\normalfont(a)}]
 $\tau_b\overset{d}=B^{-1}(S_1)$, where $B(t):=t^{-1/\alpha}b(t)$ is a strictly decreasing continuous function $B:(0,T_b)\to(0,\infty)$, with inverse $B^{-1}:(0,\infty)\to(0,T_b)$. Hence, $\p(0<\tau_b<T_b)=1$.
\item[{\normalfont(b)}]The probability of $S$ creeping at $b$, conditional on at time $\tau_b=t\in(0,T_b)$, equals 
$$
\p[S_{\tau_b}=b(\tau_b)|\tau_b=t]=\p[\Delta_S(\tau_b)=0|\tau_b=t]=-b'(t)/(-b'(t)+\alpha^{-1}t^{-1}b(t)).
$$
\item[{\normalfont(c)}] Conditional on  $S$  not creeping 
at the crossing time $\tau_b=t\in(0,T_b)$, we have
\begin{align}
    \label{eq:undershoot_law}
        &\p[S_{\tau_b-}\in\md u|\Delta_S(\tau_b)>0 ,\tau_b=t] =\frac{\theta t}{w_\alpha b(t)g_t(b(t))}(b(t)-u)^{-\alpha}g_t(u)\md u,\quad u\in(0,b(t));\\
\label{eq:creep_law}
    &\p[\Delta_S(\tau_b)\leq v| \Delta_S(\tau_b)>0, \tau_b=t, S_{\tau_b-}=u] =1-\left((b(t)-u)/ v\right)^{\alpha},\quad v\in(b(t)-u,\infty).
\end{align}
\end{enumerate}
\end{prop}

\begin{remark}
\begin{enumerate}
\item[(i)] The representation of the law of $\tau_b$ in Proposition~\ref{prop:joint_law}(a) is based on the scaling property of the stable subordinator $S$. In particular, Proposition~\ref{prop:joint_law}(a) implies that, at the crossing time $\tau_b$ of a stable subordinator  over a boundary function $b$, the law of the variable $\tau_b^{-1/\alpha}b(\tau_b)$ is stable for \text{any} non-increasing absolutely continuous function $b$.
\item[(ii)] The expression of the probability of creeping, conditional on $\tau_b=t$, in Proposition~\ref{prop:joint_law}(b) does not depend on the density of $S_t$. This makes the event of creeping easy to sample, see \hyperref[alg:triple_stable_conditional_on_time]{SFP-Alg} below.
\item[(iii)] Conditional on not creeping, the jump $\Delta_S(\tau_b)$ is easy to sample by inverting its distribution function in~\eqref{eq:creep_law}. In contrast, sampling the undershoot from the law in~\eqref{eq:undershoot_law} in finite expected running time is delicate, see \hyperref[alg:undershoot_stable]{SU-Alg} below.
\end{enumerate}
\end{remark}

\begin{proof}[Proof of Proposition~\ref{prop:joint_law}]
(a) By the scaling property of $S_t\eqd t^{1/\alpha} S_1$, note $\p[\tau_b\les t]
=\p[S_t\ges b(t)]
=\p[S_1\ges B(t)]
=\p[B^{-1}(S_1)\les t]$
for any $t>0$.\\
(b) We claim that $\tau_b$ is absolutely continuous with density given by the function $K:(0,T_b)\to (0,\infty)$, 
\begin{equation}\label{eq:density_tau}
K(t)
= g_t(b(t))\left(\alpha^{-1}t^{-1}b(t)-b'(t)\right)\quad\text{ for all $t\in(0,T_b)$.}
\end{equation}
In fact, $\p[\tau_b\les t]
=\p[B^{-1}(S_1)\les t]=\p[S_1\ges B(t)]=\int_{B(t)}^\infty g_1(s)\md s$, which is absolutely continuous since $b$ (and hence $B$) is absolutely continuous. Differentiating this identity in $t$ and applying the scaling property to the density $g_t(s)=t^{-1/\alpha}g_1(t^{-1/\alpha}s)$ yields
\[
    \begin{aligned}
    K(t)&:=\frac{\md}{\md t}\p[\tau_b\les t]=\frac{\md}{\md t}\p[B^{-1}(S_1)\les t]=-B(t)'g_1(B(t))\\&=g_1(t^{-1/\alpha}b(t))(t^{-1/\alpha-1}b(t)/\alpha-t^{-1/\alpha}b'(t))
    =g_t(b(t))\left(\alpha^{-1}t^{-1}b(t)-b'(t)\right).
    \end{aligned}
\]
Since $\p[S_{\tau_b}=b(\tau_b)|\tau_b=t]=
\p[\Delta_S(\tau_b)=0|\tau_b=t]=
(1/K(t))\frac{\md}{\md t}\p[\Delta_S(\tau_b)=0, \tau_b\leq t]$,
the representation in~\eqref{eq:density_tau} and~\eqref{subeq:creep_b} imply $\p[S_{\tau_b}=b(\tau_b)|\tau_b=t]=-b'(t)/(-b'(t)+\alpha^{-1}t^{-1}b(t))$.\\
(c) By~\eqref{subeq:joint_undershoot_law} in Theorem~\ref{thm:joint_law}, for $t\in(0,T_b)$ and $u\in(0,b(t))$, we have 
$$
\p[\tau_b\in\md t, S_{\tau_b-}\in \md u, \Delta_S(\tau_b)>0]=\1\{0<b(t)-u\} g_t(u) \ov\nu(b(t)-u)\md t \, \md u.
$$
Thus, 
$\p[\tau_b\in\md t,\Delta_s(\tau_b)>0]=\int_0^{b(t)}\ov\nu(b(t)-u)g_t(u)\md u \md t$ and
\[
\p[S_{\tau_b-}\in\md u|\Delta_s(\tau_b)>0,\tau_b=t]=\1\{0<b(t)-u\}
\frac{ g_t(u) \ov\nu(b(t)-u)}{\int_0^{b(t)} g_t(u')\ov\nu(b(t)-u')\md u'} \md u.
\]
To evaluate the denominator, let us recall that $K(t)$ is the density of $\tau_b$. Together with~\eqref{subeq:creep_b} we have
\begin{equation*}
K(t)=\p[\tau_b\in\md t,\Delta_s(\tau_b)=0]+
\p[\tau_b\in\md t,\Delta_s(\tau_b)>0]
=-b'(t)g_t(b(t))+\int_0^{b(t)} g_t(u')\ov\nu(b(t)-u')\md u'.
\end{equation*}
Combining~\eqref{eq:density_tau} we get 
$ \int_0^{b(t)} g_t(u')\ov\nu(b(t)-u')\md u'=g_t(b(t))\alpha^{-1}t^{-1}b(t)$
and~\eqref{eq:undershoot_law} follows. Also from~\eqref{subeq:joint_undershoot_law} we know that
\begin{align*}
    \p[\Delta_S(\tau_b)\in\md v| \Delta_S(\tau_b)>0, \tau_b=t, S_{\tau_b-}=u]&=\frac{\p[S_{\tau_b-}\in\md u, \Delta_S(\tau_b)\in\md v| \Delta_S(\tau_b)>0, \tau_b=t]}{\p[S_{\tau_b-}\in\md u| \Delta_S(\tau_b)>0, \tau_b=t]}\\
    &=\alpha(b(t)-u)^\alpha v^{-\alpha-1}\md v\quad\text{ for } v\in(b(t)-u,\infty)
\end{align*}
and~\eqref{eq:creep_law} follows.
\end{proof}

\subsection{Simulation of the first-passage event of a stable process}
\label{subsec:triple_stable}

The task is to justify the validity of \hyperref[alg:triple_stable_conditional_on_time]{SFP-Alg}, which samples the triplet $(\tau_b,S_{\tau_b-},S_{\tau_b})$ of the stable subordinator $S$ under $\p$. The analysis of its expected computational complexity has some technical prerequisites and will thus be deferred to Subsection~\ref{subsubsec:proof_of_main_thm_stable}. 

By Proposition~\ref{prop:joint_law}(a), $\tau_b$ has the same law as $B^{-1}(S_1)$ and can be sampled exactly in constant time (i.e. \textit{not} requiring a random number of steps) via the widely used \hyperref[alg:stable]{S-Alg} given in Appendix~\ref{app:temp_stable_marginal} below. Moreover, by Proposition~\ref{prop:joint_law}(b), the probability of the stable subordinator $S$ creeping over the function $b$ (i.e. $\Delta_S(\tau_b)=0$), given $\tau_b=t$, equals $-b'(t)/\big(-b'(t)+\alpha^{-1}t^{-1}b(t)\big)$. \hyperref[alg:triple_stable_conditional_on_time]{SFP-Alg} samples a Bernoulli variable taking the value $1$ with probability $-b'(t)/\big(-b'(t)+\alpha^{-1}t^{-1}b(t)\big)$. If the Bernoulli equals $1$, then, since the probability of $S$ jumping onto the function $b$ is zero $\p(S_{\tau_b-}<b(\tau_b), S_{\tau_b}=b(\tau_b))=0$, \hyperref[alg:triple_stable_conditional_on_time]{SFP-Alg} returns $(\tau_b,S_{\tau_b-},S_{\tau_b})=\big(t,b(t),b(t)\big)$. If, conditional on $\tau_b=t$, creeping does not occur, \hyperref[alg:triple_stable_conditional_on_time]{SFP-Alg} samples the undershoot $S_{\tau_b-}$ from the law $\p[S_{\tau_b-}\in\md u|\Delta_S(\tau_b)>0,\tau_b=t]$, given in~\eqref{eq:undershoot_law} of Proposition~\ref{prop:joint_law}(c), using \hyperref[alg:undershoot_stable]{SU-Alg} given below. Then, conditional on $S_{\tau_b-}=u$, the jump size $\Delta_S(t)$ has  by~\eqref{eq:creep_law} the same law as $(b(t)-u)U^{-1/\alpha}$, where $U$ is an independent uniform on $(0,1)$. Thus, \hyperref[alg:undershoot_stable]{SU-Alg} returns $(\tau_b,S_{\tau_b-},S_{\tau_b})=\big(t,u,u+(b(t)-u)U^{-1/\alpha}\big)$.


By the previous paragraph, for \hyperref[alg:triple_stable_conditional_on_time]{SFP-Alg} to be valid, it suffices to justify that \hyperref[alg:undershoot_stable]{SU-Alg} produces samples from the law $\rmS\Unif_\alpha(t,w)$ of $S_{\tau_b-}|\{\tau_b=t, b(\tau_b)=w, \Delta_S(\tau_b)>0\}$, which is described analytically in~\eqref{eq:undershoot_law} of Proposition~\ref{prop:joint_law}(c). Note that the simulation from the law $\rmS\Unif_\alpha(t,w)$ via \hyperref[alg:undershoot_stable]{SU-Alg} is the only step in \hyperref[alg:triple_stable_conditional_on_time]{SFP-Alg} that has random running time. In fact, the design of \hyperref[alg:undershoot_stable]{SU-Alg} is one of the main contributions of this paper and will be discussed in detail in the following subsection.

\subsection{Simulation from the law of the stable undershoot 
\texorpdfstring{$S_{\tau_b-}|\{\tau_b=t, b(\tau_b)=w, \Delta_S(\tau_b)>0\}$}{undershoot}}
\label{subsec:Undershoot_simulation}

We now describe and justify \hyperref[alg:undershoot_stable]{SU-Alg} and analyse its computational complexity. Recall from Subsection~\ref{subsec:joint_law_description} above that the function $b:[0,\infty)\to[0,\infty)$ is absolutely continuous and non-decreasing and hence almost everywhere differentiable with derivative $b'$ almost everywhere equal to the density of $b$ and set to be equal to $-1$ at the points of non-differentiability of $b$. Further recall that $r=\alpha/(1-\alpha)$, $w_\alpha=\Gamma(1-\alpha)/\alpha>0$ and the definition $T_b>0$, as well as that of the functions $\sigma_\alpha$ and $\varphi_\alpha$ defined in~\eqref{eq:zolotarev_density} and~\eqref{eq:zolotarev_representation}, respectively. 

By~\eqref{eq:undershoot_law} in Proposition~\ref{prop:joint_law}(c) and the representation of the stable density $g_t(x)$ in~\eqref{eq:zolotarev_representation}, the density of $\rmS\Unif_{\alpha}(t,w)$ is given by 
\[
f_{t,w}(x)
\coloneqq \frac{1}{w_\alpha w\varphi_\alpha((\theta t)^{-1/\alpha}w)}
\varphi_\alpha((\theta t)^{-1/\alpha}x)((\theta t)^{-1/\alpha}w-(\theta t)^{-1/\alpha}x)^{-\alpha}, \qquad x\in(0,w),
\] 
while the density of  $\xi\sim\rmS\Unif_\alpha(1/\theta,(\theta t)^{-1/\alpha}w)$ equals
\[
    f_{1/\theta,(\theta t)^{-1/\alpha}w}(x)= \frac{1}{w_\alpha w\varphi_\alpha((\theta t)^{-1/\alpha}w)}(\theta t)^{1/\alpha}\varphi_\alpha(x)((\theta t)^{-1/\alpha}w-x)^{-\alpha} , \qquad x\in(0,(\theta t)^{-1/\alpha}w).
\]
Clearly, 
$(\theta t)^{-1/\alpha}f_{1/\theta,(\theta t)^{-1/\alpha}w}((\theta t)^{-1/\alpha}x)=f_{t,w}(x)$ for $x\in(0,w)$. Thus, 
$(\theta t)^{1/\alpha}\xi\sim \rmS\Unif_{\alpha}(t,w)$. It suffices to sample $\xi$ from the law $\rmS\Unif_\alpha(1/\theta,s)$ with density 
$x\mapsto \varphi_\alpha(x)(s-x)^{-\alpha}/(w_\alpha s\varphi_\alpha(s))$, $x\in(0,s)$ and $s=(\theta t)^{-1/\alpha}w$.
Sampling efficiently from a density on $(0,s)$, proportional to
$x\mapsto \varphi_\alpha(x)(s-x)^{-\alpha}$, is hard to do directly, see Section~\ref{subsec:How_not_to_sample}. We use the integral representation of $\varphi_\alpha(x)$ in~\eqref{eq:zolotarev_representation} to extend the state space: consider $(\xi,Y_s)$ with density
\[
    (x,y)\mapsto (s-x)^{-\alpha}\sigma_\alpha(y)rx^{-r-1}\me^{-\sigma_\alpha(y)x^{-r}}/(w_\alpha s\varphi_\alpha(s)),\quad (x,y)\in(0,s)\times(0,1).
\]
Define a random variable $\zeta_s:=\xi^{-r}$. Note that $\zeta_s^{-1/r}\sim\rmS\Unif_\alpha(1/\theta,s)$,
and hence $(\theta t)^{1/\alpha}\zeta_s^{-1/r}\sim\rmS\Unif_\alpha(t,w)$.
Moreover, 
the law of the random vector $(\zeta_s,Y_s)$
is given by the density
 \begin{equation}\label{eq:scaled_law}
     \psi_s(x,y):= (s-x^{-1/r})^{-\alpha}\sigma_\alpha(y)\me^{-\sigma_\alpha(y)x}/(w_\alpha s\varphi_\alpha(s)),\quad (x,y)\in(s^{-r},\infty)\times(0,1).
 \end{equation}
 Fast exact simulation from the law given by the density $\psi_s$ is non-trivial but possible, see \hyperref[alg:undershoot_stable]{SU-Alg}.
 The first step, given in Proposition~\ref{prop:undershoot}, is to dominate the density $\psi_s$ by a mixture of densities that can both be simulated using rejection sampling.

\begin{prop}\label{prop:undershoot}
Fix any $s>0$. Let $\wt\psi_s^{(1)}$ (resp. $\wt\psi_s^{(2)}$) be the density on $(s^{-r},\infty)\times(0,1)$ proportional to  $(x,y)\mapsto\sigma_\alpha(y)\me^{-\sigma_\alpha(y)x}$ (resp. $(x,y)\mapsto\sigma_\alpha(y)\me^{-\sigma_\alpha(y)x}(x-s^{-r})^{-\alpha}$). Define the mixture density 
\[
\wt\psi_s:=p\wt\psi_s^{(1)}+(1-p)\wt\psi_s^{(2)},
\]
where
$p:=p'/(p'+1)\in(0,1)$ and 
\begin{equation}
\label{eq:psi_p'}
p':=(2-2^{\alpha})^{-\alpha}r^{-\alpha}s^{\alpha r}
\frac{\int_0^1 \exp(-\sigma_\alpha(y)s^{-r})\md y}{\Gamma(1-\alpha)\int_0^1\sigma_\alpha(y)^\alpha\exp(-\sigma_\alpha(y)s^{-r})\md y}.
\end{equation}
Then there exists a constant 
$c_\psi>0$ 
such that, for all $(x,y)\in(s^{-r},\infty)\times (0,1)$, we have
\begin{equation}\label{eq:inequality_psi_tilde}
\frac{1-\alpha}{2}
\les \frac{c_\psi\psi_s(x,y)}{\wt\psi_s(x,y)}
=\frac{(s-x^{-1/r})^{-\alpha}}{(1-2^{\alpha-1})^{-\alpha} s^{-\alpha}
    +(2r)^{\alpha}s^{-r}(x-s^{-r})^{-\alpha}}
\les 1. 
\end{equation}
\end{prop}

\begin{remark}
The numerical value of the constant $c_\psi$ in Proposition~\ref{prop:undershoot} is not required in \hyperref[alg:undershoot_stable]{SU-Alg}, which samples from the density $\psi_s$ in~\eqref{eq:scaled_law}.
The inequalities in~\eqref{eq:inequality_psi_tilde}
enable us to apply rejection sampling to a proposal from $\wt\psi_s$ to obtain an exact sample from the law given by the density $\psi_s$. The first inequality in~\eqref{eq:inequality_psi_tilde} explicitly states a decay of the acceptance probability as $\alpha\uparrow1$. Sampling from $\wt\psi_s$ requires a mixture of samples from $\wt\psi_s^{(1)}$ and $\wt\psi_s^{(2)}$. The respective algorithms are constructed and analysed in Subsections~\ref{subsec:psi_(1)-2} and~\ref{subsec:psi_(2)-2} below, respectively.
\end{remark}

\begin{proof}[Proof of Proposition~\ref{prop:undershoot}]
Denote the total masses of the positive functions 
$(x,y)\mapsto \sigma_\alpha(y)\me^{-\sigma_\alpha(y)x}$ and $(x,y)\mapsto\sigma_\alpha(y)\me^{-\sigma_\alpha(y)x}(x-s^{-r})^{-\alpha}$ by
\begin{align*}
    M_1:=&\int_0^1\int_{s^{-r}}^\infty \sigma_\alpha(y)\me^{-\sigma_\alpha(y)x}\md x\md y=\int_0^1 \exp(-\sigma_\alpha(y)s^{-r})\md y\qquad\text{ and}\\
    M_2:=&\int_0^1\int_{s^{-r}}^\infty \sigma_\alpha(y)\me^{-\sigma_\alpha(y)x}(x-s^{-r})^{-\alpha}\md x\md y=\Gamma(1-\alpha)\int_0^1\sigma_\alpha(y)^\alpha\exp(-\sigma_\alpha(y)s^{-r})\md y.
\end{align*} 
Then, $\wt\psi_s^{(1)}(x,y)=\sigma_\alpha(y)\me^{-\sigma_\alpha(y)x}/M_1$ and $\wt\psi_s^{(2)}(x,y)=\sigma_\alpha(y)\me^{-\sigma_\alpha(y)x}(x-s^{-r})^{-\alpha}/M_2$.
Moreover,
\[
p'
=\frac{(1-2^{\alpha-1})^{-\alpha} s^{-\alpha}M_1}{(2r)^\alpha s^{-r}M_2}.
\]
Note that $p'=p/(1-p)$, making  
$\wt\psi_s$ proportional to 
\begin{align*}
 (2r)^\alpha s^{-r}M_2(p'\wt\psi_s^{(1)}(x,y)+\wt\psi_s^{(2)}(x,y)) & =  (1-2^{\alpha-1})^{-\alpha} s^{-\alpha}M_1\wt\psi_s^{(1)}(x,y)+(2r)^\alpha s^{-r}M_2\wt\psi_s^{(2)}(x,y)\\
    &= \sigma_\alpha(y)\me^{-\sigma_\alpha(y)x}
    \big((1-2^{\alpha-1})^{-\alpha} s^{-\alpha}
    +(2r)^{\alpha}s^{-r}(x-s^{-r})^{-\alpha}\big).
\end{align*}
Let $c_\psi>0$ be the normalizing constant, such that for $(x,y)\in(s^{-r},\infty)\times(0,1)$,
\begin{equation}\label{eq:c_psi_def}
    \wt\psi_s(x,y)=
\frac{c_\psi}{w_\alpha s\varphi_\alpha(s)}
    \sigma_\alpha(y)\me^{-\sigma_\alpha(y)x}
    \big((1-2^{\alpha-1})^{-\alpha} s^{-\alpha}
    +(2r)^{\alpha}s^{-r}(x-s^{-r})^{-\alpha}\big).
\end{equation}
Thus $c_\psi\psi_s(x,y)/\wt\psi(x,y)=(s-x^{-1/r})^{-\alpha}/\big((1-2^{\alpha-1})^{-\alpha} s^{-\alpha}
    +(2r)^{\alpha}s^{-r}(x-s^{-r})^{-\alpha}\big)$
    in~\eqref{eq:inequality_psi_tilde} follows.
    
We now prove the inequalities in~\eqref{eq:inequality_psi_tilde}.
Since $x\mapsto(1-\alpha) s^{r+1}x + \alpha x^{-1/r}$ is increasing on $(s^{-r},\infty)$, we have $(1-\alpha) s^{r+1}x + \alpha x^{-1/r}\ges s$,  equivalently  $r^\alpha s^{-r}(x-s^{-r})^{-\alpha}\les (s-x^{-1/r})^{-\alpha}$, for all $x>s^{-r}$. This inequality and a direct calculation imply 
$((1-2^{\alpha-1})^{-\alpha}+2^\alpha)^{-1}\les c_\psi\psi_s(x,y)/\wt\psi_s(x,y)$ for all $(x,y)\in(s^{-r},\infty)\times(0,1)$.
Since $(1-\alpha)/2\leq ((1-2^{\alpha-1})^{-\alpha}+2^\alpha)^{-1}$ ,
the inequality on the left side of~\eqref{eq:inequality_psi_tilde} follows.

Since $(1-2^{\alpha-1})^{-\alpha} s^{-\alpha}\ges (s-x^{-1/r})^{-\alpha}$ for $x\ges 2^\alpha s^{-r}$
(note that $x\mapsto (s-x^{-1/r})^{-\alpha}$ is decreasing and equal to $(1-2^{\alpha-1})^{-\alpha} s^{-\alpha}$
at $x= 2^\alpha s^{-r}$),
the inequality on the right side of~\eqref{eq:inequality_psi_tilde} holds for  $x\ges 2^\alpha s^{-r}$. Consider the case $s^{-r}<x<2^\alpha s^{-r}$. Since $x\mapsto(1-\alpha)s^{r+1}x+2\alpha x^{-1/r}$ decreases on $[s^{-r},2^\alpha s^{-r}]$, after substituting $r=\alpha/(1-\alpha)$, we obtain $(1-\alpha)s^{1/(1-\alpha)}x+2\alpha x^{-(1-\alpha)/\alpha}\les (1+\alpha)s$. This inequality is equivalent to $(2\alpha/(1-\alpha))^\alpha s^{-\alpha/(1-\alpha)}(x-s^{-\alpha/(1-\alpha)})^{-\alpha}\ges (s-x^{-(1-\alpha)/\alpha})^{-\alpha}$. Substituting 
$r=\alpha/(1-\alpha)$, we obtain 
$(2r)^\alpha s^{-r}(x-s^{-r})^{-\alpha}\ges (s-x^{-1/r})^{-\alpha}$,
and the right side of~\eqref{eq:inequality_psi_tilde} follows.
\end{proof}

\subsubsection{Simulation from the density \texorpdfstring{$\wt\psi_s^{(1)}$}{psi 1}: proof of Proposition~\ref{prop:psi_1}}
\label{subsec:psi_(1)-2}

The proof of Proposition~\ref{prop:psi_1} requires the following two elementary lemmas about certain properties of the function $\sigma_\alpha$ defined in~\eqref{eq:zolotarev_density}.

 \begin{lem}
\label{lem:varphi_convex}
The derivative of any order of $\log\sigma_\alpha$, and hence of $\sigma_\alpha$, is positive on $(0,1)$. 
\end{lem}

\begin{proof}
Clearly, by definition~\eqref{eq:zolotarev_density},  $\log(\sigma_\alpha(x))
=(1-\alpha)^{-1}(\alpha\psi_\alpha(x)
+(1-\alpha)\psi_{1-\alpha}(x))$, where $\psi_a:x\mapsto\log(\sin(a\pi x))-\log(\sin(\pi x))$ for $a\in(0,1)$. Recall that $\cot(\pi x)=(\pi x)^{-1}-(2/\pi)\sum_{n=1}^\infty \zeta(2n)x^{2n-1}$ for $|x|\in(0,1)$ where $\zeta$ is the Riemann zeta function. Thus, for any $a\in(0,1)$, we have
\[
\psi'_a(x) 
= a\pi\cot(a\pi x)-\pi\cot(\pi x)
= 2\sum_{n=1}^\infty 
    \zeta(2n)(1 - a^{2n})x^{2n-1},
\quad x\in(0,1).
\]
This shows that $\psi'_a$ and all its derivatives are positive on $(0,1)$, implying the same is true of  $\log(\sigma_\alpha)$. Thus,
$\sigma_\alpha''=(\me^{\log(\sigma_\alpha)})''=\me^{\log(\sigma_\alpha)}(\log(\sigma_\alpha)'^2+\log(\sigma_\alpha)'')\geq0$, making $\sigma_\alpha$  convex on $(0,1)$. Similarly, a derivative of $\sigma_\alpha$ of any order is a linear combination with non-negative coefficients of derivatives of $\log\sigma_\alpha$, multiplied with $\me^{\log(\sigma_\alpha)}$, and thus non-negative itself on the interval $(0,1)$.
\end{proof}

\begin{lem}\label{lem:sigma_linear_upperbound}
For $x\in(0,1/2)$, we have $0\les\sigma_\alpha(x)-\sigma_\alpha(0+)\les(\pi-2\me^{-1})(1-\alpha)x$.
\end{lem}

\begin{proof}
By Lemma~\ref{lem:varphi_convex}, $\sigma_\alpha$ is convex and increasing. Thus
$$0\les(\sigma_\alpha(x)
- \sigma_\alpha(0+))/x\les (\sigma_\alpha(1/2)-\sigma_\alpha(0+))/(1/2),$$ implying
$\sigma_\alpha(x)
-\sigma_\alpha(0+) \les 2(\sigma_\alpha(1/2)-\sigma_\alpha(0+))x$.
For $t\in(0,\pi)$, we have $\sin (t)\les\min\{1,t\}$ and hence
\begin{equation}
\label{eq:sigma_at_half}
\sigma_\alpha(1/2)=\sin(\alpha\pi/2)^r\sin((1-\alpha)\pi/2)\les (1-\alpha)\pi/2.
\end{equation}
By $\log(1/\alpha)\les 1/\alpha-1$ we have $(1-\alpha)^{-1}\alpha\log\alpha\ges-1$. Thus
$\sigma_\alpha(0+)=\alpha^r(1-\alpha)\ges \me^{-1}(1-\alpha)$ and~\eqref{eq:sigma_at_half} imply
$\sigma_\alpha(x)
-\sigma_\alpha(0+)\leq (1-\alpha)\pi/2- \me^{-1}(1-\alpha)$, proving the lemma.
\end{proof}

\begin{proof}[Proof of Proposition~\ref{prop:psi_1}]
Recall that $\sigma_\alpha$ is convex by Lemma~\ref{lem:varphi_convex}. To sample from  $\wt\psi_s^{(1)}$, we first sample $Y_s$ from the marginal log-concave density proportional to $\xi(y):=\exp(-(\sigma_\alpha(y)-\sigma_\alpha(0+))s^{-r})$ on $(0,1)$ 
via \hyperref[alg:Devroye]{LC-Alg}. Then, conditional on $Y_s=y$, we sample $\zeta_s$ from the law whose density is proportional to the function $x\mapsto\sigma_\alpha(y)\exp(-\sigma_\alpha(y)x)\1_{\{x>s^{-r}\}}$ on $x\in(s^{-r},\infty)$. Note that this makes $\zeta_s$, conditional on $Y_s$, into a shifted exponential random variable. 

As explained in Section~\ref{sec:devroye} below, the computational cost of \hyperref[alg:Devroye]{LC-Alg} has two parts: (I) the cost of finding the value of $a_1$ in line~\ref{line_find_a_1} of \hyperref[alg:Devroye]{LC-Alg} and (II) the expected 
cost of the accept-reject step in \hyperref[alg:Devroye]{LC-Alg}.
By~\cite[\S4]{devroye2012note}, the expected cost of (II) is bounded above by $5$. 

Note that cost (I) of finding $a_1$ within \hyperref[alg:psi_1]{$\psi^{(1)}$-Alg} is not constant in the variable $s$,
which is itself random in \hyperref[alg:triple_stable_conditional_on_time]{SFP-Alg}. 
Thus we need to analyse the cost of finding $a_1$ in \hyperref[alg:Devroye]{LC-Alg} as a function of $s$.
By Lemma~\ref{lem:sigma_linear_upperbound}, for $y\in(0,1/2)$, we have $\xi(y)\ges\exp(-(\pi-2\me^{-1})(1-\alpha)ys^{-r})$.
Thus we set $y:=(\pi-2\me^{-1})^{-1}(\log4)(1-\alpha)^{-1}s^r$ and note that the following inequality holds: $\xi(y)\ges1/4$. 
Therefore the number of steps
in the binary search in Step~\ref{step_binary} of \hyperref[alg:Devroye]{LC-Alg} is bounded above by 
$$1+\log(1/y)\les 1+\log^+\big((\pi-2\me^{-1})(\log4)^{-1}(1-\alpha)s^{-r}\big)\les\kappa_{\hyperref[alg:psi_1]{(1)}}(1+\alpha(1-\alpha)^{-1}\log_2^+(s^{-1})),$$
where the constant in $\kappa_{\hyperref[alg:psi_1]{(1)}}$ depends neither on $s\in(0,\infty)$ nor $\alpha\in(0,1)$,
implying the claim.
\end{proof}

Before proving Proposition~\ref{prop:psi_2}, we first clarify in the next subsection the the time cost of applying the Newton--Raphson method in  lines~1, 8, 15 and 21 of \hyperref[alg:psi_2]{$\psi^{(2)}$-Alg}. As we shall see, the costs depend on $\alpha$ and may, but need not, depend on $s$.

\subsubsection{Time cost of applying the Newton--Raphson method in \hyperref[alg:psi_2]{$\psi^{(2)}$-Alg}.}
\label{subsec:NewtonRaphson-psi_2}

Inverting $\sigma_\alpha$ is equivalent to inverting the convex increasing function 
\[
x\mapsto\log\bigg(\frac{\sin(\alpha\pi x)^\alpha\sin((1-\alpha)\pi x)^{1-\alpha}}{\sin(\pi x)}\bigg),
\quad x\in(0,1).
\]
Equivalently, for any $s>\alpha\log\alpha+(1-\alpha)\log(1-\alpha)$, we must find the root of the convex increasing function $x\mapsto \varsigma(x)-s$ where
\begin{equation}
\label{eq:equiv_sigma}
\varsigma(x)
\coloneqq\alpha\log\sin(\alpha\pi x) + (1-\alpha)\log\sin((1-\alpha)\pi x)
-\log\sin(\pi x),
\quad x\in(0,1).
\end{equation}
\begin{prop}
\label{prop:inverse_sigma}
 Define the auxiliary function 
\begin{equation}
\label{eq:sigma_auxiliary_function}
M(x,k):=\frac{1}{2\pi\alpha(1-\alpha)(x-2^{1-k})x^2(1-x)^2}.
\end{equation}
For any $s>\alpha\log\alpha+(1-\alpha)\log(1-\alpha)$, the computational cost of finding the root of $\varsigma(x)$ via \hyperref[alg:inversion_newton_raphson]{NR-Alg} with auxiliary function $M(x,k)$ is
$\Oh(|\log(x_*(1-x_*))+\log(\alpha(1-\alpha))|+\log N)$ where $x_*=\varsigma^{-1}(s)$ is the solution and $N$ specifies the number of precision bits.
\end{prop}

\begin{proof}
Since $\varsigma(x)$ is convex, if our initial estimate $x_0$ (obtained via binary search) of the root is larger than the true root, then the Newton--Raphson sequence
\[
x_{n+1}=x_n-\frac{\varsigma(x_n)-s}{\varsigma'(x_n)},
\quad n\in\N,
\]
decreases and has quadratic convergence to the root $x_*$. Using the power series of the cotangent function (as in the proof of Lemma~\ref{lem:varphi_convex}) and the inequality $\sin(\pi x)\ges \pi x(1-x)$, for $(x_0,k)$ satisfying $x_*\les x_0-2^{1-k}$, we have
\begin{align*}
\frac{1}{2}\frac{\sup_{x\in[x_*,x_0]}|\varsigma''(x)|}{\inf_{x\in[x_*,x_0]}|\varsigma'(x)|}
&\les\frac{\pi}{2}
    \bigg|\frac{\csc^2(\pi x_0)-\alpha^3\csc^2(\alpha\pi x_0)-(1-\alpha)^3\csc^2((1-\alpha)\pi x_0)}{\alpha^2\cot(\alpha\pi x_*)+(1-\alpha)^2\cot((1-\alpha)\pi x_*)-\cot(\pi x_*)}\bigg|\\
&\les\frac{3\csc^2(\pi x_0)}{2(1 - \alpha^{3}-(1-\alpha)^{3})x_*}
\les\frac{3}{2\pi(1 - \alpha^{3}-(1-\alpha)^{3})x_*x_0^2(1-x_0)^2}\\
&\les\frac{1}{2\pi\alpha(1-\alpha)(x_0-2^{1-k})x_0^2(1-x_0)^2}=M(x_0,k).
\end{align*}
 Since $(x_0-2^{1-k})x_0^2(1-x_0)^2
\ges x_*^3(1-x_*-2^{-k})^2$, by the stopping condition of the binary search of line~\ref{alg:inversion_newton_raphson_bisection_stop} in \hyperref[alg:inversion_newton_raphson]{NR-Alg}, we only need 
$$
2^{1-k}\cdot\frac{1}{2\pi\alpha(1-\alpha)x_*^3(1-x_*-2^{-k})^2}\les\frac{1}{2}.
$$
implying $k=\Oh(|\log(x_*(1-x_*))+\log(\alpha(1-\alpha))|)$. Thus, the binary search for this $x_0>x_*$ requires $\Oh(|\log(x_*(1-x_*))+\log(\alpha(1-\alpha))|)$ steps.
\end{proof}


\begin{prop}
\label{prop:inverse_u*sigma}
Define the auxiliary function 
\begin{equation}
    \label{eq:u*sigma^alpha_auxiliary}
    M(x,k):=\frac{\sigma_\alpha(1/2)^\alpha}{\sigma_\alpha(0)^\alpha}
    \bigg(1+\frac{x_0\alpha}{2}
    \big(\sigma_\alpha''(x)/\sigma_\alpha(x)
    +(\alpha-1)\sigma_\alpha'(x)^2/\sigma_\alpha(x)^{2}
        \big)\bigg).
\end{equation}
The total cost of numerical inversion of the function $u\mapsto u\sigma_\alpha(u)^\alpha$ on the interval $[0,z_*]\subset[0,1/2]$ via \hyperref[alg:inversion_newton_raphson]{NR-Alg} with auxiliary function $M(x,k)$ is
$\Oh(|\log(1-\alpha)|+\log N)$ where $N$ specifies the number of precision bits.
\end{prop}

\begin{proof}
 On $[0,z_*]$, the function is convex and increasing with derivative $\sigma_\alpha(u)^\alpha+u\alpha\sigma_\alpha'(u)\sigma_\alpha(u)^{\alpha-1}$. Thus, given $y\in(0,\sigma(z_*))$, if our initial estimate $x_0$ (obtained via binary search) of the root $x_*$ of the function $u\sigma_\alpha(x)^\alpha-y$ is larger than $x_*$ and sufficiently close to $x_*$, then the Newton--Raphson sequence
\[
x_{n+1}
=x_n-\frac{x_n\sigma_\alpha(x_n)^\alpha-y}{\sigma_\alpha(x_n)^\alpha
+x_n\alpha\sigma_\alpha'(x_n)\sigma_\alpha(x_n)^{\alpha-1}}
=x_n-\frac{x_n-y/\sigma_\alpha(x_n)^\alpha}{1
+x_n\alpha\sigma_\alpha'(x_n)/\sigma_\alpha(x_n)},
\quad n\in\N,
\]
is decreasing and has quadratic convergence to the root $x_*$. Note that  $\log(\sigma_\alpha)''=\sigma_\alpha''/\sigma_\alpha-(\sigma_\alpha'/\sigma_\alpha)^2$ and $\log(\sigma_\alpha)'=\sigma_\alpha'/\sigma_\alpha$ are increasing, and thus
\begin{align*}
\frac{1}{2}\frac{\sup_{x\in[x_*,x_0]}|(x\sigma_\alpha(x)^\alpha)''(x)|}{\inf_{x\in[x_*,x_0]}|(x\sigma_\alpha(x)^\alpha)'(x)|}&=
\frac{1}{2}
\frac{\sup_{x\in[x_*,x_0]}(2\sigma_\alpha(x)^\alpha
+x\alpha(\sigma_\alpha''(x)\sigma_\alpha(x)^{\alpha-1}+(\alpha-1)\sigma_\alpha'(x)^2\sigma_\alpha(x)^{\alpha-2}))}{\inf_{x\in[x_*,x_0]}(\sigma_\alpha(x)^\alpha
+x\alpha\sigma_\alpha'(x)\sigma_\alpha(x)^{\alpha-1})}\\
\les&\frac{\sigma_\alpha(1/2)^\alpha}{\sigma_\alpha(0)^\alpha}\frac{1
+x_0\alpha(\sigma_\alpha''(x_0)/\sigma_\alpha(x_0)+(\alpha-1)\sigma_\alpha'(x_0)^2/\sigma_\alpha(x_0)^{2})/2}{1
+(x_0-2^{1-k}z_*)\alpha\sigma_\alpha'(x_0-2^{1-k}z_*)/\sigma_\alpha(x_0-2^{1-k}z_*)}\\
\les&\frac{\sigma_\alpha(1/2)^\alpha}{\sigma_\alpha(0)^\alpha}(1
+\frac{x_0\alpha}{2}(\sigma_\alpha''(x_0)/\sigma_\alpha(x_0)+(\alpha-1)\sigma_\alpha'(x_0)^2/\sigma_\alpha(x_0)^{2}))=M(x_0,k).
\end{align*}
Since $z_*<1/2$, 
\[
\frac{\sigma_\alpha(1/2)^\alpha}{\sigma_\alpha(0)^\alpha}
=\bigg(\frac{\sin(\alpha\pi/2)\sin{(1-\alpha)\pi/2}}{\alpha^\alpha(1-\alpha)^{1-\alpha}}\bigg)^{\alpha^2/(1-\alpha)}
\] has uniform upperbound for $\alpha\in(0,1)$ and
\begin{multline*}
1
+x_0\alpha(\sigma_\alpha''(x_0)/\sigma_\alpha(x_0)+(\alpha-1)\sigma_\alpha'(x_0)^2/\sigma_\alpha(x_0)^{2})/2
\les1+\tfrac{1}{4}\big((\log\sigma_\alpha)''(1/2)+\alpha((\log\sigma_\alpha)'(1/2)^2\big)\\
=1+\frac{1}{4}\bigg(-\frac{\alpha}{1-\alpha}\bigg(\frac{\alpha\pi}{\sin(\alpha\pi/2)}\bigg)^2
    -\bigg(\frac{(1-\alpha)\pi}{\sin((1-\alpha)\pi/2) }\bigg)^2
    +\pi^2
    +\alpha\bigg(\frac{\alpha}{1-\alpha}\frac{\cos(\alpha\pi/2)}{\sin(\alpha\pi/2)}+\cos(\alpha\pi/2)\bigg)^2\bigg)^2
\end{multline*}
By the stopping condition of the binary search of line~\ref{alg:inversion_newton_raphson_bisection_stop} in \hyperref[alg:inversion_newton_raphson]{NR-Alg}, we only need the final expression in the previous display to be smaller than $2^{k-2}$, implying $k=\Oh(|\log(1-\alpha)|)$. Thus, the binary search for $x_0>x_*$ requires $\Oh(|\log(1-\alpha)|)$ iterations. 
\end{proof}

For $z\in(1/2,1)$, define the function
\begin{equation}
\label{eq:unnormalilzed_distribution_c1}
F(x)=\sin(\pi(1-\alpha))(1-r)^{-1}(1-x)^{1-r}+\pi\alpha(1-\alpha)\cos(\pi\alpha)(2-r)^{-1}(1-x)^{2-r},\quad x\in(1/2,z). 
\end{equation}

\begin{prop}
\label{prop:inverse_c1}
Define auxiliary function
\begin{equation}
\label{eq:c2_auxiliary}
M(x,k):=\frac{(1-z)^{-(r+1)}(r\sin(\pi(1-\alpha))+\pi\alpha(1-\alpha)\cos(\pi\alpha)(1-r))}{(1-1/2)^{-r}\big(\sin(\pi(1-\alpha))+\pi\alpha(1-\alpha)\cos(\pi\alpha)(1-z)\big)}.
\end{equation}
For $\alpha\in(0,1/2)$, the total cost of numerical inversion of the function $F$ via \hyperref[alg:inversion_newton_raphson]{NR-Alg} with auxiliary function $M(x,k)$ is
$\Oh(|\log(1-z)|+\log N)$ where $N$ specifies the number of precision bits.
\end{prop}

\begin{proof}
Note that $F''(x)=(1-x)^{-(r+1)}(r\sin(\pi(1-\alpha))-\pi\alpha(1-\alpha)\cos(\pi\alpha)(1-r)(1-x))$ for $x\in(1/2,z)$. Given $y\in(F(z),F(1/2))$, if our initial estimate $x_0$ (obtained via binary search) of the root $x_*$ of the function $F(x)-y$ is larger than $x_*$ and sufficiently close to $x_*$, then the Newton--Raphson sequence
\[
x_{n+1}
=x_n-\frac{\sin(\pi(1-\alpha))(1-r)^{-1}(1-x_n)^{1-r}+\pi\alpha(1-\alpha)\cos(\pi\alpha)(2-r)^{-1}(1-x_n)^{2-r}-y}{(1-x_n)^{-r}(\sin(\pi(1-\alpha))+\pi\alpha(1-\alpha)\cos(\pi\alpha)(1-x_n))}
\quad n\in\N,
\]
 has quadratic convergence to the root $x_*$. Note that
\begin{align*}
\frac{1}{2}\frac{\sup_{x\in[x_*,x_0]}|F''(x)|}{\inf_{x\in[x_*,x_0]}|F'(x)|}=&\frac{1}{2}
\frac{\sup_{x\in[1/2,z]}|(1-x)^{-(r+1)}\big(r\sin(\pi(1-\alpha))-\pi\alpha(1-\alpha)\cos(\pi\alpha)(1-r)(1-x)\big)|}{\inf_{x\in[1/2,z]}|(1-x)^{-r}\big(\sin(\pi(1-\alpha))+\pi\alpha(1-\alpha)\cos(\pi\alpha)(1-x)\big)|}\\
\les&
\frac{(1-z)^{-(r+1)}(r\sin(\pi(1-\alpha))+\pi\alpha(1-\alpha)\cos(\pi\alpha)(1-r))}{(1-1/2)^{-r}\big(\sin(\pi(1-\alpha))+\pi\alpha(1-\alpha)\cos(\pi\alpha)(1-z)\big)}=M(x_0,k)
\end{align*}
By the stopping condition of the binary search of line~\ref{alg:inversion_newton_raphson_bisection_stop} in \hyperref[alg:inversion_newton_raphson]{NR-Alg}, we only need 
$$
2^{1-k}\frac{(1-z)^{-(r+1)}(r\sin(\pi(1-\alpha))+\pi\alpha(1-\alpha)\cos(\pi\alpha)(1-r))}{(1-1/2)^{-r}\big(\sin(\pi(1-\alpha))+\pi\alpha(1-\alpha)\cos(\pi\alpha)(1-z)\big)}\les\frac{1}{2},
$$
implying $k\ges\Oh(|\log(1-z)|)$. Thus the binary search for $x_0>x_*$ requires $\Oh(|\log(1-z)|)$ iterations. 
\end{proof}

\subsubsection{Simulation from the density \texorpdfstring{$\wt\psi_s^{(2)}$}{psi 2}: proof of Proposition~\ref{prop:psi_2}.}
\label{subsec:psi_(2)-2}

Before proving Proposition~\ref{prop:psi_2}, we need following lemmas, describing the properties of functions $\rho$ and $\sigma_\alpha$.
\begin{lem}
\label{lem:1_over_rho_concave}
The function $1/\rho$, where $\rho$ is defined in~\eqref{eq:zolotarev_density}, is concave on $(0,1)$.
\end{lem}

\begin{proof}
Clearly, by definition~\eqref{eq:zolotarev_density},  $\log(\sigma_\alpha(x))
=(1-\alpha)^{-1}(\alpha\psi_\alpha(x)
+(1-\alpha)\psi_{1-\alpha}(x))$, where $\psi_a:x\mapsto\log(\sin(a\pi x))-\log(\sin(\pi x))$ for $a\in(0,1)$. Recall that $\cot(\pi x)=(\pi x)^{-1}-(2/\pi)\sum_{n=1}^\infty \zeta(2n)x^{2n-1}$ for $|x|\in(0,1)$ where $\zeta$ is the Riemann zeta function. Thus, for any $a\in(0,1)$, 
\[
\psi'_a(x) 
= a\pi\cot(a\pi x)-\pi\cot(\pi x)
= 2\sum_{n=1}^\infty 
    \zeta(2n)(1 - a^{2n})x^{2n-1},
\quad x\in(0,1).
\]
This shows that $\psi'_a$ and all its derivatives are positive on $(0,1)$, implying the same is true of  $\log(\sigma_\alpha)$. Thus,
$\sigma_\alpha''=(\me^{\log(\sigma_\alpha)})''=\me^{\log(\sigma_\alpha)}(\log(\sigma_\alpha)'^2+\log(\sigma_\alpha)'')\geq0$, making $\sigma_\alpha$  convex on $(0,1)$. Similarly, a derivative of $\sigma_\alpha$ of any order is a linear combination with non-negative coefficients of derivatives of $\log\sigma_\alpha$, multiplied with $\me^{\log(\sigma_\alpha)}$, and thus non-negative itself on the interval $(0,1)$.

To prove that $1/\rho$ is concave, let $f\coloneqq \log\rho=(1-\alpha)\log\sigma_\alpha$ and note that it suffices to show that $f''\ges (f')^2$. If $\alpha=1/2$, then $x\mapsto f''(x)-(f'(x))^2$ is constant since $\tfrac{\md}{\md x}\cot(x)=-1-\cot^2(x)$ and the double angle formula $2\cot(2x)\cot(x)=\cot(x)^2-1$ imply that $f''(x) = \pi^2(\tfrac{3}{4}
    +\cot^2(\pi x)
    -\tfrac{1}{4}\cot^2(\tfrac{\pi}{2} x))$, and 
\[
f'(x)^2 = \pi^2\big(\cot^2(\pi x)
    +\tfrac{1}{4}\cot^2(\tfrac{\pi}{2} x)
    -\cot(\pi x)\cot(\tfrac{\pi}{2} x)\big)
= \pi^2\big(\tfrac{1}{2}+\cot^2(\pi x)
    -\tfrac{1}{4}\cot^2(\tfrac{\pi}{2} x)\big)
=f''(x)-\tfrac{\pi^2}{4}.
\]

To prove the general case $\alpha\in(0,1)$, denote $a_n(\alpha)\coloneqq \zeta(2n)(1-\alpha^{2n+1}-(1-\alpha)^{2n+1})$, then we have $f'(x) = 2\sum_{n=1}^\infty a_n(\alpha)x^{2n-1}$, $x\in(0,1)$, implying that
\[
f''(x)
= 4\sum_{n=1}^\infty (n-\tfrac{1}{2})a_n(\alpha) x^{2n-2}
\quad\text{and}\quad
f'(x)^2
= 4\sum_{n=2}^\infty 
    \Bigg(\sum_{k=1}^{n-1} a_k(\alpha) a_{n-k}(\alpha)\Bigg)x^{2n-2}.
\]
We will show that $f''(x)\ges f'(x)^2$ termwise. When $\alpha=1/2$, we have $f''=(f')^2+\pi^2/4$ so all the coefficients of $f''$ and $(f')^2$ of $x^{2n-2}$ for $n\ges 2$ agree:
\[
\big(n-\tfrac{1}{2}\big)\zeta(2n)
= \sum_{k=1}^{n-1}
    \zeta(2k)\zeta(2n-2k)\frac{(1-4^{-k})(1-4^{k-n})}{1-4^{-n}},
    \quad n\ges 2.
\]
To establish the termwise inequality for general $\alpha\in(0,1)$, it thus suffices to show that
\begin{gather*}
\big(n-\tfrac{1}{2}\big)\zeta(2n)
\ges \sum_{k=1}^{n-1}
    \zeta(2k)\zeta(2n-2k)p_{n,k}(\alpha),
\quad\text{where}\\
p_{n,k}(\alpha)
\coloneqq\frac{(1-\alpha^{2k+1}-(1-\alpha)^{2k+1})
        (1-\alpha^{2(n-k)+1}-(1-\alpha)^{2(n-k)+1})}{
    (1-\alpha^{2n+1}-(1-\alpha)^{2n+1})}.
\end{gather*}
Since the equality is attained when $\alpha=1/2$, it suffices to show that $p_{n,k}$ attains its maximum at $\alpha=1/2$. Since $p_{n,k}(1/2)>0=p_{n,k}(0+)=p_{n,k}(1-)$, it suffices to show that $1/2$ is the only critical point of $p_{n,k}$. This will follow once we show that the derivative
\[
q_m(x)
\coloneqq\tfrac{\md}{\md x}\log(1-x^m-(1-x)^m)
=m\frac{(1-x)^{m-1}-x^{m-1}}{1-x^m-(1-x)^m},
\]
is positive and decreasing (resp. negative and increasing) in $m\ges 3$ for any $x\in(0,1/2)$ (resp. $x\in(1/2,1)$). Indeed, $\tfrac{\md}{\md x}\log(p_{n,k}(x))=q_{2k+1}(x)+q_{2(n-k)+1}(x)-q_{2n+1}(x)$ so the only solution to $\tfrac{\md}{\md x}p_{n,k}(x)=0$ is $x=1/2$ (where $q_m(1/2)=0$ for any $m\ges 3$).

For $3\les i<m$, the inequality $q_i(x)>q_m(x)$ for $x\in(0,1/2)$ is equivalent to
\[
A_{i,m}(c):=i(c^i-1)((c+1)^m-c^m-1)
-m(c^m-1)((c+1)^i-c^i-1)>0,
\quad\text{where}\quad 
c\coloneqq\frac{1-x}{x}>1. 
\]
This inequality is further equivalent to 
\[ 
A_{i,m}(c)=i\sum_{j=1}^{m-1}\binom{m}{j}c^{j+i}
+m\sum_{j=1}^{i-1}\binom{i}{j}c^{j}
-\Bigg( m\sum_{j=1}^{i-1}\binom{i}{j}c^{j+m}
+i\sum_{j=1}^{m-1}\binom{m}{j}c^{j}\Bigg)>0.
\]
It is easy to see that the positive and negative  coefficients of the powers of $c$ are the same (with different corresponding powers of $c$). An elementary but tedious calculation shows that $A_{i,m}(c)$ can be rewritten as 
\begin{align*}
A_{i,m}(c)
&=\sum_{j=1}^{i-1}\Bigg(\binom{m}{m-j}i-\binom{i}{j}m\Bigg)\big(c^{m-j+i}-c^{j}\big)\\
&\qquad+\sum_{j=i}^{i+\lfloor(m-i)/2\rfloor}\Bigg(\binom{m}{m-j}i-\binom{m}{m+i-j}i\Bigg)\big(c^{m-j+i}-c^{j}\big).
\end{align*}
Since $c>1$ and $\binom{m}{j}i>\binom{i}{j}m$ for $j\les i-1$ and $\binom{m}{j}i\ges \binom{m}{j-i}i$ for $i\les j\les i+(m-i)/2$, we deduce that $A_{i,m}(c)>0$ completing the proof.
\end{proof}

\begin{lem}\label{lem:sigma_bound}
For $x\in(0,1)$, we have 
$$4^{-1/(1-\alpha)}\sin(\alpha\pi)^{1/(1-\alpha)}(1-x)^{-1/(1-\alpha)}
\les\sigma_\alpha(x)\les
\alpha^{\alpha/(1-\alpha)}(1-\alpha)(1-x)^{-1/(1-\alpha)}.
$$
In particular, if $z\in(0,1)$ satisfies $\sigma_\alpha(z)=\alpha s^r$ for some $s>0$ (recall $r=\alpha/(1-\alpha)$), we have 
\begin{equation}
\label{eq:bound_on_solution}
1/(1-z)\les4\alpha^{1-\alpha}
s^{\alpha}/\sin(\alpha\pi).
\end{equation}
\end{lem}

\begin{proof}
By the concavity of function $x\mapsto\sin(x)$, we have $\sin(\alpha\pi x)\ges\sin(\alpha\pi)x$ and $\sin((1-\alpha)\pi x)\ges\sin((1-\alpha)\pi)x$ for $x\in[0,1]$. Moreover, we have $\sin(\alpha\pi x)\les\alpha\pi x$ and $\pi x(1-x)\les\sin(\pi x)\les 4x(1-x)$. 
\end{proof}

\begin{lem}\label{lem:sigma'_upperbound}
For $x\in(0,1)$, $\sigma_\alpha'(x)\les
\tfrac{\md}{\md x}[\alpha^r(1-\alpha)(1-x)^{-(r+1)}]=
\alpha^r(1-\alpha)(r+1)(1-x)^{-(r+2)}$. 
\end{lem}

\begin{proof}
Note that $\sigma_\alpha'(x)=\sigma_\alpha(x)(\log(\sigma_\alpha))'(x)$. By Lemma~\ref{lem:sigma_bound}, we only need to show that
\[
(r+1)^{-1}(\log(\sigma_\alpha))'(x)=\alpha^2\pi\cot(\alpha\pi x)+(1-\alpha)^2\pi\cot((1-\alpha)\pi x)-\pi\cot(\pi x)\les (1-x)^{-1}.
\]
The power series expansion of $\cot$ shows that the left-hand side is equivalent to
$$
\sum_{n=1}^\infty 2\zeta(2n)(1-\alpha^{2n+1}-(1-\alpha)^{2n+1})x^{2n-1}
$$
where $\zeta$ is Riemann zeta function. Note that $1-\alpha^{2n+1}-(1-\alpha)^{2n+1}$ attains its maximum value when $\alpha=1/2$, so it suffices to establish the following inequality
\[
\pi\cot(\pi x/2)/2-\pi\cot(\pi x)\les(1-x)^{-1}.
\]
In fact, let $t=\cot(\pi x/2)$, then
\[
\frac{\pi t}{2}-\frac{\pi(t^2-1)}{2t}
=\frac{\pi}{2t}
=\frac{\pi}{2}\tan(\pi x/2)
=\frac{\pi}{2}\cot(\pi(1-x)/2)
\les \frac{1}{1-x}.\qedhere
\]
\end{proof}

Now we can complete the proof of Proposition~\ref{prop:psi_2}.
\begin{proof}[Proof of Proposition~\ref{prop:psi_2}]
In line~\ref{step:inversion_sigma},
 \hyperref[alg:psi_2]{$\psi^{(2)}$-Alg} sets $z:=0$ when $\alpha s^r<\sigma_\alpha(0+)=(1-\alpha)\alpha^{\alpha/(1-\alpha)}$.  Otherwise the equation $\sigma_\alpha(x)=\alpha s^{r}$,
$x\in(0,1)$, has a unique solution by Lemma~\ref{lem:varphi_convex}
and we may define $z:=\sigma_\alpha^{-1}(\alpha s^r)$ 
to be that solution. 
This inversion is carried out by the Newton--Raphson method, converging quadratically fast since $\sigma_\alpha$ is log-convex, see Section~\ref{sec:NewtonRaphson} below for details. 
By Proposition~\ref{prop:inverse_sigma}
the computational cost of finding $z$ is bounded above by
$\Oh(\log(1/(z(1-z)))+|\log(\alpha(1-\alpha))|+\log N)$.
Since $\sigma_\alpha(z)-\sigma_\alpha(0+)=\alpha s^r-\alpha^r(1-\alpha)$,
if $z<1/2$,
Lemma~\ref{lem:sigma_linear_upperbound}
implies
$
1/z\les(\pi-2\me^{-1})(1-\alpha)/(\alpha s^r-\alpha^r(1-\alpha))
$. 
Hence, by~\eqref{eq:bound_on_solution} in  Lemma~\ref{lem:sigma_bound}, we obtain
\begin{align*}
\frac{1}{z}\frac{1}{1-z}
\les \frac{(\pi-2\me^{-1})(1-\alpha)}{\alpha s^r-\alpha^r(1-\alpha)}\frac{4\alpha^{1-\alpha}s^\alpha}{\sin(\alpha\pi)}\les\frac{ 4(\pi-2\me^{-1})s^\alpha}{\pi\alpha^2(s^r-\alpha^{r-1}(1-\alpha))},
\end{align*}
where the second inequality follows from
$(1-\alpha)\alpha^{-\alpha}/\sin(\alpha\pi)\les\alpha^{-2}/\pi$, which holds since for every $\alpha\in(0,1)$ we have $\sin(\pi\alpha)\ges\pi\alpha(1-\alpha)$. 
Thus, the computational cost of finding $z$ is bounded above by
$$\Oh\left(\log^+\left(1/(s^r-\alpha^{r-1}(1-\alpha))\right)+\alpha\log^+(s)+|\log(\alpha)| +|\log(1-\alpha)|+\log N\right).$$

The simulation from $\wt\psi_s^{(2)}$ can be decomposed into two steps: first sample from the marginal density proportional to $\tilde f(y)= \sigma_\alpha(y)^\alpha \exp(-\sigma_\alpha(y)s^{-r})$ and then from the conditional law proportional to $x\mapsto \1\{x>s^{-r}\}(x-s^{-r})^{-\alpha}\exp(-\sigma_\alpha(y)(x-s^{-r}))$ (which is a gamma with shape $1-\alpha$ and rate $\sigma_\alpha(y)$ and translated by $s^{-r}$). Now let us establish a simulation algorithm from the density $\tilde f$.

 To sample from $\tilde f$, we first toss (at most two) coins to determine on which of the  following (possibly degenerate) intervals lies the sample: on $[0,z_*]$, $[z_*,z]$ or $[z,1]$, where  $z_*=z\wedge(1/2)$. The time cost of this step is constant, see Remark~\ref{rem:prob_discrete_rv}. 
\begin{enumerate}[label=\alph*), leftmargin=2em]
\item Case $D=2$ (the sample lies in $[z,1]$). Recall that by Lemma~\ref{lem:varphi_convex}  we have $\sigma_\alpha'',\sigma_\alpha',\sigma_\alpha>0$ on $(0,1)$ and that $z$ is the unique solution to $\sigma_\alpha(z)=\alpha s^r$. By definition of $\tilde f$, the following  equalities 
\[
(\log\tilde f)'=(\alpha/\sigma_\alpha-s^{-r})\sigma_\alpha'\quad\&\quad
(\log\tilde f)''=(\alpha/\sigma_\alpha-s^{-r})\sigma_\alpha''-(\sigma_\alpha'/\sigma_\alpha)^2<0
\]
hold on the interval $(z,1)$. Since $\sigma_\alpha$ is strictly increasing and $\sigma_\alpha''>0$, $\tilde f|_{[z,1]}$ log-concave on $[z,1]$ with mode at $z$. Thus we may apply \hyperref[alg:Devroye]{LC-Alg} to obtain a sample in the interval $[z,1]$. 

Again, the computational cost of \hyperref[alg:Devroye]{LC-Alg} has two parts: (I) the cost to finding the value of $a_1$ and (II) the expected cost of the accept-reject step. By~\cite[\S4]{devroye2012note}, the expected cost of (II) is bounded above by $5$. 

To analyse the cost of finding $a_1$ in \hyperref[alg:Devroye]{LC-Alg}, our idea is to find $y\in(z,(z+1)/2)$ such that $\tilde f(y)/\tilde f(z)\ges 1/4$ since $(z+1)/2$ is the first number we check in the binary search. And then we only need to estimate the upperbound of $\log(1/(y-z))$. Recall that $\tilde f(x):=\sigma_\alpha(x)^\alpha\exp(-\sigma_\alpha(x)s^{-r})$ and $\tilde f(z)=\alpha s^r$. In fact, for $y\in(z,1)$ satisfying $\sigma_\alpha(y)=(\alpha+2\log2)s^r$, we have 
\[
\frac{\tilde f(y)}{\tilde f(z)}=\frac{(\alpha+2\log2)s^r\exp(-(\alpha+2\log2))}{\alpha s^r\exp(-\alpha)}\ges\frac{1}{4}.
\]
By Lemma~\ref{lem:sigma'_upperbound} and the fact that $\sigma_\alpha$ is convex, we have
\begin{align*}
(r+1)\alpha^r(1-\alpha)(1-(1+z)/2)^{-(r+2)}
&=\left(\alpha^r(1-\alpha)(1-x)^{-(r+1)}\right)'\Big|_{x=\frac{1+z}{2}}\\
&\ges\sigma_\alpha'\Big(\frac{z+1}{2}\Big)
\ges\frac{\sigma_\alpha(y)-\sigma_\alpha(z)}{y-z}
=\frac{(\log4)s^r}{y-z}.
\end{align*}
Together with~\eqref{eq:bound_on_solution} in Lemma~\ref{lem:sigma_bound}, we have
\begin{align*}
\frac{1}{y-z}\les\frac{2^{r+2}\alpha^r}{(\log4)s^r(1-z)^{r+2}}
&\les\frac{2^{(2-\alpha)/(1-\alpha)}\alpha^{\alpha/(1-\alpha)}}{(\log4)s^{\alpha/(1-\alpha)}}\frac{4^{(2-\alpha)/(1-\alpha)}\alpha^{2-\alpha} s^{\alpha(2-\alpha)/(1-\alpha)}}{\sin(\alpha\pi)^{(2-\alpha)/(1-\alpha)}}\\
&=s^\alpha\frac{8^{(2-\alpha)/(1-\alpha)}\alpha^{\alpha/(1-\alpha)+2-\alpha}}{\sin(\alpha\pi)^{(2-\alpha)/(1-\alpha)}}.
\end{align*}

Therefore the steps of binary search in Devroye's algorithm is bounded by
\[
1+\log(1/(y-z))\les\Oh(1+\alpha\log^+(s)+(1-\alpha)^{-1}|\log(1-\alpha)|).
\]

\item Case $D=0$ (the sample lies in $[0,z_*]$, where $z_*=z\wedge(1/2)$). We propose a sample $Y_s$ from the distribution function  
\begin{equation}
\label{eq:distribution_u*sigma^alpha}
u\mapsto u\sigma_\alpha(u)^\alpha/(z_*\sigma_\alpha(z_*)^\alpha)
\end{equation}
on $[0,z_*]$ via the inversion method. By Proposition~\ref{prop:inverse_u*sigma}, the computation cost of inversion is bounded above by $\Oh(|\log(1-\alpha)|+\log N)$. Since $Y_s$ has a strictly positive (by Lemma~\ref{lem:varphi_convex}) density, proportional to $u\mapsto \sigma_\alpha(u)^\alpha+\alpha u\sigma_\alpha'(u)\sigma_\alpha(u)^{\alpha-1}$, we accept the sample $Y_s$ with probability $h(Y_s)$, where
\begin{equation}
\label{eq:prob_h}
\begin{aligned}
1\ges h_b(u)
&\coloneqq
\frac{\sigma_\alpha(u)^\alpha\me^{-\sigma_\alpha(u)s^{-r}}}{\sigma_\alpha(u)^\alpha
+\alpha u\sigma_\alpha'(u)\sigma_\alpha(u)^{\alpha-1}}
=\frac{\me^{-\sigma_\alpha(u)s^{-r}}}{1
+\alpha u\sigma_\alpha'(u)/\sigma_\alpha(u)}
\ges\frac{\me^{-\alpha}}{1
+(\alpha/2)(\log\sigma_\alpha)'(\tfrac{1}{2})}\\
&=\me^{-\alpha}\big(1
+(r\pi/2)\big(\alpha^2\cot\big(\alpha\tfrac{\pi}{2}\big) + (1-\alpha)^2\cot\big((1-\alpha)\tfrac{\pi}{2}\big)\big)\big)^{-1} \qquad\text{for $u\in[0,z_*]$.}
\end{aligned}
\end{equation}
The first inequality in~\eqref{eq:prob_h} holds since both  $\sigma_\alpha'>0$ and $\sigma_\alpha>0$ on $(0,1)$. The second inequality follows from Lemma~\ref{lem:varphi_convex}
because $(\log \sigma_\alpha)'$ is increasing on $(0,1)$ and $z_*\leq 1/2$. A direct calculation using the definition of $\sigma_\alpha$ in~\eqref{eq:zolotarev_density} yields the final expression in~\eqref{eq:prob_h}, which is decreasing in $\alpha$ and approaching its infimum $\me^{-1}(2+\pi^2/4)^{-1}\approx 0.0823$ as $\alpha\to 1$.

    \item The case $D=1$ (where $z\in(1/2,1)$ and the sample lies in $[1/2,z]$), is split in two subcases: 
    \begin{enumerate}[label=(\roman*),leftmargin=1em]
        \item Case $\alpha\les 1/2$. We propose from the density proportional to 
\begin{equation}
    \label{eq:density_c1}
    u\mapsto \1_{\{1/2<u\les z\}}
    (1-u)^{-r}(\sin(\pi(1-\alpha)) 
        + \pi\alpha(1-\alpha) \cos(\pi\alpha)(1-u)),
\end{equation}
 When $\alpha=1/2$, the density is proportional to $u\mapsto\1_{\{1/2<u\les z\}}(1-u)^{-r}$.  We can sample from this via inversion directly and the time cost is a constant. When $\alpha\in(0,1/2)$, we can sample from the density via inversion via Newton--Raphson method: the CDF of our target density equals 
\begin{equation}
\label{eq:disti=ribution_c1}
x\mapsto(F(x)-F(1/2)/(F(z)-F(1/2))\qquad x\in[1/2,z]
\end{equation}
with $F$ defined in~\eqref{eq:unnormalilzed_distribution_c1}.
 So to get a sample from the target density we only need to solve $F(x)=U$ where $U\sim\Unif(F(1/2),F(z))$. By Proposition~\ref{prop:inverse_c1} and Lemma~\ref{lem:sigma_bound}, the time cost of inversion is bounded above by 
\[
\Oh(\log(1/1-z))+\log N)
\les\Oh\bigg(\log\bigg(\frac{4\alpha^{1-\alpha}s^\alpha}{\sin(\alpha\pi)}\bigg)+\log N\bigg)
\les\Oh(\alpha\log_2^+(s)+\log N).
\]

Then we accept the sample $U$ with probability $h_{c,1}(U)$, where $h_{c,1}$ is given by 
\begin{align}
\label{eq:h_c1}
1\ges h_{c,1}(u)
&\coloneqq
\frac{\sigma_\alpha(u)^\alpha
    \me^{-\sigma_\alpha(u)s^{-r}}
    \sin(\pi(1-\alpha))^{1-\alpha}
    2^{r}(1-u)^{r}}
    {\sin(\pi (1-\alpha))
        +\pi\alpha(1-\alpha)\cos(\pi\alpha)(1-u)}
\ges \frac{4}{3\pi\sqrt{\me}}
\approx 0.25742.
\end{align}

We now prove the second inequality in~\eqref{eq:h_c1}.
By concavity, for $u\in[1/2,1]$, we have
\begin{align*}
\sin(\pi (1-\alpha)u)^\alpha
&=\sin(\pi(1-\alpha))^\alpha 
    - \pi \alpha(1-\alpha)\int_u^1\sin(\pi(1-\alpha)x)^{\alpha-1}\cos(\pi (1-\alpha)x)\md x\\
&\les \sin(\pi (1-\alpha))^{\alpha-1}
    \big(\sin(\pi(1-\alpha)) 
    + \pi \alpha(1-\alpha) \cos(\pi\alpha)(1-u)\big).
\end{align*}
Moreover, again by concavity, the quotient satisfies
\[
1\ges\frac{\sin(\pi (1-\alpha)u)^\alpha\sin(\pi (1-\alpha))^{1-\alpha}}{
    \sin(\pi(1-\alpha)) 
    + \pi \alpha(1-\alpha) \cos(\pi\alpha)(1-u)}
\ges\frac{\sin(\pi(1-\alpha)/2)^\alpha\sin(\pi (1-\alpha))^{1-\alpha}}{
    \sin(\pi(1-\alpha)) 
    + \pi \alpha(1-\alpha) \cos(\pi\alpha)/2}
\ges\frac{2}{3}.
\]
Using the elementary inequalities $2(1-x)\les\sin(\pi x)\les \min\{1, \pi (1-x)\}$ for $x\in(1/2,1)$, we have
\begin{align*}
1\ges h_{c,1}(u)=
\frac{\sigma_\alpha(u)^\alpha
    \me^{-\sigma_\alpha(u)s^{-r}}
    2^{r}(1-u)^{r}}
    {\sin(\pi (1-\alpha)u)^\alpha}\ges
\me^{-\alpha}\frac{2^{r(\alpha+1)}}{\pi^r}
    \frac{2}{3}(1-\alpha u)^{\alpha r}
\ges \me^{-\alpha}\bigg(\frac{2}{\pi}\bigg)^{r}\frac{2}{3}
\ges \frac{4}{3\pi\sqrt{\me}}.
\end{align*}

\item Case $\alpha>1/2$. We propose from the distribution function 
\begin{equation}
\label{eq:distribution_c2}
u\mapsto (\rho(u)^{r-1}-\rho(1/2)^{r-1})/(\rho(z)^{r-1}-\rho(1/2)^{r-1}),
\end{equation} which can be sampled from via the inversion method. Note that now the solution $x_*\in(1/2,z)$ in the inversion method, by Proposition~\ref{prop:inverse_sigma} and Lemma~\ref{lem:sigma_bound}, the time cost of inversion method is bounded above by 
\begin{align*}
\Oh\big(\log(1/(1-z))+|\log(1-\alpha)|+\log N\big)
&\les\Oh\bigg(\log\bigg(\frac{4\alpha^{1-\alpha}s^\alpha}{\sin(\alpha\pi)}\bigg)+|\log(1-\alpha)|+\log N\bigg)\\
&=\Oh\big(\alpha\log_2^+(s)+|\log(1-\alpha)|+\log N\big).
\end{align*}
  Then accept the proposal with probability 
\begin{equation}
    \label{eq:h_c2}
    \begin{aligned}
h_{c,2}(u)
&\coloneqq
C_\alpha\frac{\sigma_\alpha(u)^\alpha\me^{-\sigma_\alpha(u)s^{-r}}}{\rho'(u)\rho(u)^{r-2}}
=C_\alpha\me^{-\sigma_\alpha(u)s^{-r}}\frac{\rho(u)^2}{\rho'(u)}\\
&
=C_\alpha\frac{\me^{-\sigma_\alpha(u)s^{-r}}\sin(\alpha\pi u)^\alpha\sin((1-\alpha)\pi u)^{1-\alpha}/\sin(\pi u)}{\pi(\alpha^2\cot(\alpha\pi u) + (1-\alpha)^2\cot((1-\alpha)\pi u) - \cot(\pi u))},
\end{aligned}
\end{equation}
where
\begin{align*}
C_\alpha\coloneqq\frac{\rho'(\tfrac{1}{2})}{\rho(\tfrac{1}{2})^2}
&=\pi\frac{\alpha^2\cot(\alpha\tfrac{\pi}{2}) + (1-\alpha)^2\cot((1-\alpha)\tfrac{\pi}{2})}{\sin(\alpha\tfrac{\pi}{2})^\alpha\sin((1-\alpha)\tfrac{\pi}{2})^{1-\alpha}}\\
&\ges
(1-\alpha)(\alpha^2\pi+2\sin(\alpha\tfrac{\pi}{2}))
\ges (1-\alpha)(\tfrac{\pi}{4}+\sqrt{2}).
\end{align*}
By Lemma~\ref{lem:1_over_rho_concave}, $\rho(u)^2/\rho'(u) = -(\frac{\md}{\md u}(\rho(u)^{-1}))^{-1}$ is decreasing, so $h_{c,2}(u)\les 1$ and 
\[
\frac{h_{c,2}(u)}{C_\alpha}
\ges\me^{-\alpha}\frac{\rho(1)^2}{\rho'(1)}
=\me^{-\alpha}\frac{\sin(\alpha\pi)^\alpha\sin((1-\alpha)\pi)^{1-\alpha}}{\pi^2}
\ges\me^{-\alpha}\frac{1-\alpha}{\pi}.
\]
This shows that $h_{c,2}(u)/(1-\alpha)^2$ is uniformly bounded away from $0$ for $\alpha\in(1/2,1)$ and $u\in[1/2,1]$.
\end{enumerate}
\end{enumerate}
All in all, the expected running time is bounded above by 
$$\kappa_{\hyperref[alg:psi_2]{(2)}} \left[1/(1-\alpha)^{2}+1/\log^+(s^{\alpha/(1-\alpha)}-\alpha^{(1-2\alpha)/(1-\alpha)}(1-\alpha))+\alpha\log^+(s)+|\log(\alpha)|+\log N\right],$$ 
where the constant $\kappa_{\hyperref[alg:psi_2]{(2)}}$ depends on neither $s\in(0,\infty)$ nor $\alpha\in(0,1)$. 
\end{proof}

\subsubsection{Proof of Proposition~\ref{prop:undershoot_algorithm}.}
\label{subsubsec:proof_conditional_law_complexity}

In this subsection we prove Proposition~\ref{prop:undershoot_algorithm} using Propositions~\ref{prop:psi_1}, \ref{prop:psi_2} and~\ref{prop:undershoot}. Let $t$ and $w$ be as in the statement of Proposition~\ref{prop:undershoot_algorithm} and recall that $\theta>0$ is the ``intensity'' in the L\'evy measure~\eqref{eq:stable_Levy_measure} of the stable subordinator $S$. Recall also that $s=(\theta t)^{-1/\alpha}w$. As explained in the beginning of  Section~\ref{subsec:Undershoot_simulation}, if we sample $(\zeta_s,Y_s)$ with density $\psi_s$ given in~\eqref{eq:scaled_law}, then the random variable $(\theta t)^{1/\alpha}\zeta_s^{-1/r}$ follows the law of $S_{t-}|\{\tau_b=t,\,b(\tau_b)=w,\,\Delta_S(\tau_b)>0\}$. By the inequalities in~\eqref{eq:inequality_psi_tilde} of Proposition~\ref{prop:undershoot}, \hyperref[alg:undershoot_stable]{SU-Alg} samples from $\wt \psi_s$ (defined in Proposition~\ref{prop:undershoot}) and performs an accept-reject step to obtain a sample $(\zeta_s,Y_s)$ from the density $\psi_s$. This implies that \hyperref[alg:undershoot_stable]{SU-Alg} indeed samples from the law $\rmS\Unif_{\alpha}(t,w)$.

It remains to prove that the expected running time of \hyperref[alg:undershoot_stable]{SU-Alg} is bounded as stated. As explained in Remark~\ref{rem:prob_p'} above, the computation of $p'$ is assumed to take a constant amount of time. By~\eqref{eq:inequality_psi_tilde}, the acceptance probability in line~\ref{step:alg_undershoot_accept} of \hyperref[alg:undershoot_stable]{SU-Alg} equals $c_\psi\psi_s(\zeta_s,Y_s)/\wt\psi_s(\zeta_s,Y_s)$ and is bounded below by $(1-\alpha)/2$, uniformly in $s$ (and all other parameters except $\alpha$). Thus the accept-reject step in \hyperref[alg:undershoot_stable]{SU-Alg} on average takes no more than $2/(1-\alpha)$ steps.
Since $\wt\psi_s$ is mixture of densities $\wt\psi_s^{(1)}$ and $\wt\psi_s^{(2)}$ (see Proposition~\ref{prop:undershoot_algorithm}), we sample a Bernoulli variable taking value $1$ with 
probability $p$. If we get $1$, then sample from $\wt\psi_s^{(1)}$ 
and otherwise from $\wt\psi_s^{(2)}$. Thus, the computational complexities  
of Algorithms~\hyperref[alg:psi_1]{$\psi^{(1)}$-Alg} and~\hyperref[alg:psi_2]{$\psi^{(2)}$-Alg}, given in  Propositions~\ref {prop:psi_1} and~\ref{prop:psi_2}), imply the bound in Proposition~\ref{prop:undershoot_algorithm}.\qed

\subsection{Proofs of Theorem~\ref{thm:expectation_of_time_complexity}}
\label{subsubsec:proof_of_main_thm_stable}

We begin by recalling a well-known property of the moments (or Mellin transform) of stable random variables, see~\cite[Example 25.10]{ken1999levy}. For $\eta\in(-\infty,\alpha)$ and $t>0$, we have 
\begin{equation}
\label{eq:mellin_tranform}
\E[S_t^\eta]=(t\theta)^{\eta/\alpha}\Gamma(1-\eta/\alpha)/\Gamma(1-\eta)<\infty.
\end{equation}

Since $s=(\theta t)^{-1/\alpha}w$
and $\tau_b=B^{-1}(S_1)$ by 
Proposition~\ref{prop:joint_law}, where $S_1$ is the value of the stable subordinator at time $1$ and $B^{-1}$ is the inverse of the decreasing function $t\mapsto t^{-1/\alpha}b(t)$,
we have 
\[
(\theta \tau_b)^{-1/\alpha}b(\tau_b)= \theta^{-1/\alpha}S_1\overset{d}=S_{1/\theta}
\]
and hence $s\overset{d}=S_{1/\theta}$. Note that the density of $S_{1/\theta}$ is $\varphi_\alpha$, which does not depend on the value of $\theta$ (see the representation of density of $S_t$ in~\eqref{eq:zolotarev_representation}).

\begin{lem}
\label{lem:varphi_norm}
The following limit holds:
$(1-\alpha)\sup_{x\in\R_+}\varphi_\alpha(x)\to 1/\pi^2$ as $\alpha\to 1$. 
\end{lem}

\begin{proof}
Note from~\eqref{eq:stable_Levy_measure} and the L\'evy--Khintchine formula that the Fourier transform of $\varphi_\alpha$ is, for any $u\in\R$, given by
\[
\E[\me^{\mi uS_{1/\theta}}]
=\exp(-(-\mi u)^\alpha)
=\exp\big(-|u|^\alpha\me^{-\mi\sign(u)\pi\alpha/2}\big),
\]
where $\sign(u)\coloneqq\1_{(0,\infty)}(u)-\1_{(-\infty,0)}(u)$ is the sign of $u$. 
Recall that $\varphi_\alpha(x)=\frac{1}{2\pi}\int_{\R}\me^{-\mi ux}\E[\me^{\mi uS_{1/\theta}}] \md u$
for any $x\in\R$.
Thus, 
\begin{align*}
\sup_{x\in\R_+}\varphi_\alpha(x)
&\les\frac{1}{2\pi}\int_{\R}\big|\exp\big(-|u|^\alpha\me^{-\mi\sign(u)\pi\alpha/2}\big)\big|\md u\\
&=\frac{1}{2\pi}\int_{\R}\me^{-|u|^\alpha\cos(\pi\alpha/2)}\md u
=\frac{1}{2\pi}\cos(\pi\alpha/2)^{-1/\alpha}\int_{\R}\me^{-|u|^\alpha}\md u,
\end{align*}
where $\int_{\R}\me^{-|u|^\alpha}\md u\to \int_{\R}\me^{-|u|}\md u=2$ as $\alpha\to 1$ by dominated convergence and $\cos(\pi\alpha/2)^{-1/\alpha}(1-\alpha)\to 2/\pi$ as $\alpha\to 1$.
\end{proof}

\begin{lem}
\label{lem:bound_of_expectation_complicated_thing}
With $s$ defined in line~\ref{line_1:alg:undershoot_stable} of \hyperref[alg:undershoot_stable]{SU-Alg}, we have:
\begin{enumerate}[leftmargin=2em]
\item[{\normalfont(a)}] $\E\left[\log^+\big(1/\big(s^{\alpha/(1-\alpha)}-\alpha^{(2\alpha-1)/(1-\alpha)}(1-\alpha)\big)\big)\right]=\Oh((1-\alpha)^{-1})$ as $\alpha\uparrow1$;
\item[{\normalfont(b)}] $\E\big[\log^+\big(1/\left(s^{\alpha/(1-\alpha)}-\alpha^{(2\alpha-1)/(1-\alpha)}(1-\alpha)\big)\big)\right]\to0$ as $\alpha\downarrow0$.
\end{enumerate}
\end{lem}

\begin{proof}[Proof of Lemma~\ref{lem:bound_of_expectation_complicated_thing}]
Denote $A:=\alpha^{1-1/r}(1-\alpha)^{1/r}=\alpha^{2-1/\alpha}(1-\alpha)^{1/\alpha-1}$.
\begin{enumerate}[leftmargin=2em]
    \item[(a)] When $\alpha>1/2$, note that 
$\|\varphi_\alpha\|_\infty\coloneqq\sup_{x\in\RP}\varphi_\alpha(x)=\Oh((1-\alpha)^{-1})$ as $\alpha\uparrow1$ by Lemma~\ref{lem:varphi_norm}. Thus,
\begin{align*}
\E\bigg[\log^+\bigg(\frac{1}{s^r-\alpha^{r-1}(1-\alpha)}\bigg)\bigg]
&=-\int_A^{(A^r+1)^{1/r}}\hspace{-5mm}\varphi_\alpha(x)\log(x^r-A^r)\md x\\
&\leq -\|\varphi_\alpha\|_\infty  \int_A^{(A^r+1)^{1/r}}\hspace{-5mm}\log(x^r-A^r)\md x.
\end{align*}
By substituting $x=(y+A^r)^{1/r}$ we get
\begin{multline*}
-\int_A^{(A^r+1)^{1/r}}\hspace{-5mm}\log(x^r-A^r)\md x
=-\frac{1}{r}\int_0^1(y+A^r)^{1/r-1}\log(y)\md y
\les -\frac{A^{r(1/r-1)}}{r}\int_0^1\log(y)\md y\\
=\alpha^{(2\alpha-1)(1-2\alpha)/(\alpha(1-\alpha))-1}(1-\alpha)^{(1-2\alpha)/\alpha+1}\to 1,
\quad\text{ as }\alpha\uparrow1.
\end{multline*}
Hence we have $\E[\log^+(1/(s^{\alpha/(1-\alpha)}-\alpha^{(2\alpha-1)/(1-\alpha)}(1-\alpha)))]=\Oh((1-\alpha)^{-1})$ as $\alpha\uparrow 1$.

\item[(b)]
Note that the function $y\mapsto y\me^{-y}$ attains its maximum when $y=1$, 
implying 
$y\me^{-y}\leq 1/\me $ for all $y\in\R_+$.
Thus
\[
\varphi_\alpha(x)=r
    \int_0^1 \sigma_\alpha(u)x^{-r-1}\me^{-\sigma_\alpha(u)x^{-r}}\md u
\les r\int_0^1 x^{-1}\me^{-1}\md u
=\frac{r}{\me x}.
\]
The expectation in part~(b) of the lemma can now be bounded as follows: 
\[
\E\bigg[\log^+\bigg(\frac{1}{s^r-\alpha^{r-1}(1-\alpha)}\bigg)\bigg]
=-\int_A^{(A^r+1)^{1/r}}\hspace{-5mm}\varphi_\alpha(x)\log(x^r-A^r)\md x
\les-\frac{r}{\me}\int_A^{(A^r+1)^{1/r}}\hspace{-5mm}\log(x^r-A^r)\frac{\md x}{x}.
\]
By substituting $x=(y+A^r)^{1/r}$ we get
\[
-\int_0^1(y+A^r)^{-1/r+1/r-1}\log(y)\md y
\les -A^{-r}\int_0^1\log(y)\md y)=(1-\alpha)^{-1}\alpha^{1-\alpha/(1-\alpha)}\to 0
\quad\text{as }\alpha\downarrow0.\qedhere
\]
\end{enumerate}
\end{proof}

We can now give a proof Theorem~\ref{thm:expectation_of_time_complexity} on the expected running time of \hyperref[alg:triple_stable_conditional_on_time]{SFP-Alg}.

\begin{proof}[Proof of Theorem~\ref{thm:expectation_of_time_complexity}]
Case $t_*=\infty$. \hyperref[alg:triple_stable_conditional_on_time]{SFP-Alg} samples $(\tau_b,S_{\tau_b-},S_{\tau_b})$ under $\p$ by Proposition~\ref{prop:joint_law}. By Proposition~\ref{prop:undershoot_algorithm}, to find an upper bound on the average running time, we only need to take expectation of the average running time of \hyperref[alg:undershoot_stable]{SU-Alg} in $s$. In other words, upper bound $\E[\alpha\log_2^+(s)]$, $\E[\alpha\log_2^+(s^{-1})]$ and $\E[\log^+(1/(s^{\alpha/(1-\alpha)}-\alpha^{(2\alpha-1)/(1-\alpha)}(1-\alpha)))]$.

By~\eqref{eq:mellin_tranform} (with $\eta$ equal to either $\alpha/2$ or $-\alpha/2$), the inequality $\alpha\log(x)\les 2x^{\alpha/2}$ for all $x>0$ and the fact that $\inf_{x\in(0,2)}\Gamma(x)>1/2$, we have
\[
\E[\alpha\log^+(s)]\les2\E\big[S_{1/\theta}^{\alpha/2}
    \1\{S_{1/\theta}>1\}\big]
\les2\E\big[S_{1/\theta}^{\alpha/2}\big]
\les4\Gamma(1/2),
\]
\[
\E\big[\alpha\log^+(s^{-1})\big]
\les2\E\big[S_{1/\theta}^{-\alpha/2}
    \1\{S_{1/\theta}<1\}\big]
\les2\E\big[S_{1/\theta}^{-\alpha/2}\big]
\les4\Gamma(3/2).
\]
Together with Lemma~\ref{lem:bound_of_expectation_complicated_thing}, the average running time of \hyperref[alg:triple_stable_conditional_on_time]{SFP-Alg} is bounded above by 
$$
\kappa_0\big((1-\alpha)^{-3}
    +|\log\alpha|+\log N\big)
$$
where the constant $\kappa_0$ does not depend on $\alpha$.

To show the running time of \hyperref[alg:triple_stable_conditional_on_time]{SFP-Alg} has exponential moments, 
we only need to verify that the running times of Algorithms~\hyperref[alg:psi_1]{$\psi^{(1)}$-Alg} and~\hyperref[alg:psi_2]{$\psi^{(2)}$-Alg}, with random parameter $s\eqd S_{1/\theta}$, have exponential moments. By Propositions~\ref{prop:psi_1} and~\ref{prop:psi_2}, we only need to show that $\log^+(s^{-1})$, $\log^+(s)$ and $\log^+(s^r-\alpha^{r-1}(1-\alpha))^{-1})$ have exponential moments. Note from~\eqref{eq:mellin_tranform} that $\E[\exp(\eta\log^+(s^{-1}))]=\E[s^{-\eta}\1\{s<1\}]<\infty$ for any $\eta\ges 0$, $\E[\exp(\eta\log^+(s))]=\E[s^\eta\1\{s>1\}]<\infty$ for any $\eta<\alpha$, and 
\begin{multline*}
\E\big[\exp\big(\eta\big(\log^+(1/(s^{\alpha/(1-\alpha)}-\alpha^{(2\alpha-1)/(1-\alpha)}(1-\alpha))\big)\big)\big]
=\E\left[\frac{1}{(s^r-A^r)^\eta}\1\{A<s<(A^r+1)^{1/r}\}\right]\\
=\int_A^{(A^r+1)^{1/r}}\frac{\varphi_\alpha(x)}{(x^r-A^r)^\eta}\md x
\les\int_A^{(A^r+1)^{1/r}}\frac{\|\varphi_\alpha\|_\infty}{(x^r-A^r)^\eta}\md x
=\|\varphi_\alpha\|_\infty\int_0^1\frac{(y+A^r)^{1/r-1}}{y^\eta}\md y <\infty.
\end{multline*}
for any $\eta\in(0,1)$, where we recall that $\|\varphi_\alpha\|_\infty=\sup_{x\in\R_+}\varphi_\alpha(x)$. Thus, these steps indeed have running times with exponential moments, completing the proof.

Case $t_*<\infty$. Again by Proposition~\ref{prop:joint_law}, \hyperref[alg:triple_stable_conditional_on_time]{SFP-Alg} samples $(\tau_b,S_{\tau_b-},S_{\tau_b})$ under $\p$ conditional on $\tau_b\les t_*$ since we keep sampling $\tau_b$ until $\tau_b\les t_*$. Indeed, note that $\tau_b\les t_*$ is equivalent to $V_1\eqd S_1$ being larger or equal to $B(t_*)$, where $B$ is specified in Proposition~\ref{prop:joint_law}, and has acceptance probability equal to $\p[\tau_b\les t_*]$. The other steps are the same as in the case $t_*=\infty$.

The running time of \hyperref[alg:triple_stable_conditional_on_time]{SFP-Alg} in the case $t_*<\infty$ consists of two parts: (I) the naive sampling of $V_1\eqd S_1$ until $V_1\ges B(t_*)$ in Line~\ref{line:tau<t_*} of \hyperref[alg:triple_stable_conditional_on_time]{SFP-Alg} and (II) the running time of \hyperref[alg:undershoot_stable]{SU-Alg} with random input $t=\tau_b$  in Line~\ref{line:undershoot_t<t_*} of \hyperref[alg:triple_stable_conditional_on_time]{SFP-Alg}, where $\tau_b$ is conditioned to be smaller or equal to $t_*$. Let $T_\alpha(t)$ be the expected runtime of \hyperref[alg:undershoot_stable]{SU-Alg} with parameter $t$. Then the running time of \hyperref[alg:triple_stable_conditional_on_time]{SFP-Alg} in the case $t_*<\infty$ is bounded by
\[
\frac{1}{\p[\tau_b\les t_*]}
    +\E[T_\alpha(\tau_b)|\tau_b\les t_*]
=\frac{1+\E[T_\alpha(\tau_b)\1_{\{\tau_b\les t_*\}}]}{\p[\tau_b\les t_*]}
\les\frac{1+\E[T_\alpha(\tau_b)]}{\p[S_1\ges B(t_*)]},
\]
where $\E[T_\alpha(\tau_b)]\les\kappa_{\hyperref[alg:triple_stable_conditional_on_time]{SFP}}((1-\alpha)^{-3}+|\log\alpha|+\log N)$ by Theorem~\ref{thm:expectation_of_time_complexity}. It remains to lower bound the probability $\p[S_1\ges B(t_*)]$. 

Pick $\lambda>0$ and note that
\[
\p[S_{t_*}< b(t_*)]
\les\E[\me^{\lambda b(t_*)-\lambda S_{t_*}}\1\{S_{t_*}<b(t_*)\}]
\les\E[\me^{\lambda b(t_*)-\lambda S_{t_*}}]
=\me^{\lambda b(t_*)-\lambda^\alpha \theta t_*}.
\]
Suppose $\lambda^\alpha\theta t_*>\lambda b(t_*)$, then the elementary inequality $1-\me^{-x}\ges x\me^{-x}$ for $x>0$ gives
\[
\p[S_{t_*}\ges b(t_*)]
=1-\p[S_{t_*}< b(t_*)]
\ges 1-\me^{\lambda b(t_*)
    -\lambda^\alpha\theta t_*}
\ges(\lambda^\alpha\theta t_*-\lambda b(t_*))
    \me^{\lambda b(t_*)-\lambda^\alpha\theta t_*}.
\]
Since $x\me^{-x}$ attains its maximum $\me^{-1}$ when $x=1$, it is optimal to let $\lambda$ be the `solution' to $\lambda^\alpha\theta t_*-\lambda b(t_*)=1$. This solution, however, need not exist if $\theta t_*/b(t_*)^\alpha$ is small. Note that $\lambda\mapsto\lambda^\alpha\theta t_*-\lambda b(t_*)$ is increasing and continuous on $[0,(\alpha\theta t_*/b(t_*))^{r+1}]$ and decreasing thereafter, with maximal value $\alpha^r(1-\alpha)(\theta t_*)^{r+1}b(t_*)^r$. Thus, we may take $\lambda$ to be the solution of 
\[
\lambda^\alpha\theta t_*-\lambda b(t_*)
=\min\{1,\alpha^r(1-\alpha)(\theta t_*)^{r+1}b(t_*)^r\}.
\]
This leads to the lower bound 
\[
\p[S_{t_*}\ges b(t_*)]
\ges \min\{1,\alpha^r(1-\alpha)(\theta t_*)^{r+1}b(t_*)^r\}\me^{-1},
\]
implying the claim and completing the proof.
\end{proof}

\subsection{Proofs of Theorem~\ref{thm:tempered_expect_time} and Corollary~\ref{cor:improved_tempered_expected_time}}

We begin by introducing a construction of the probability measures $\p_q$ given the base probability measure $\p$ via the Esscher transform. Fix any $\theta>0$ and $\alpha\in(0,1)$ and let $S$ be a driftless stable subordinator under the probability measure $\p$ with L\'evy measure $\nu$. Denote by $(\mF_t)_{t\ges 0}$ the right-continuous filtration generated by $S$. Given $q\ges 0$, the probability measure $\p_q$ can be constructed via its Radon--Nikodym derivative 
\[
\frac{\md\p_q}{\md\p}\bigg|_{\mF_t}
\coloneqq M_q(t)
\coloneqq\exp\big(-qS_t 
    + q^\alpha \theta t\big),
\quad t\ges 0.
\] 
Indeed, by~[Sato, Thm 33.2], $S$ is subordinator under $\p_q$ with L\'evy measure $\nu_q(\md x)=\me^{-qx}\nu(\md x)$, making it a tempered stable process as desired. 

Recall $\chi_b=(\tau_b,S_{\tau_b-},S_{\tau_b})$, fix some $t>0$ and recall that $\{\tau_b\les t\}\in\mF_t$. For any measurable $A\subset[0,t]\times[0,\infty)^2$ we have $\{\chi_b\in A\}\in\mF_t$ and hence 
\begin{equation*}
\p_q\big[\chi_b\in A\big]
=\E\big[\1_A(\chi_b)M_q(t)\big]
=\E\big[\1_A(\chi_b)\1_{(qS_t,\infty)}(E)\big]\me^{q^\alpha\theta t}
=\p\big[\chi_b\in A\,\big|\,E\ges qS_t\big],
\end{equation*}
where $E\sim\Exp(1)$ is independent of $S$ under $\p$. In particular, we have
\begin{equation}
\label{eq:esscher_tau}
\p_q\big[\chi_b\in A\,\big|\,\tau_b\les t\big]
=\p\big[\chi_b\in A\,\big|\,\tau_b\les t,E\ges qS_t\big].
\end{equation}

\begin{proof}[Proof of Theorem~\ref{thm:tempered_expect_time}]
For a fixed time $t_*$, \hyperref[alg:triple_temper_stable]{TSFP-Alg} first samples the event $\{\tau_b\les t_*\}$ under~$\p$. Indeed, this is done by sampling $s\sim\mathcal{L}(S_{t_*})$ under $\p_q$ and noting that the event $s\les b(t_*)$ is equivalent to the corresponding tempered stable process, of which $s$ is the final value, to not cross the boundary on the time horizon $[0,t_*]$. Thus, if $s\les b(t_*)$, by the Markov property of $S$, we may simply start a new iteration from the initial time-space point $(t_*, s)$. If instead $s<b(t_*)$, we need to sample $\chi_b|\{\tau_b\les t_*\}$ under $\p_q$. Moreover, in that case, by~\eqref{eq:esscher_tau}, we need only repeatedly sample $\chi_b=(\tau_b,S_{\tau_b-},S_{\tau_b})$ under $\p$ conditional on $\tau_b\les t_*$ and an independent $E\sim\Exp(1)$, until $E\ges q S_{t_*}$. For that, \hyperref[alg:triple_stable_conditional_on_time]{SFP-Alg} samples $(t,u,v)\sim\mathcal{L}(\chi_b)$ under $\p$ conditional on $\tau_b\les t_*$. By the strong Markov property, to check if $E\ges qS_{t_*}$, we need only sample independently (and conditionally given $\tau_b=t$) $w\sim \mathcal{L}(S_{t_*}-S_{t})=\mathcal{L}(S_{t_*-t})$ via \hyperref[alg:stable]{S-Alg} and set $S_{t_*}=v+w$. This justifies the validity of \hyperref[alg:triple_temper_stable]{TSFP-Alg}.

Now let us consider the computational complexity of \hyperref[alg:triple_temper_stable]{TSFP-Alg}. The expected number of iterations in Line~\ref{line:loop_until_tau<t_*} of \hyperref[alg:triple_temper_stable]{TSFP-Alg} is
\[
\E_q[\lfloor\tau_b/t_*\rfloor+1]\les \E_q[\tau_b]/t_*+1.
\]
In each iteration we use \hyperref[alg:tempered_stable]{TS-Alg} with expected complexity uniformly bounded by $c_0\coloneqq 4.2154$ according to~\cite[Thm~3.1]{TemperedStableDassios}. After that, we sample the triplet conditional on $\tau_b\les t_*$ via \hyperref[alg:triple_stable_conditional_on_time]{SFP-Alg}. By Theorem~\ref{thm:expectation_of_time_complexity}, the expected running time of \hyperref[alg:triple_stable_conditional_on_time]{SFP-Alg} is bounded by
\[
\kappa_{\hyperref[alg:triple_stable_conditional_on_time]{SFP}}((1-\alpha)^{-3}+|\log\alpha|+\log N)/\p[S_{t_*}\ges b(t_*)].
\]

The acceptance probability in Line~\ref{line:check_temper} of \hyperref[alg:triple_temper_stable]{TSFP-Alg} is 
\[
\p[E\ges qS_{t_*}|S_{t_*}\ges b(t_*)]=\E[\me^{-qS_{t_*}}\1\{S_{t_*}\ges b(t_*)\}]/\p[S_{t_*}\ges b(t_*)].
\]
All in all, the expected running time of \hyperref[alg:triple_temper_stable]{TSFP-Alg} is bounded by 
\begin{equation}
\label{eq:cost-of-TSFP}
c_0\bigg(1+\frac{\E[\tau_b]}{t_*}\bigg)
+\frac{\kappa_{\hyperref[alg:triple_stable_conditional_on_time]{SFP}}((1-\alpha)^{-3}+|\log\alpha|+\log N)}{\E[\me^{-qS_{t_*}}\1\{S_{t_*}\ges b(t_*)\}]}.
\end{equation}
It remains to upper bound $\E[\tau_b]$ and lower bound $\E[\me^{-qS_{t_*}}\1\{S_{t_*}\ges b(t_*)\}]$.

Recall that 
\[
\E_q[\me^{-uS_1}]
=\exp((q^\alpha-(u+q)^\alpha)\theta),
\quad u\ges 0.
\]
Note that, by concavity, $(x+y)^\alpha-x^\alpha=\int_x^{x+y}\alpha z^{\alpha-1}\md z\ges \alpha y(x+y)^{\alpha-1}$. Thus, by Lemma~\ref{lem:exp_moment} with $u=1/b(0)$, we have 
\[
\E_q[\tau_b]
\les \frac{\theta^{-1}\me^{ub(0)}}{(u+q)^\alpha-q^\alpha}
=\frac{\theta^{-1}\me}{(q+1/b(0))^\alpha-q^\alpha}
\les\frac{\me (q+1/b(0))^{1-\alpha}b(0)}{\alpha\theta}
=\frac{\me (1+1/(qb(0)))^{1-\alpha}qb(0)}{\alpha q^\alpha\theta}.
\]
On the other hand, note that
\begin{align*}
\E[\me^{-qS_{t_*}}\1\{S_{t_*}< b(t_*)\}]
&\les\E[\me^{-2qS_{t_*}+qb(t_*)}\1\{S_{t_*}<b(t_*)\}]\\
&\les\E[\me^{-2qS_{t_*}+qb(t_*)}]
=\me^{-2^\alpha q^\alpha \theta t_* + qb(t_*)}
\les\me^{-2^\alpha q^\alpha \theta t_* + qb(0)}.
\end{align*}
Hence, if $t_*$ satisfies $(2^\alpha-1)q^\alpha\theta t_*>qb(0)$, then the elementary inequality $1-\me^{-x}\ges x\me^{-x}$ for $x\ges 0$ gives
\begin{multline*}
\E[\me^{-qS_{t_*}}\1\{S_{t_*}\ges b(t_*)\}]
=\me^{-q^\alpha\theta t_*}-\E[\me^{-qS_{t_*}}
    \1\{S_{t_*}< b(t_*)\}]\\
\ges
\me^{-q^\alpha \theta t_*}(1-\me^{-(2^\alpha-1) q^\alpha \theta t_* + 2qb(0)})
\ges((2^\alpha-1)q^\alpha\theta t_*-2qb(0))\me^{-2^\alpha q^\alpha \theta t_*+ 2qb(0)}.
\end{multline*}

Denote $T_1\coloneqq qb(0)/(2^\alpha-1)$, $T_2\coloneqq qb(0)$ and set $t_*=(2qb(0)+1-2^{-\alpha})/((2^\alpha-1)q^\alpha\theta)=(2T_1+2^{-\alpha})/(q^\alpha\theta)$ and note that $2^\alpha-1\in[\alpha\log2,\alpha]$. Then the expected running time of \hyperref[alg:triple_temper_stable]{TSFP-Alg} is bounded by
\begin{align*}
&c_0\bigg(1+\me \frac{(1+1/(qb(0)))^{1-\alpha} qb(0)
        }{\alpha q^\alpha\theta t_*}\bigg)
+\frac{1+\kappa_{\hyperref[alg:triple_stable_conditional_on_time]{SFP}}((1-\alpha)^{-3}+|\log\alpha|+\log N)}{((2^\alpha-1)q^\alpha\theta t_*-2qb(0))\me^{-2^\alpha q^\alpha \theta t_*+ 2qb(0)}}\\
=\,&c_0\bigg(1+\me \frac{(1+1/T_2)^{1-\alpha}T_2
        }{\alpha (2T_1+2^{-\alpha})}\bigg)
+\frac{1+\kappa_{\hyperref[alg:triple_stable_conditional_on_time]{SFP}}((1-\alpha)^{-3}+|\log\alpha|+\log N)}{((2^\alpha-1)(2T_1+2^{-\alpha})-2T_2)\me^{-2^\alpha (2T_1+2^{-\alpha}) + 2T_2}}\\
=\,&c_0\bigg(1+\me \frac{(1+1/T_2)^{1-\alpha} T_2
        }{\alpha (2(2^\alpha-1)T_2+2^{-\alpha})}\bigg)
+\me^{1+2T_1}\frac{1+\kappa_{\hyperref[alg:triple_stable_conditional_on_time]{SFP}}((1-\alpha)^{-3}+|\log\alpha|+\log N)}{1-2^{-\alpha}}\\
\les\,&c_0\bigg(1+\frac{\me}{\log 2}\1_{\{T_2\ges 1\}}
        +\frac{4\me}{\alpha}\1_{\{T_2<1\}}\bigg)
+2^\alpha\me^{1+2T_1}\frac{1+\kappa_{\hyperref[alg:triple_stable_conditional_on_time]{SFP}}((1-\alpha)^{-3}+|\log\alpha|+\log N)}{\alpha\log2}\\
\les\,&\kappa_{\hyperref[alg:triple_temper_stable]{TSFP}}\me^{2T_1}((1-\alpha)^{-3}+|\log\alpha|+\log N)/\alpha,
\end{align*}
where $\kappa_{\hyperref[alg:triple_temper_stable]{TSFP}}$ is a constant that depends neither on $\alpha\in(0,1)$, $\theta\in(0,\infty)$ nor $q\in(0,\infty)$.
\end{proof}

\begin{proof}[Proof of Corollary~\ref{cor:improved_tempered_expected_time}]

\hyperref[alg:improved_triple_temper_stable]{TSFFP-Alg} applies \hyperref[alg:triple_temper_stable]{TSFP-Alg} at most $1+\lfloor b(0)/R\rfloor$ times. In each application, the expected running time is bounded by 
\[
\kappa_{\hyperref[alg:triple_temper_stable]{TSFP}}\me^{2qR/(2^\alpha-1)}((1-\alpha)^{-3}+|\log\alpha|+\log N)/\alpha.
\]
With our choice of $R=(2^\alpha-1)/(2q)$, the average running time of \hyperref[alg:improved_triple_temper_stable]{TSFFP-Alg} is bounded by
\[
\kappa_{\hyperref[alg:triple_temper_stable]{TSFP}}\me(1+2qb(0)/(2^\alpha-1))((1-\alpha)^{-3}+|\log\alpha|+\log N)/\alpha.
\]
Since $\alpha\log2\les 2^\alpha-1\les\alpha$, the claim follows.
\end{proof}

\section*{Acknowledgements}

JGC and AM are supported by the EPSRC grant  EP/V009478/1 and by The Alan Turing
Institute under the EPSRC grant EP/X03870X/1.
AM is also supported by the EPSRC grant EP/W006227/1.
FL is funded by The China Scholarship Council and The University of Warwick  PhD scholarship. 

\bibliographystyle{plain}

\bibliography{ref_first_passage}

\appendix

\section{Some results for general subordinators}
\subsection{Joint law of 
\texorpdfstring{$(\tau_b,S_{\tau_b-},\Delta_S (\tau_b))$}{triplet} 
for general subordinators}
\label{app:proof_of_Thm_Joint_law}

Theorem~\ref{thm:joint_law} below and its proof are taken from~\cite{chi2016exact}, see also~\cite{https://doi.org/10.48550/arxiv.2205.06865}. This appendix is included for completeness, because Theorem~\ref{thm:joint_law} underpins our key \hyperref[alg:undershoot_stable]{SU-Alg}.

Let $S=(S_t)_{t\in\RP}$ be a subordinator 
with L\'evy measure $\nu$ on $(0,\infty)$ with infinite mass
and a non-negative drift.
Assume that the density of $S_t$ exists for $t>0$ and is continuous as a map $(t,x)\mapsto g_t(x)$ on $(0,\infty)\times \RP$. 
Kallenberg's condition, given by the requirement 
 $\int_\R \E [\me^{i u S_t}]\md u<\infty$ for all $t>0$, is easily seen to imply the continuity of the density of $S_t$. As above, let $b:\RP\to\RP$ be a non-increasing absolutely continuous function with $b(0)>0$ and $T_b=\inf\{t>0:b(t)=0\}\in(0,\infty]$ (with convention $\inf\emptyset :=\infty$).
Note that, without loss of generality, we may
assume that the drift $\mu\geq0$ of $S$ is zero, since otherwise we may consider $\tilde b(t):=b(t)-t\mu$ without modifying the first-passage event in an essential way. 
For any $x\in\R$, let $\delta(\md u -x)$ denote the Dirac delta measure concentrated at $x$.

\begin{thm}
\label{thm:joint_law}
 For a subordinator $S$ satisfying assumptions in the previous paragraph and
$(t,u,v)\in(0,T_b)\times[0,\infty)\times(0,\infty)$ we have:
\begin{subequations}
\begin{align} 
    &\p[\tau_b\in\md t, S_{\tau_b-}\in \md u ,\Delta_S(\tau_b)\in\md v]=\1\{0<b(t)-u<v\} g_t(u)\md t \, \md u\, \nu(\md v);\label{subeq:joint_undershoot_law}\\
    &
    \begin{multlined}
    \p[\tau_b\in\md t,S_{\tau_b-}\in\md u,\Delta_S(\tau_b)=0]
    =-b'(t)g_t(b(t))\,\md t\,\delta(\md u-b(t)).\label{subeq:creep_b}
    \end{multlined}
\end{align}
\end{subequations}
\end{thm}

The proof of~\eqref{subeq:joint_undershoot_law} is based on the compensation formula for the Poisson point process of jumps of $S$.  The proof of~\eqref{subeq:creep_b} essentially rests on the identification of the probability of creeping by expressing the density of the crossing time $\tau_b$ in two ways. It is this step that requires the continuity of the density $g_t(x)$. 

\begin{proof}[Proof of Theorem~\ref{thm:joint_law}]
First let us consider~\eqref{subeq:joint_undershoot_law}.
Let $f:(0,T_b)\times[0,\infty)\times(0,\infty)\to\RP$ be a Borel function, such the $f(t,u,v)=0$ for $v= b(t)-u$.
By part~(a) of the theorem proved above, we have $\p(\tau_b<T_b)=1$.
Thus, since $b$ is non-increasing and $S$ has increasing paths, we have
\begin{equation}
\label{eq:jump}   
f(\tau_b,S_{\tau_b-},\Delta_S (\tau_b))=\sum_{t\in(0,T_b)} f(t,S_{t-},\Delta_S(t))\1\{0\les b(t)-S_{t-}<\Delta_{S}(t)\}.
\end{equation}
For each $t>0$, define 
a $\sigma_\alpha$-algebra $\mF_t:=\sigma_\alpha(S_s:s\in[0,t])$ generated by $S$ during the time interval $[0,t]$
the random function $H_t(\cdot)=f(t,S_{t-},\cdot)\1\{0\les b(t)-S_{t-}<\cdot\}$ on $ (0,\infty)$. $H=(H_t)$ is a predictable process with respect to the filtration $(\mF_t)_{t\in\RP}$.
By~\eqref{eq:jump} and the compensation formula~\cite[Ch.~0.5]{MR1406564} for the Poisson point process of jumps $\{\Delta_S(t):t\in\RP\}$, we get
\[
\begin{aligned}
\E[f(\tau_b,S_{\tau_b-},\Delta_S(\tau_b))]&=\int_0^{T_b}\md t\, \E\big[\int_{(0,\infty)} f(t,S_{t-},v)\1\{0\les b(t)-S_{t-}< v\}\nu(\md v)\big]\\
&\overset{(*)}=\int_0^{T_b} \md t\int_{[0,\infty)\times(0,\infty)} f(t,u,v)\1\{0\les b(t)-u< v\}g_t(u)\md u\, \nu(\md v)\\
&=\int_{(0,T_b)\times[0,\infty)\times(0,\infty)}f(t,u,v)\1\{0\les b(t)-u< v\} g_t(u) \md t\,\md u\,\nu(\md v),
\end{aligned}
\]
where $(*)$ follows from the Fubini theorem and  a.s. equality $S_{t-}=S_t$ for any $t>0$. Since $f$ is arbitrary, this
implies that the vector $(\tau_b,S_{\tau_b-},\Delta_S(\tau_b))$,
on the event $\{\Delta_S(\tau_b)>0\}\setminus \{S_{\tau_b-}<b(\tau_b),S_{\tau_b}=b(\tau_b)\}$, has the law 
$\1\{0\les b(t)-u< v\} g_t(u) \md t\,\md u\,\nu(\md v)$.
Since $\1\{0\les b(t)-u<v\} g_t(u)=\1\{0< b(t)-u<v\} g_t(u)$
the equality holds $\md t\,\md u\,\nu(\md v)$-a.e.,
since the set $\{(t,u,v)\in(0,T_b)\times[0,\infty)\times (0,\infty): b(t)=u \}$ is not charged by the measure $\md t\,\md u\,\nu(\md v)$.
The compensation formula applied to
 the function $f:(0,T_b)\times[0,\infty)\times(0,\infty)\to\RP$, given by
$f(t,u,v):=\1_{\{b(t)-u=v, 0<v\}}$,  as in the previous display
and the fact that $S_t$ has a density imply that $\p(S_{\tau_b-}<b(\tau_b),S_{\tau_b}=b(\tau_b))=0$.
Hence we have characterised the law of$(\tau_b,S_{\tau_b-},\Delta_S(\tau_b))$,
on the event $\{\Delta_S(\tau_b)>0\}$ and~\eqref{subeq:joint_undershoot_law} follows.

For any $c\in(0,\infty)$ and $t>0$, define $\varpi(c,t):=\int_0^c\ov\nu(c-u)g_t(u)\md u$, where 
$\ov\nu(x)=\nu((x,\infty))$. By~\eqref{subeq:joint_undershoot_law},
we have 
$\p[\tau_b\in\md t,S_{\tau_b}>b(\tau_b)]=\varpi(b(t),t)\md t$.
In particular, when $b$ is constant, i.e. $ b(t)=c_0>0$ for all $t\in(0,\infty)$, 
we have 
$\p[S_{\tau_{c_0}}>c_0]=1$ 
by~\cite[Prop.~3.iv]{MR1406564},
since $S$ is assumed to have no drift.
Thus,
$\p[\tau_{c_0}\in\md t]=\p[\tau_{c_0}\in\md t, S_{\tau_{c_0}}>c_0]=\varpi(c_0,t)\md t$ and $t\mapsto \varpi(c_0,t)$ is a density of $\tau_{c_0}$. 

Now consider~\eqref{subeq:creep_b}. 
Fix $t$, set $a:=b(t)$ and let $\tau_a:=\inf\{s>0:S_s>a\}$.
Pick $\epsilon>0$ and note
\begin{align*}
\{t-\epsilon<\tau_b\leq t\} & =\{S_{t-\epsilon}<b(t-\epsilon)\}\cap\{S_t\ges b(t)\} \\
& = \{S_{t-\epsilon}<a\leq S_t\}\cup\{b(t)\leq S_{t-\epsilon}<b(t-\epsilon)\} \\
& =\{t-\epsilon<\tau_a\leq t\}\cup\{b(t)\leq S_{t-\epsilon}<b(t-\epsilon)\}
\end{align*}
Letting $q(\epsilon):=\p[t-\epsilon<\tau_b\leq t]$, for $\epsilon>0$, we have $q(\epsilon)=\p[S_{t-\epsilon}<b(t-\epsilon),S_t>b(t)]=q_1(\epsilon)+q_2(\epsilon)$, where $q_1(\epsilon):=\p[S_{t-\epsilon}<a\leq S_t]=\p[t-\epsilon<\tau_a\leq t]$, 
$q_2(\epsilon):=\p[b(t)\leq S_{t-\epsilon}<b(t-\epsilon)]$. On one hand, $q_1(\epsilon)/\epsilon\to\varpi(a,t)=\varpi(b(t),t)$ as $\epsilon\downarrow0$. On the other hand, $q_2(\epsilon)=\int_0^{b(t-\epsilon)-b(t)}g_{t-\epsilon}(a+x)\md x$. Recall that the mapping $(t,x)\mapsto g_t(x)$ is continuous on $(0,\infty)\times(0,\infty)$, we have $q_2(\epsilon)/\epsilon\to-b'(t)g_t(b(t))$ as $\epsilon\downarrow0$. 
Thus 
\begin{align*}
        \frac{\p[\tau_b\in\md t]}{\md t}&=\lim_{\epsilon\downarrow0}\frac{\p[\tau_b\leq t]-\p[\tau_b\leq t-\epsilon]}{\epsilon}\\
        &=-b'(t)g_t(b(t))+\varpi(b(t),t)=-b'(t)g_t(b(t))+
    \frac{\p[\tau_b\in\md t, S_{\tau_b}>b(\tau_b)]}{\md t},
\end{align*}
implying 
$\p[\tau_b\in\md t, S_{\tau_b}=b(\tau_b)]=-b'(t)g_t(b(t))$
and hence~\eqref{subeq:creep_b}.
\end{proof}

\subsection{Moments of a crossing time of a level for infinite activity subordinators}

A first-passage time over a level of an infinite activity L\'evy process has 
exponential moments of all order. 

\begin{lem}\label{lem:exp_moment}
Given a subordinator $X$ with L\'evy measure $\nu$ on $(0,\infty)$ satisfying $\nu((0,\infty))=\infty$. For any $c>0$, define $\tau_c:=\inf\{t>0:X_t>c\}$ and $\psi(u):=\int_{(0,\infty)}(1-\me^{-u x})\nu(\md x)$ for $u>0$. Then, for any $p\in(0,\infty)$, there exists $u>0$ such that $\psi(u)>p$ and 
\[
    \E[\me^{p\tau_c}]\les1+\frac{\me^{uc}}{\psi(u)-p}.
\]
Moreover, $\E[\tau_c]\les\me^{uc}/\psi(u)$ for $u>0$.
\end{lem}

\begin{proof}[Proof of Lemma~\ref{lem:exp_moment}]
Given a differentiable increasing function $\varphi$ with $\varphi(0)=0$ and a positive random variable $\xi$ we have
\[
\E[\varphi(\xi)]
=\int_0^\infty \varphi(x)\p[\xi\in\md x]
=\int_0^\infty \int_0^x\varphi'(y)\md y\p[\xi\in\md x]
=\int_0^\infty \varphi'(y)\p[\xi> y]\md y.
\]
First let us set $\varphi(x)=\me^{p x}-1$.
Then, for $u>0$, we have
\[
\p[\tau_c>t]
\les\p[X_t\les c]
=\p[\me^{-uX_t}\ges \me^{-uc}]
\les \E[\me^{-uX_t}]\me^{u c}
=\exp(uc-t\psi(u)).
\]
Therefore, for any $u>0$ with $\psi(u)>p$ (recall $u$ is a free parameter), we have
\[
\int_0^\infty \me^{pt}\p[\tau_c>t]\md t
\les \int_0^\infty \me^{pt + u c-t\psi(u)}\md t
=\frac{\me^{uc}}{\psi(u)-p}.
\]
Now set $\varphi(x)=x$, we have 
\[
\E[\tau_c]=\int_0^\infty\p[\tau_c>y]\md y\les\int_0^\infty \me^{uc-t\psi(u)}\md t=\frac{\me^{uc}}{\psi(u)}.\qedhere
\]
\end{proof}

\section{Tempering}
\label{app:temper}

Fix any $\theta>0$ and $\alpha\in(0,1)$. Let $S$ be a driftless stable subordinator under the probability measure $\p$ with L\'evy measure $\nu(\md x)=w_\alpha^{-1}\theta x^{-\alpha-1}\md x$. Denote by $(\mG_t)_{t\ges 0}$ be the right-continuous filtration generated by $S$. Given $q\ges 0$, let $\p$ be the probability measure with Radon--Nikodym derivative 
\[
\frac{\md\p_q}{\md\p}\bigg|_{\mG_t}
=M_q(t)
\coloneqq\exp\big(-qS_t 
    + q^\alpha \theta t\big),
\quad t>0.
\] 
Then, by~[Sato, Thm 33.2], $S$ is a tempered stable subordinator under $\p_q$ with L\'evy measure $\nu_q(\md x)=w_\alpha^{-1}\theta\me^{-qx}x^{-\alpha-1}\md x$. Denote by $\E$ and $\E_q$ the expectation under the measures $\p$ and $\p_q$, respectively. We have following lemma:
\begin{lem}
\label{lem:tempered_law}
    For any $B\in\mathcal{G}_t=\sigma(S_s:s\in[0,t])$, let $E\sim\Exp(1)$ and $E\indep\mathcal G_t$. Then
    \begin{equation}
        \label{eq:temper_to_stable}
        \p_q[B]=\p[B|E\ges q S_t].
    \end{equation} 
    Moreover, for any $C\in\mathcal G_t$ satisfying $\p[C]>0$, we have
    \begin{equation}
        \label{eq:temper_to_stable_conditional}
        \p_q[B|C]=\p[B|C, E\ges q S_t].
    \end{equation}
\end{lem}

\begin{proof}
Note that 
\[
\p[E\ges q S_t|S_t]=\exp(-qS_t)\qquad\&\qquad \p[E\ges qS_t]=\E[\exp(-qS_t)]=\exp(-q^\alpha\theta t).
\]
Therefore
\begin{align*}
\p[B|E\ges q S_t]=&\frac{\E[\1_B\1_{\{E\ges q S_t\}}]}{\p[E\ges q S_t]}=\frac{\E[\E[\1_B\1_{\{E\ges q S_t\}}|\mathcal G_t]]}{\exp(-q^\alpha\theta t)}=\E_q[\1_B\E[1_{\{E\ges q S_t\}}|\mathcal G_t]\exp(q^\alpha\theta t)]\\
=&\E[1_B\exp(-q S_t+q^\alpha\theta t)]=\E[\1_BM_q(t)]=\p_q[B].
\end{align*}
Moreover,
\[
\p_q[B|C]=\frac{\p_q[B,C]}{\p_q[C]}=\frac{\p[B,C|E\ges q S_t]}{\p[C|E\ges q S_t]}=
\frac{\p[B,C,E\ges q S_t]/\p[E\ges qS_t]}{\p[C,E\ges q S_t]/\p[E\ges qS_t]}=\p[B|C, E\ges qS_t].\qedhere
\]
\end{proof}

\section{Useful Algorithms}
\label{app:algs}

\subsection{Simulation of stable and tempered stable random variables}
\label{app:temp_stable_marginal}
Let $S_t$ be a positive stable random variable with density 
$g_t$ given in~\eqref{eq:zolotarev_representation}. 
More precisely, 
 $S_t\eqd t^{1/\alpha}S_1$ for all $t\in(0,\infty)$,
 and, by~\cite[\S4.4]{uchaikin2011chance}, for any $x\in(0,\infty)$ we have
 $g_t(x)=\varphi_\alpha((\theta t)^{-1/\alpha}x)(\theta t)^{-1/\alpha}$
 where 
 $\varphi_\alpha(x)= r
    \int_0^1 \sigma_\alpha(u)x^{-r-1}\me^{-\sigma_\alpha(u)x^{-r}}\md u$ (with $r:=\alpha/(1-\alpha)\in(0,\infty)$),
    $\sigma_\alpha(u)=
\rho(u)^{r+1}$ and 
    $\rho(u)=
\sin(\alpha\pi u)^{\alpha}\sin((1-\alpha)\pi u)^{1-\alpha}/\sin(\pi u)$
for $u\in(0,1)$.
The following algorithm samples from the law of $S_t$.

\begin{algorithm*}
\caption{(Chambers--Mallows--Stuck, S-Alg) Simulation from the law of $S_t$
under $\p_0$}
\label{alg:stable}
\begin{algorithmic}[1]
\Require{Parameters $\alpha\in(0,1)$ and $\theta,t>0$}
\State{Generate $U\sim\Unif(0,1)$ and $E\sim\Exp(1)$}
\State{\Return $(\theta t)^{1/\alpha}(\sigma_\alpha(U)/E)^{(1-\alpha)/\alpha}$}
\end{algorithmic}
\end{algorithm*}

Note that the distribution function of 
$S_t$ is given by
$x\mapsto \int_0^1\exp(-\sigma_\alpha(u)(x/(\theta t)^{1/\alpha})^{-r}\md u$,
which clearly equals the distribution function of the output of \hyperref[alg:stable]{S-Alg}. Moreover, \hyperref[alg:stable]{S-Alg} clearly has constant computational cost, which does not depend on the parameters $\theta,t,\alpha$.

For the simulation of $S_t$ under $\p_q$ (i.e., a tempered stable marginal), we use the algorithm in~\cite[Alg.~3.1]{TemperedStableDassios}, which outperforms~\cite{TemperedStableDevroye} and is essentially a rejection-sampling algorithm with expected complexity uniformly bounded by $4.2154$ for all parameter combinations, see~\cite[Thm~3.1]{TemperedStableDassios}. Recall the definition of the function $\rho$ in~\eqref{eq:zolotarev_density}. Define the error function $\erf(x)\coloneqq (2/\sqrt\pi)\int_0^x\me^{-t^2}\md t$ and the functions $\Psi_i:(0,1)\times[0,\infty)\to\R$, $i\in\{0,1,2,3,4\}$, given by
\begin{gather*}
\Psi_0(a,x)
\coloneqq\frac{\erf\big(\sqrt{a(1-a)x\pi^2/2}\big)}{\sqrt{2\pi a(1-a)x}},\\
\Psi_1(a,x)
\coloneqq\frac{\Gamma(ax)\me^{ax-1}}{(ax)^{x}}
        \big(\tfrac{a}{1-a}+ax\big)^{1+(1-a)x},\quad
\Psi_2(a,x)
\coloneqq\frac{\Gamma(1+(1-a)x)\me^{(1-a)x}}{((1-a)x)^{(1-a)x}},\\
\Psi_3(a,x)
\coloneqq\frac{\Gamma(1+ax)\me^{ax-1}(1+1/[(1-a)x])^{1+(1-a)x}}{(ax)^{ax}\sqrt{2\pi a(1-a)x}},\quad
\Psi_4(a,x)
\coloneqq\frac{\Gamma(1+(1-a)x)\me^{(1-a)x}}{((1-a)x)^{(1-a)x}\sqrt{2\pi a(1-a)x}}.
\end{gather*}
Denote by $\Gam(a,\theta)$ the gamma distribution with density $x\mapsto \theta^a x^{a-1}\me^{-\theta x}/\Gamma(a)$ and by $\Normal(0,\theta^2)|_{[0,1]}$ the truncated normal distribution with center $0$ and scale $\theta$ but restricted to the interval $[0,1]$, i.e., with distribution function $x\mapsto\erf\big(x/\sqrt{2\theta^2}\big)/\erf\big(1/\sqrt{2\theta^2}\big)$. Note as well that a sample from $\Normal(0,\theta^2)|_{[0,1]}$ can be obtained in constant time as there exist algorithms that numerically evaluate the inverse of $\erf$ accurately and with logarithmic cost in the number of precision bits $N$.

\begin{algorithm}
\caption{({\cite[Alg.~3.1]{TemperedStableDassios}}, TS-Alg) Simulation from the law of $S_t$ under $\p_q$}
\label{alg:tempered_stable}
\begin{algorithmic}[1]
\Require{Parameters $\alpha\in(0,1)$ and $\theta,t>0$ and $q\ges 0$}
\State{Set the constants $\lambda\gets(\theta t)^{1/\alpha}q$, $\xi=\lambda^\alpha$, $r=\alpha/(1-\alpha)$ and $C_i=\Psi_i(\alpha,\xi)$ for $i\in\{0,1,2,3,4\}$}
\If{$C_1=\min\{C_1,C_2,C_3,C_4\}$}
    \Repeat
    \State{Sample $U,\,V\sim\Unif(0,1)$ and $X\sim\Gam(\alpha\xi,1)$ and set $s\gets X/\lambda$}
    \Until{$V\les r\me^{\xi}\Gamma(\alpha\xi)(\rho(U)\xi)^{r+1}X^{-r-\alpha\xi}\exp\big(-(\rho(U)\xi)^{r+1}X^{-r}\big)/C_1$}
\ElsIf{$C_2=\min\{C_1,C_2,C_3,C_4\}$}
    \Repeat
    \State{Sample $U,\,V\sim\Unif(0,1)$ and $X\sim\Gam(1+(1-\alpha)\xi,1)$ and set $s\gets \rho(U)^{1/\alpha}X^{-1/r}$}
    \Until{$V\les \me^{\xi}\Gamma(1+(1-\alpha)\xi)X^{-(1-\alpha)\xi}\exp(-\lambda s)/C_2$}
\ElsIf{$C_3=\min\{C_1,C_2,C_3,C_4\}$}
    \Repeat
    \State{Sample $U\sim\Normal(0,1/[\pi^2\alpha(1-\alpha)\xi])|_{[0,1]}$, $V\sim\Unif(0,1)$ and $X\sim\Gam(\alpha\xi,1)$ and set $s\gets X/\lambda$}
    \Until{$V\les C_0r\me^{\xi}\Gamma(\alpha\xi)(\rho(U)\xi)^{r+1}X^{-r-\alpha\xi}
        \exp\big(\alpha(1-\alpha)\xi U^2/2-(\rho(U)\xi)^{r+1}X^{-r}\big)/C_3$}
\Else
    \Repeat
    \State{Sample $U\sim\Normal(0,1/[\pi^2\alpha(1-\alpha)\xi])|_{[0,1]}$, $V\sim\Unif(0,1)$ and $X\sim\Gam(1+(1-\alpha)\xi,1)$ and set $s\gets\rho(U)^{1/\alpha}X^{-1/r}$}
    \Until{$V\les C_0\me^{\xi}\Gamma(1+(1-\alpha)\xi)X^{-(1-\alpha)\xi}
        \exp\big(\alpha(1-\alpha)\xi U^2/2-\lambda s\big)/C_4$}
\EndIf
\State{\Return $(\theta t)^{1/\alpha}s$}
\end{algorithmic}
\end{algorithm}

\subsection{Rejection sampling from log-concave density}
\label{sec:devroye}
Devroye's extremely efficient method~\cite{devroye2012note} for the simulation from a log-concave densities is crucial in several steps of our main \hyperref[alg:triple_stable_conditional_on_time]{SFP-Alg}. The general algorithm in~\cite{devroye2012note} samples from densities defined on $\R_+$. Here we phrase it for log-concave densities defined on $[0,1]$, as all the densities arising in 
\hyperref[alg:triple_stable_conditional_on_time]{SFP-Alg}
are defined on a compact interval contained in $[0,1]$. Recall that $f:[0,1]\to\R_+$ is log-concave by definition if $\log(f)$ is a concave function on $[0,1]$.

\begin{algorithm}
\caption{(\cite{devroye2012note}, LC-Alg) Simulation from a log-concave density on $[0,1]$}
\label{alg:Devroye}
\begin{algorithmic}[1]
\Require{Log-concave $f:[0,1]\to\R_+$ with mode at $0$ and $f(0)=1$, proportional to the target density}
\State{Find $a_1\in\{2^{-1},2^{-2},\ldots\}$ such that $f(a_1)\ges 1/4\ges f(2a_1)$} \label{line_find_a_1}
\label{step_binary}
\State{Define $h(x) = \1_{(0,a_1]}(x) + \1_{(a_1,2a_1]}(x)f(a_1) + \1_{(2a_1,1]}(x)\exp(a_1^{-1}[(2a_1-x)\log f(a_1) + (x-a_1)\log f(2a_1)])$}
\State{Set $a_0 \gets a_1 + a_1f(a_1) + a_1f(2a_1)/\log(f(a_1)/f(2a_1))$ \Comment{(note $a_0=\int_0^1h(x)\md x$)}}
\Repeat
    \State{Generate $V_1,V_2,V_3\sim\Unif(0,1)$}
    \If{$a_0V_2\les a_1$}
        \State{Set $X\gets a_1V_1$}
    \ElsIf{$a_0V_2\les a_1 + a_1f(a_1)$}
        \State{Set $X\gets a_1 + a_1V_1$}
    \Else
        \State{Set $X\gets 2a_1 + a_1\log(1/V_1) / \log(f(a_1)/f(2a_1))$}
    \EndIf
\Until{$X\les 1$ and $V_3\les f(X)/h(X)$}
\State{\Return $X$}
\end{algorithmic}
\end{algorithm}

Computational cost of \hyperref[alg:Devroye]{LC-Alg} consist of two parts: the computational cost of finding $a_1$,
which may be random if the density $f$ depends on parameters that are themselves random, and the random computational cost of the accept-reject step, which is in expectation bounded above by $5$~\cite[Section~4.4]{devroye2012note}.

\subsection{Numerical inversion}\label{sec:NewtonRaphson}
In \hyperref[alg:psi_2]{$\psi^{(2)}$-Alg}, we will apply Newton--Raphson method (\hyperref[alg:inversion_newton_raphson]{NR-Alg} below) to find the root $x_*$, satisfying $f(x_*)=0$, of an increasing function $f:(0,1)\to\R$. The search terminates when the computed approximation to $x_*$ coincides with the actual root $x_*$ in the first $N$ bits (e.g., in double-precision floating-point arithmetic $N=53$). 

In order to control the running time of \hyperref[alg:inversion_newton_raphson]{NR-Alg} when applied to
a twice differentiable increasing function $f:(0,1)\to\R$ with root $x_*$, we need access to an auxiliary function $M:[0,1]\times\mathbb N\to\R_+$ such that 
\begin{equation}
\label{eq:newton_raphson_auxiliary}
M(x_0,k)\ges \frac{1}{2}\frac{\sup_{x\in[x_*,x_0]}|f''(x)|}{\inf_{x\in[x_*,x_0]}|f'(x)|}\qquad\text{for all $(x_0,k)\in[0,1]\times\N$ satisfying $|x_*-x_0|\les 2^{1-k}$.} 
\end{equation}
In practice, for a specific function $f$, $M(x_0,k)$  may, but need not, depend on either $x_0$ or $k$ (see e.g. Propositions~\ref{prop:inverse_sigma}.~\ref{prop:inverse_u*sigma} and~\ref{prop:inverse_c1}).
The main purpose of $M$ is to guarantee that the starting point $x_0$ of the Newton--Raphson method is sufficiently close to the root $x_*$ so that the convergence is quadratic. Constructing $M$ for specific functions $f$ requires
\textit{a priori} estimates on $f$.

\begin{algorithm}
\label{alg:inversion_newton_raphson}
  \caption{(NR-Alg) Newton--Raphson root-finding method}
  \begin{algorithmic}[1]
\Require{Twice differentiable increasing function $f:(0,1)\to\R$, derivative 
$f':(0,1)\to(0,\infty)$, auxiliary function $M:[0,1]\times\mathbb N\to\R_+$ in~\eqref{eq:newton_raphson_auxiliary}}
\State{Set $n=0$, $x_0=k=1$  }
\Repeat \label{alg:inversion_newton_raphson_bisection_start}
    \If{$f(x_0-2^{-k})>0$}
    \State{Set $x_0\leftarrow x_0-2^{-k}$  }
    \EndIf
    \State{Set  $k\leftarrow k+1$}
\Until{$2^{1-k}M(x_0,k)<1/2$} \label{alg:inversion_newton_raphson_bisection_stop}
\Repeat
    \State {$x_{n+1}=x_n-f(x_n)/f'(x_n)$, $n\leftarrow n+1$}
\Until{$|x_{n}-x_{n-1}|\les 2^{-N}$}
\State{\Return $x_n$}
\end{algorithmic}
\end{algorithm}

\hyperref[alg:inversion_newton_raphson]{NR-Alg}
starts at the right-end point of $[0,1]$
and, using bisection (line~\ref{alg:inversion_newton_raphson_bisection_start} to line~\ref{alg:inversion_newton_raphson_bisection_stop},
finds a starting point $x_0$ for the Newton--Raphson method to have quadratic convergence. In fact, the error sequence $(\ve_n)_{n\in\N}$ (where $\ve_n:=|x_n-x_*|$ is the error at the $n$-th iteration) satisfies $\ve_{n+1}\les M\ve_n^2$, implying that $\ve_n\les M^{2^n-1}\ve_0^{2^n}$. Thus, for the Newton--Raphson methodology to attain quadratic convergence, it suffices to find $x_0$ such that $M\ve_0\les 1/2$, as this would imply that $\ve_n\les (1/2)^{2^n-1}$, making $N$ bits of precision attainable in $\log N$ iterations. The computational cost of \hyperref[alg:inversion_newton_raphson]{NR-Alg} has two parts: (I) the cost to finding the value of $x_0$ and (II) the cost to compute the Newton--Raphson sequence $x_{n+1}=x_n-f(x_n)/f'(x_n)$ until the error reaches desired precision. (I) depends on $2^{1-k}M(x_0,k)$ and is in fact bounded above by $N$. (II) is bounded by $\log N$ as we explained before.

Householder's method (\hyperref[alg:inversion_householder]{H-Alg} below) of order $k\in\N$ generalises Newton--Raphson (the latter is recovered when $k=1$). Moreover, its convergence rate is of order $k+1$: under certain regularity assumptions, the error $\ve_n\coloneqq|x_n-x|$ of the $n$-th iteration satisfies $\ve_{n+1}\les M_k\ve_n^{k+1}$ for $n\in\N\cup\{0\}$ and an appropriate $M_k>0$. This implies that $\ve_n\les M_k^{((k+1)^n-1)/k}\ve_0^{(k+1)^n}$ for $n\in\N$. The drawback of this method is that it requires an evaluation of $f,\,f',\,\ldots,f^{(k)}$ at each step. This drawback usually outweighs the improved convergence, except in the cases where evaluation of $f$ requires a numerical integration and the evaluation of its derivatives is much cheaper by comparison (this is the case in Subsection~\ref{subsec:direct_numerical_inversion} above). We stress that \hyperref[alg:inversion_householder]{H-Alg} is only used to compare the performance of our simulation algorithms \hyperref[alg:psi_1]{$\psi^{(1)}$-Alg} and \hyperref[alg:psi_2]{$\psi^{(2)}$-Alg} with a simulation using a direct numerical inversion (via \hyperref[alg:inversion_householder]{H-Alg}) of the corresponding distribution functions.

\begin{algorithm}
  \caption{(H-Alg) Householder method}\label{alg:inversion_householder}
  \begin{algorithmic}[1]
\Require{Index $k$, initial guess $x_0\in(0,1)$ and $(k+1)$-differentiable increasing function $f:(0,1)\to\R$ with a zero on $[0,1]$ and derivatives $g_m(x)=\tfrac{\md^{m}}{\md x^{m}}(1/f(x))$ for $m\in\{k-1,k\}$.}
\State{Set $n=0$,}
\Repeat
    \State {$x_{n+1}\gets x_n+k g_{k-1}(x_n)/g_{k}(x_n)$, $n\leftarrow n+1$}
\Until{$|x_{n}-x_{n-1}|\les 2^{-N}$}
\State{\Return $x_n$}
\end{algorithmic}
\end{algorithm}

\section{Practically exact simulation algorithms}
\label{app:exact}

The answer to the question in the title of Subsection~\ref{subsec:exact?}, whether an implemented simulation algorithm can be exact, depends on the definition of exact simulation. The simplest, and in some sense most natural, definition requires the output of the algorithm to have the same law as the variable it is simulating. In this case, the answer is clearly no: working with $N\in\N$ bits of precision makes the output lie on a grid, implying that the total variation (TV) distance between a simulated output and any target law with a density equals one.\footnote{We thank Luc Devroye for pointing this out to us.} A way of dealing with this issue might be to understand and minimize a natural distance between the output of an implemented algorithm and a draw from the target law. Four natural ways of measuring the distance between algorithm's output and the actual law are as follows.

\noindent \underline{(I) $L^\infty$-Wasserstein distance.}
Let $A$ (resp. $I$) denote the output of an implemented algorithm (resp. ``ideal''  sample form the target law).
For any coupling $\lambda$ of the laws of $A$ and $I$, let $\p_\lambda$ be a probability measure such that $(A,I)$ has law $\lambda$ under $\p_\lambda$. The $L^\infty(\p_\lambda)$-distance is given by the essential supremum $\p_\lambda\text{-}\esssup |A-I|\coloneqq\inf\{a>0:\p_\lambda(|A-I|>a)=0\}$ with convention $\inf\emptyset=\infty$ and the $L^\infty$-Wasserstein distance equals the infimum of $L^\infty(\p_\lambda)$-distances over all couplings $\lambda$. Numerical inversion algorithms ought to control precisely this distance. However, it appears that even a simple accept-reject method, implemented on a computer cannot reduce this distance, regardless of the precision used (i.e. for any $N\in\N$). Indeed, consider an implementation of \hyperref[alg:Devroye]{LC-Alg} (in  Appendix~\ref{sec:devroye} below), together with the natural coupling, obtained by running the algorithm on a computer and on an ideal computer (IC), which carries out all evaluations without error.  Even if the constants $a_0$, $a_1$ and $f(a_1)$ could be represented exactly on the grid, the event-detection occurring in lines $6$ and $8$ will have a chance of being misdetected by the computer, simply because the variable $V_2$ is not uniform on $(0,1)$ but instead lives on the grid. Therefore, there is a (small, but non-zero) chance that the value of $I$ originated from line $7$, while the value of $A$ came from line $9$. 
On this event of positive probability, we have $|A-I|=a_1>0$,
regardless of the precision we are using. The same  phenomenon occurs with positive probability 
with lines $9$ and $11$, making the 
$L^\infty$-distances for the natural coupling in fact infinite. Of course, the $L^\infty$-Wasserstein distance may be smaller, but there appears to be no easy way of finding a good bound on it other than through the natural coupling. 

\noindent \underline{(II) Total variation distance on a grid.} Another natural way of measuring the distance between the variables $A$ and $I$ would be to project $I$
onto the state space (i.e. grid in the computer) of $A$
and consider the total variation distance (TV-distance) between the two discrete variables. 
However, rounding errors can make 
the TV-distance be equal to one for any level $N\in\N$ of precision bits used on a computer: consider, for example, the deterministic random variable $0$ and assume the simulation algorithm first samples the deterministic variable $\pi$ and then computes $\sin(\pi)$.
This algorithm on an IC
returns zero, which is already on the grid (so the projection does nothing). However, 
regardless of the number of precision bits $N$, our computer may, due to rounding errors, return a non-zero value on the grid, see Figure~\ref{fig:python}. In particular, in this case $A$ and the projection of $I$ onto the grid are at TV-distance one for all $N\in\N$.  Note in passing that the $L^\infty$-distance in this case decreases with increasing precision $N$, while the TV-distance does not. 

\begin{figure}[ht]
\centering
    \includegraphics[width=12cm]{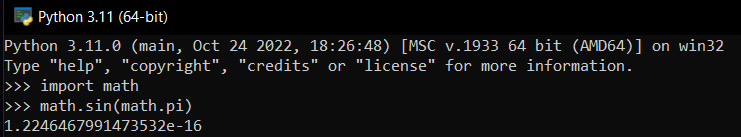}\caption{Rounding errors\label{fig:python} in Python produce singular random variables.}
\end{figure}

\noindent \underline{(III) $L^\infty$-Wasserstein distance of cdf mapping.} In~\cite[\S~7.2, p.~159]{MR2143197}, the authors propose controlling the difference $|F_I(A)-F_I(I)|$ a.s., where $F_I$ denotes the distribution function of $I$. If $F_I$ is continuous, it is then uniformly continuous. If $\omega_I$ denotes the modulus of continuity of $F_I$, then $|F_I(A)-F_I(I)|\les \omega_I(|A-I|)$, in which case, it suffices to control the $L^\infty$-Wasserstein distance as in (I). However 
this distance may also be bounded away from zero with tiny (but positive) probability, thus suffering from the same issue as the one described for $L^\infty$-Wasserstein distance.

\noindent \underline{(IV) Prokhorov distance.} The Prokhorov distance metrises weak convergence, making it a prime candidate for what a ``natural'' distance is. (Recall that, given two probability measures $\mu_1$ and $\mu_2$ on the metric space $(d,\mathbb{S})$, then $d_P(\mu_1,\mu_2)\coloneqq\inf\{\ve>0:\mu_1(F)\les\mu_2(F^\ve)+\ve\text{ for all closed }F\subset\mathbb{S}\}$, where $F^\ve\coloneqq\{x:d(x,F)<\ve\}$, is the Prokhorov distance between them.) By~\cite[Cor.~1]{MR230338}, 
\[
d_P(\mathcal{L}(I),\mathcal{L}(A))
=\inf\{\ve>0:\text{there exists a coupling $\lambda$ between $I$ and $A$ with }\lambda(|I-A|>\ve)\les\ve\}.
\]
In words, the Prokhorov distance is the smallest $\ve>0$ such that, under a suitable coupling and with probability at least $1-\ve$, the distance between the ideal variable $I$ and the output $A$ of the algorithm is at most $\ve$. In this sense, if an algorithm contains $k_1$ numerical inversions and $k_2$ reject sampling steps such that (i) each accept-reject step controls the error of accepting or rejecting a samples by $\ve_0$ and (ii) each numerical inversion step controls the distance between the output and the true inverse by $\ve_0$, then the Prokhorov distance between $I$ and $A$ is bounded by $\max\{k_1,k_2\}\ve_0$. 

The discussion in (I)--(IV) above illustrates how one can quantify a deviation from the strong definition of exact simulation. In this paper we use the following definitions: an algorithm implemented on a computer is \textit{exact} if, given exact values of the iid uniforms and assuming all the elementary operations carried out by the algorithm are exact, the returned value has the correct law (i.e. with zero TV-distance from the target law). In this definition, elementary operations include evaluations of elementary functions and their integrals as well as inversions. Moreover, we say that an algorithm is \emph{practically exact}, with respect to a chosen distance $d$ and a given tolerance $\ve$, if $d(\mathcal{L}(I),\mathcal{L}(A))\les\ve$. 

Our main aim in this paper is to implement a fast (practically) exact algorithm (see Appendix~\ref{app:TSFFP} below) and give theoretical guarantees for its running time, so that the algorithm can be used within a Monte Carlo estimation procedure. In exact simulation algorithms, function inversion is often implicitly assumed to produce exact values and take a bounded number of steps, uniformly in the input (see e.g.~\cite[Table~1, line~2]{chi2016exact}). 
In the case of our main simulation method (\hyperref[alg:improved_triple_temper_stable]{TSFFP-Alg} above), we instead carefully analyse the cost of each Newton--Raphson call needed to perform an inversions within the error margin, smaller than what can be represented on the computer. All other elementary operations, including the evaluation of the definite integrals in \hyperref[alg:undershoot_stable]{SU-Alg} (Step~\ref{line_1:alg:undershoot_stable})
and~\hyperref[alg:psi_2]{$\psi^{(2)}$-Alg} (Step~\ref{step:generate_discrete_rv}) are assumed to have constant complexity (the integrands are well behaved and close to those arising in the representation of a stable density, making the numerical procedures converge fast, see remarks following the algorithms). 

Finally, we note that the error between simulated variables on a grid and their idealised counterparts (e.g. for the increment of Brownian motion) are typically not taken into account when analysing the convergence properties of Monte Carlo estimates. This further motivates our definition of (practically) exact simulation for an implemented algorithm. 

\subsection{\texorpdfstring{\hyperref[alg:improved_triple_temper_stable]{TSFFP-Alg}}{TSFFP-Alg} is practically exact}
\label{app:TSFFP}

Using the definition above, \hyperref[alg:improved_triple_temper_stable]{TSFFP-Alg} is practically exact with respect to the Prokhorov distance and with tolerance equal to $C2^{-N}$, where $C$ is a constant that may be computed in terms of $(\theta,\alpha,q,b(0))$ and $N$ is the number of precision bits (typically, $N=53$). Indeed, \hyperref[alg:improved_triple_temper_stable]{TSFFP-Alg} consists of a loop of at most $1+\lfloor b(0)/R\rfloor$ repeated calls to \hyperref[alg:triple_temper_stable]{TSFP-Alg}. In each call to \hyperref[alg:triple_temper_stable]{TSFP-Alg}, we have (in lines~8-11) 3 accept-reject steps (accounting for all the calls to other algorithms) and at most 2 numerical inversions, each with an error bounded by $2^{-N}$, and (in lines~2--7) a loop where each step calls \hyperref[alg:tempered_stable]{TS-Alg}, containing one numerical inversion and one accept-reject step (again, each with an error bounded by $2^{-N}$). The number of iterations of this loop in \hyperref[alg:triple_temper_stable]{TSFP-Alg} is equal to $1+\lfloor\tau_b/t_*\rfloor$, where, by Markov's inequality and Lemma~\ref{lem:exp_moment}, we have 
\[
\p(1+\lfloor\tau_b/t_*\rfloor\ges n)
\les \me^{-pt_*(n-1)}
\bigg(1+\frac{\me^{uR}}{\theta((u+q)^\alpha-q^\alpha)-p}\bigg),
\]
for any $n\in\N$, $u>0$ and $p<\theta((u+q)^\alpha-q^\alpha)$. Pick $u=(3^{1/\alpha}-1)q$ and $p=\theta q^\alpha$, then we have 
\[
\p(1+\lfloor\tau_b/t_*\rfloor\ges n)
\les\me^{-(\theta q^\alpha-1) t_*(n-1)}
\big(1+\me^{(3^{1/\alpha}-1)qR}/
    (\theta q^\alpha)\big)
\les n_02^{-N},
\]
for all $n\ges n_0$ and some minimal $n_0\in\N$ that can be written explicitly in terms of all other parameters. Thus, the accumulation of errors in Porkhorov distance of these repeated calls to \hyperref[alg:tempered_stable]{TS-Alg} is bounded by $n_02^{-N}$. Hence, \hyperref[alg:improved_triple_temper_stable]{TSFFP-Alg} is practically accurate with tolerance $(n_0+3)(1+b(0)/R)2^{-N}$.

\end{document}